\documentclass[11pt]{article}
\usepackage{lmodern}
\usepackage{microtype}
\usepackage[margin=1.2in]{geometry}
\usepackage{amsmath,amssymb,amsthm,mathtools,mathrsfs}
\usepackage{enumitem}
\usepackage{booktabs}
\usepackage{array}
\usepackage{xcolor}
\usepackage{hyperref}
\usepackage{aliascnt}
\usepackage{tocloft}
\usepackage{hyperref}
\hypersetup{
    hidelinks
}

\allowdisplaybreaks
\newtheorem{theorem}{Theorem}[section]

\newaliascnt{corollary}{theorem}
\newtheorem{corollary}[corollary]{Corollary}
\aliascntresetthe{corollary}

\newaliascnt{proposition}{theorem}
\newtheorem{proposition}[proposition]{Proposition}
\aliascntresetthe{proposition}

\newaliascnt{lemma}{theorem}
\newtheorem{lemma}[lemma]{Lemma}
\aliascntresetthe{lemma}

\newaliascnt{remark}{theorem}
\newtheorem{remark}[remark]{Remark}
\aliascntresetthe{remark}

\newaliascnt{example}{theorem}
\newtheorem{example}[example]{Example}
\aliascntresetthe{example}

\theoremstyle{definition}
\newaliascnt{definition}{theorem}
\newtheorem{definition}[definition]{Definition}
\aliascntresetthe{definition}
\newtheorem*{definition*}{Definition}
\newtheorem*{remark*}{Remark}

\usepackage[nameinlink,capitalise,noabbrev]{cleveref}
\crefname{theorem}{Theorem}{Theorems}
\crefname{corollary}{Corollary}{Corollaries}
\crefname{proposition}{Proposition}{Propositions}
\crefname{lemma}{Lemma}{Lemmas}
\crefname{remark}{Remark}{Remarks}
\crefname{example}{Example}{Examples}
\crefname{definition}{Definition}{Definitions}

\newcommand{\Isom}{\operatorname{Isom}}
\newcommand{\Ax}{\operatorname{Ax}}

\newcommand{\diam}{\operatorname{diam}}
\newcommand{\pack}{\operatorname{pack}}

\newcommand{\EdgeSum}{\mathcal W}
\newcommand{\Cay}{\operatorname{Cay}}
\newcommand{\coef}{\mathrm{coef}}
\newcommand{\BS}{\mathrm{BS}}
\newcommand{\ax}{\mathrm{ax}}
\newcommand{\rel}{\mathrm{rel}}
\newcommand{\red}{\mathrm{red}}

\newcommand{\one}{\mathbf 1}

\newcommand{\Z}{\mathbb Z}

\newcommand{\e}{\mathrm e}

\title{Accessible CAT$(-1)$ groups of  critical exponent less than one}
\author{Yong Hou}
\date{}

\begin{document}
\maketitle

\begin{abstract}
Let $X$ be proper CAT$(-1)$ and let $\Gamma\le\Isom(X)$ be finitely generated and discrete.
The sharp structural theorem states that, if $\Gamma$ is accessible over finite subgroups and
$\delta_X(\Gamma)<1$, then $\Gamma$ is geometrically finite and virtually free, has a finite
graph of groups with finite edge groups and virtually cyclic infinite vertex groups, and
$\partial\Gamma\to\Lambda_\Gamma$ collapses exactly their conjugate two point boundaries.
The result is hereditary, and below $1/2$ every finitely generated subgroup is
convex-cobounded (\cref{thm:accessible-main}).  Hence non-virtually-free accessible groups
have $\delta_X(\Gamma)\ge1$ (\cref{cor:accessible-gap}).  Consequences cover finitely
presented groups, groups with uniformly bounded finite-subgroup orders, characteristic-zero
linear groups, and Kleinian groups, also infinite parabolic-free Kleinian groups have finite-index
classical Schottky subgroups
(\cref{cor:accessibility-extension,cor:linear-groups,cor:kleinian-classical}). Hence we cover substantial larger class than \cite{LiuWang2023},\cite{Hou2001}, also see \cref{rem:strictness-sharpness}. Finally, we also state consequences for finite JSJ representatives and hierarchies (\cref{thm:JSJ,thm:hierarchy,cor:hierarchy-dimension}).

\end{abstract}
\setcounter{tocdepth}{1}

\tableofcontents

\section{Introduction}

Let $X$ be a metric space, let $o\in X$, and let a countable group $\Gamma$ act by
isometries.  The orbit Poincar\'e series and critical exponent are
\[
 P_{\Gamma,o}(s)=\sum_{g\in\Gamma}\e^{-s d(o,go)},
 \qquad
 \delta_X(\Gamma)=\inf\{s>0:P_{\Gamma,o}(s)<\infty\}.
\]
For a proper CAT$(-1)$ space, Bourdon's formula
$\rho_o(\xi,\eta)=\e^{-(\xi\mid\eta)_o}$ defines a metric on $\partial X$
\cite{Bourdon1995}.  The generalized Bishop--Jones theorem identifies
$\delta_X(\Gamma)$ with the Hausdorff dimension of the conical limit set for every
non-elementary discrete action \cite[Theorem~1.2.1]{DasSimmonsUrbanski2017}, see also
\cite{Cavallucci2025}.  For a convex-cobounded action the conical limit set is the full limit
set, and the orbit map is a quasi-isometric embedding
\cite[Theorems~12.2.7 and~12.2.12]{DasSimmonsUrbanski2017}.
 We study structures of accessible groups when the critical exponent is less than 1. Dunwoody proves accessibility for finitely presented
 groups \cite{Dunwoody1985}, while Linnell proves it for finitely generated groups with a
uniform bound on the orders of finite subgroups \cite{Linnell1983}.  Hence our statement
contains the torsion-free case and simultaneously covers large torsion classes. 

Two constructions provide the main tools.  A boundary-energy argument on the full Cayley
graph, using selected ideal endpoints and Poincar\'e-series tails, controls limit-set pieces
after deletion of any finite vertex set and gives an ultrametric on $E(\Gamma)$ whose
canonical map to $\Lambda_\Gamma$ is surjective, $1$-Lipschitz, and equivariant
(\cref{thm:finite-vertex-deletion-main}).  It yields vanishing Hausdorff measure, the Cantor
and dimension conclusions, and infinitely many ends (\cref{cor:boundary-dimension-ends}).
A finite-state Bass--Serre construction encodes unique normal forms by weighted
non-backtracking matrices.  Monotone finite-transition approximation and
Perron--Frobenius theory give endpoint entrywise convergence and spectral radius at most one
on each strongly connected diagonal block containing a directed closed walk, together with
the exact normal-form abscissa (\cref{thm:bass-serre-main,cor:resolvent-dimension-bounds}).
Axis-based comparison matrices give matrix and cycle lower bounds
(\cref{cor:finite-matrix-cycle}). For finite JSJ representatives and hierarchies these hold
simultaneously at every node (\cref{thm:JSJ,thm:hierarchy,cor:hierarchy-dimension}).

The boundary construction
uses ideal-triangle construction that appears in the critical-exponent$<1$ argument,
but it works on the full Cayley graph after deletion of any finite vertex set.  It also gives an
ultrametric defined by tails of the Poincar\'e series and a Lipschitz quotient from the end
space. For the matrix part, Hou's displacement inequality for free families
\cite[Theorem~1.1]{Hou2001} and its curvature-free form due to Balacheff--Merlin
\cite[Theorem~1]{BalacheffMerlin2023} are established results.  Exact growth-series formulas for
additive free and amalgamated products occur in \cite{AllenEtAl2011}, while finite pressure
formulas for tree metrics occur in \cite{HersonskyHubbard1997}. The theorem below allows arbitrary
symmetric subadditive orbit lengths and infinite vertex and edge groups.  On each strongly
connected diagonal block containing a directed closed walk, it proves that every relative-coset
sum is finite at $s=\delta_X(\Gamma)$.  It also applies to arbitrary families of subgroup
splittings and determines the critical exponent of the normal-form length from the coefficient
series and matrices obtained from finite sets of transitions.  Hou's classification is used only after the CAT$(-1)$
structure theorem, to obtain the Kleinian classical-Schottky consequence \cite{Hou2023}.

\begin{definition*}
A finite graph-of-groups decomposition is \emph{reduced} if no collapsible edge remains in
its Bass--Serre tree, in the standard sense of Serre
\cite[Chapter~I, Section~5.4]{Serre2003}.  A finitely generated group $G$ is
\emph{accessible over finite subgroups} if there is an integer $N_G$ such that every reduced
finite graph-of-groups decomposition of $G$ with finite edge groups has at most $N_G$
unoriented edges.  This is the usual complexity-bound definition of accessibility
\cite[Introduction, pp.~1805--1806]{LouderTouikan2017}.
\end{definition*}

\begin{definition*}
By Stallings' theorem on ends
\cite[Chapter~5]{Stallings1971}, accessibility is equivalent to the existence of a
\emph{terminal Stallings--Dunwoody decomposition}: a finite graph of groups with finite edge
groups in which every vertex group is finite or one-ended.  Equivalently, every sequence of
reduced nontrivial refinements obtained by splitting a current vertex group over a finite
subgroup terminates after finitely many steps. Each such refinement increases the number of
edges, while the accessibility bound forbids an infinite increasing sequence.  See
Dunwoody's accessibility theorem \cite{Dunwoody1985} and the explicit discussion of the
Grushko--Stallings--Dunwoody decomposition and terminating refinements in
\cite[Introduction, pp.~1805--1807]{LouderTouikan2017}.  This is what we mean below by the
\emph{Stallings--Dunwoody notion} of accessibility.
\end{definition*}
A finitely generated two-ended group is virtually cyclic by the Freudenthal--Hopf--Stallings
theory \cite[Chapter~5]{Stallings1971}, hence virtually free.  By the
Karrass--Pietrowski--Solitar characterization
\cite[Theorem~1]{KarrassPietrowskiSolitar1973}, it is the fundamental group of a finite graph
of finite groups.  Thus any two-ended terminal vertex can be replaced by such a finite graph
of finite groups, which explains the normalization ``finite or one-ended'' above.

Let $S=S^{-1}$ be a finite generating set of $\Gamma$.  The undirected Cayley graph
$\Cay(\Gamma,S)$ has vertex set $\Gamma$ and an edge $\{g,gs\}$ for every
$g\in\Gamma$ and $s\in S$, and each edge has length one.  Its graph distance from $1$ is the
word length $|g|_S$, and
\[
 B_S(n)=\{g\in\Gamma:|g|_S\le n\}\qquad(n\ge0).
\]
We set $B_S(-1)=\varnothing$ solely as an indexing convention.  The notation
$\Cay(\Gamma,S)\setminus B_S(n)$ means the induced subgraph on
$\Gamma\setminus B_S(n)$.  
\begin{definition*}
A complementary component is \emph{infinite} precisely when its
vertex set is infinite.  These are the standard Cayley-graph conventions, see
\cite[Chapter~I.8]{BridsonHaefliger1999}.
\end{definition*}

A \emph{ray} is a one-way infinite simple edge path in $\Cay(\Gamma,S)$.  Two rays are
equivalent if, for every finite vertex set $Q$, tails of the two rays lie in the same connected
component of $\Cay(\Gamma,S)\setminus Q$.  The set of equivalence classes is the end space
$\mathcal E(\Gamma)$.  If $U$ is an infinite component of
$\Cay(\Gamma,S)\setminus Q$, let
\[
 \mathcal E(U)=\{\epsilon\in\mathcal E(\Gamma):
 \text{some, equivalently every, ray representing $\epsilon$ has a tail in $U$}\}.
\]
The sets $\mathcal E(U)$, as $Q$ ranges over finite vertex sets, form the basis of the
\emph{end topology}.  We write
\[
 \widehat\Gamma=\Gamma\sqcup\mathcal E(\Gamma)
\]
for the \emph{end compactification of the vertex set}: vertices are isolated, and a basic
neighborhood of an end represented by an infinite component $U$ outside a finite vertex set
is $U\sqcup\mathcal E(U)$.  Equivalently, $\widehat\Gamma$ is the subspace
$\Gamma\sqcup\mathcal E(\Gamma)$ of the Freudenthal compactification
$|\Cay(\Gamma,S)|$ of the geometric realization of the Cayley graph.  The full Freudenthal
compactification also contains the interiors of the graph edges, those points are not used in
this paper.  For locally finite graphs, these constructions are described in
\cite[Section~8.5, especially the definition of $|G|$ and Theorem~8.5.2]{Diestel2010}.
The equivalence between the ray definition and the topological formulation is discussed in
\cite[Section~2]{DiestelKuhn2003}.

For the boundary and transfer-matrix theorems, let $X$ be proper CAT$(-1)$, let
$\Gamma\le\Isom(X)$ be finitely generated, non-elementary, and discrete, and assume
$0<\delta_X(\Gamma)<\infty$.  Fix a finite symmetric generating set $S$, choose distinct
$A,B\in\Lambda_\Gamma$, and take $o$ on the geodesic $(A,B)$.  Let $D$ be a
thin-triangle constant and set $L=\max_{s\in S}d(o,so)$.

Choose a map $g\mapsto\xi_g$ from $\Gamma$ to $\partial X$ such
that
\[
 \xi_g\in\{gA,gB\},\qquad d\bigl(go,[o,\xi_g)\bigr)\le D
 \quad(g\in\Gamma).
\]
The existence of such a map is proved in \cref{lem:boundary-choice}.  For a finite set $Q\subset\Gamma$, let
$\mathscr U(Q)$ be the connected components of $\Cay(\Gamma,S)\setminus Q$.  For
$U\in\mathscr U(Q)$, the standard cluster set of the selected boundary points in $U$ is
\begin{equation}\label{eq:intro-general-cluster}
 K_{Q,U}=\bigcap_{F\sqsubset U}\overline{\{\xi_g:g\in U\setminus F\}}
 \subseteq\Lambda_\Gamma,
\end{equation}
where $F\sqsubset U$ means that $F$ is finite and the closure is taken in the compact visual
boundary.  Equivalently, $K_{Q,U}$ is the set of limits of sequences $\xi_{g_n}$ for which
$g_n\in U$ eventually leaves every finite subset of $U$.  Thus $K_{Q,U}=\varnothing$ for a
finite component and is nonempty for an infinite component.  In every diameter sum below we
use the convention $\diam(\varnothing)=0$.
\begin{definition*}
For a compact metric space $(Z,d)$ and $\varepsilon>0$, its packing number is
\[
 \pack_\varepsilon(Z,d)=
 \sup\bigl\{|E|:E\subseteq Z,\ d(x,y)\ge\varepsilon
 \text{ for all distinct }x,y\in E\bigr\}.
\]
If $e=\{g,gs\}$ is an undirected Cayley edge, set
$w_o(e)=\rho_o(\xi_g,\xi_{gs})$.  For $0<q\le1$ define the edge-weight sum
\begin{equation}\label{eq:intro-boundary-edge-sum}
 \EdgeSum_{q,o}(Q;\xi)=
 \sum_{e\subset \Cay(\Gamma,S)\setminus Q}w_o(e)^q,
\end{equation}
where every undirected edge is counted once.  We refer to this quantity simply as the
\emph{edge-weight sum outside $Q$}.  It depends on the fixed generating set and on the
chosen map $g\mapsto\xi_g$, as the notation $(Q;\xi)$ collects.  Let $c_\infty(Q)$ be the number
of infinite components of $\Cay(\Gamma,S)\setminus Q$, and set 
$\varepsilon_m(Z,d)=\sup\{\varepsilon>0:\pack_\varepsilon(Z,d)\ge m\}$.  Every inequality below holds for every choice of $(S,o,A,B)$ satisfying the assumptions above and every map $g\mapsto\xi_g$ satisfying the displayed conditions.
\end{definition*}

Let $(Z,d)$ be a metric space, let $E\subseteq Z$, and let $q\ge0$.  
The $q$-dimensional Hausdorff measure and Hausdorff dimension are
$
 \mathcal H^q(E,d),
 \dim_H(E,d)=\inf\{q:\mathcal H^q(E,d)=0\}.
$
When the metric is clear we may suppress it, but statements
about the CAT$(-1)$ boundary below use the Bourdon metric $\rho_o.$

When $P_{\Gamma,o}(q)<\infty$, define the tail function
\[
 \Theta_{q,o,S}(n)=|S|\e^{q(L+2D)}
 \sum_{|g|_S>n}\e^{-q d(o,go)}\qquad(n\ge-1).
\]
For distinct ends $\epsilon,\eta\in\mathcal E(\Gamma)$, let
$N_S(\epsilon,\eta)$ be the largest $n\ge-1$ for which $\epsilon$ and $\eta$ belong to the
same basic end neighborhood $\mathcal E(U)$ for some component
$U$ of $\Cay(\Gamma,S)\setminus B_S(n)$. Equivalently, representative rays have tails in
that same component.  We set $N_S(\epsilon,\epsilon)=\infty$.

For an isometric action of a group $G$ on $X$ and $o\in X$, its quasiconvex core is
\[
 C_o(G)=\bigcup_{g,h\in G}[go,ho].
\]
The action is \emph{convex-cobounded} if there exists $R<\infty$ such that
$C_o(G)\subseteq\bigcup_{g\in G}gB_X(o,R)$, equivalently if $G$ acts coboundedly on
$C_o(G)$; see
\cite[Proposition~7.5.3 and Definition~12.2.5]{DasSimmonsUrbanski2017}.

For a word-hyperbolic group $G$, $\partial G$ denotes the Gromov boundary of a finite Cayley
graph.  It is independent, up to a natural $G$-equivariant homeomorphism, of the chosen finite
generating set by quasi-isometry invariance
\cite[Chapter~III.H, Theorem~3.9]{BridsonHaefliger1999}.  In
\cref{thm:accessible-main}, virtual freeness first implies that $\Gamma$ is
word-hyperbolic.  Each infinite virtually cyclic peripheral subgroup $P_i$ is two-ended and
quasiconvex; $\partial P_i$ denotes its two-point Gromov boundary, identified with its embedded
limit set in $\partial\Gamma$.  See
\cite[Chapter~III.$\Gamma$, Lemma~3.5 and Proposition~3.7; Chapter~III.H,
Theorem~3.9]{BridsonHaefliger1999}.

Let $I$ be an arbitrary nonempty index set, which may be finite or infinite.  For every
$\beta\in I$, let $H_\beta\le\Gamma$ and let $\mathcal G_\beta$ be a finite connected graph
of groups with fundamental group $H_\beta$, underlying finite graph $Y_\beta$, and
Bass--Serre tree $T_\beta$.  No nesting, uniform bound, or compatibility among different
$\beta$ is assumed.  The Bass--Serre correspondence and normal form are
\cite[Chapter~I, Section~5.4, Theorem~13]{Serre2003}. The coordinate construction
used here is given in \cref{sec:bass-serre}.

Choose a base vertex $v_{\beta,*}\in Y_\beta$, compatible lifts of all vertices and oriented
edges to $T_\beta$, and transport elements $\tau_{\beta,f}\in H_\beta$.  Write
$G_{\beta,v}$ and $G_{\beta,f}\le G_{\beta,o(f)}$ for the corresponding vertex and edge
stabilizers.  

\begin{definition*}
For a left relative coset $c=rG_{\beta,f}\in
G_{\beta,o(f)}/G_{\beta,f}$, we define
\[
 L_{\beta,f,o}(c)=
 \min_{h\in G_{\beta,f}}d(o,rh\tau_{\beta,f}o).
\]
\end{definition*}
Choose $h_{\beta,f}(c)$ attaining the minimum and set
$b_{\beta,f}(c)=rh_{\beta,f}(c)\tau_{\beta,f}$ These are representatives that realize the displayed minimum.  A transition from an oriented edge
$e$ to $f$ is allowed when $t(e)=o(f)$, except that the identity coset is omitted when
$f=\bar e$, thereby forbidding immediate backtracking.  The weighted non-backtracking matrix
has entries
\[
 K_{\beta,ef,o}(s)=
 \sum_{c\text{ allowed on }e\to f}\e^{-sL_{\beta,f,o}(c)}.
\]

Let $E_{\beta,*}$ be the oriented edge states occurring in a reduced coset labelled loop based
at $v_{\beta,*}$.  Restriction to these states is the \emph{trim} matrix
$K_{\beta,*,o}(s)$: each retained state is reachable from the initial support and can reach the
terminal support.  The coefficient series and terminal-state indicator are
\[
 V_{\beta,o}(s)=\sum_{h\in G_{\beta,v_{\beta,*}}}\e^{-sd(o,ho)},
\]
\[
 (u_{\beta,o}(s))_f=
 \begin{cases}
 \displaystyle\sum_{c\in G_{\beta,v_{\beta,*}}/G_{\beta,f}}
 \e^{-sL_{\beta,f,o}(c)},&o(f)=v_{\beta,*},\\[1ex]
 0,&o(f)\ne v_{\beta,*},
 \end{cases}
 \qquad
 (q_\beta)_f=\one_{\{t(f)=v_{\beta,*}\}}.
\]
Every $h\in H_\beta$ has a unique reduced Bass--Serre form
$h=b_{\beta,f_1}(c_1)\cdots b_{\beta,f_n}(c_n)h_*$ with
$h_*\in G_{\beta,v_{\beta,*}}$.  The associated Bass--Serre normal-form length and its critical exponent are
\[
 L^{\BS}_{\beta,o}(h)=
 \sum_{j=1}^nL_{\beta,f_j,o}(c_j)+d(o,h_*o),
 \qquad
 \delta^{\BS}_{\beta,o}=
 \inf\left\{s>0:\sum_{h\in H_\beta}\e^{-sL^{\BS}_{\beta,o}(h)}<\infty\right\}.
\]
Thus the exponent is the abscissa of convergence of the normal-form Poincar\'e series.  We
write $\delta_{\beta,\coef}(o)$ for the abscissa of simultaneous convergence of the
coefficient series, namely the infimum of the $s>0$ for which $V_{\beta,o}(s)$ and every
entry of $u_{\beta,o}(s)$ and $K_{\beta,*,o}(s)$ are finite.
See \cref{def:relative-length,def:relative-matrix,lem:bass-serre-normal-form,def:coef-abscissa,thm:bass-serre-resolvent}.

The matrix $K_{\beta,*,o}(s)$ is a \emph{weighted non-backtracking matrix}, and its states are
oriented edges of $Y_\beta$, a transition is allowed only when the terminal vertex of the
first edge is the initial vertex of the second, immediate reversal is forbidden, and the
matrix entry is the sum of the exponential weights of all allowed relative-coset choices.
This is the weighted version of the Hashimoto or non-backtracking edge matrix; compare
\cite[Section~1]{Hashimoto1989}.  Weighted transition matrices and their path-sum
interpretation are standard in weighted-automaton theory. See
\cite[Sections~1.1--1.2]{DrosteKuichVogler2009} and
\cite[Section~2]{Mohri2009}.

A state is \emph{accessible} if it can be reached from the initial support and
\emph{co-accessible} if a terminal state can be reached from it.  An automaton or transition
matrix is \emph{trim} when every state is both accessible and co-accessible. This is the
standard automata-theoretic usage
\cite[Definition~1.3.9]{BertheRigo2016}.  Thus $K_{\beta,*,o}(s)$ is obtained from the full
oriented-edge matrix by deleting states that occur in no based reduced normal form.

For a finite directed graph, a \emph{strongly connected component} is a maximal set of
states in which every state is reachable from every other state by a directed walk.  A finite
nonnegative matrix is irreducible precisely when its support digraph is strongly connected
\cite[Theorem~3.2.1]{BrualdiRyser1991}.  Simultaneous permutation of rows and columns puts
any finite nonnegative matrix into \emph{Frobenius normal form}, whose diagonal blocks are
the strongly connected components \cite[Theorem~3.2.4]{BrualdiRyser1991}.

We write $\mathscr C_\beta$ for the strongly connected diagonal blocks whose support contains
a directed closed walk of positive length.  An acyclic diagonal block is nilpotent and does
not contribute exponential return growth.  For $C\in\mathscr C_\beta$, let
$K_{\beta,C,o}(s)$ denote the corresponding diagonal block, with every allowed transition
retained.  Thus each entry is the full, possibly countable, sum over the allowed relative
cosets.  We call this the full block only for brevity.  If $F$ is a finite set of individual
allowed transitions in $C$, let $K^F_{\beta,C,o}(s)$ be the matrix obtained by retaining
exactly the transitions in $F$.  We call it a finite truncation.  It need not be a principal
submatrix.  At $s=0$ each retained term equals $1$, so
$K^F_{\beta,C,o}(0)$ keeps transition multiplicities.

For a finite nonnegative matrix $M$, $\rho(M)$ denotes its spectral radius.  On an
irreducible block it is the Perron root given by the Perron--Frobenius theorem
\cite[Chapter~1]{Seneta2006}.  For a finite set $F$ of transitions, define
\[
 \sigma^{\rel}_{\beta,C,F}(o)=
 \inf\{s>0:\rho(K^F_{\beta,C,o}(s))<1\}.
\]
We call this number the critical parameter of $K^F_{\beta,C,o}$.  The superscript $``\rel"$
distinguishes it from the corresponding critical parameter of the comparison matrix defined
below.  This is the finite-dimensional analogue of the convergence-parameter language for
nonnegative matrices, compare \cite[Sections~2--3]{VereJones1967}.  \Cref{lem:positive-cycle-length} proves that this number is $0$ when
$\rho(K^F_{\beta,C,o}(0))\le1$ and otherwise is the unique positive solution of
$\rho(K^F_{\beta,C,o}(s))=1$. See \cref{prop:finite-transition-variation}.  In every formula
below, the maximum over an empty collection of blocks in $\mathscr C_\beta$ and the
supremum over an empty collection of finite transition sets are understood to be $0$.
\begin{remark*}
The terms weighted non-backtracking matrix, accessible state, co-accessible state, trim
matrix, strongly connected component, irreducible matrix, and Frobenius normal form are used
in the standard senses cited below.  The expressions ``full block'' and ``finite truncation''
are only abbreviations for the two explicitly defined matrices in the next paragraph. They do
not name new classes of matrices.
\end{remark*}

\begin{definition*}
An isometry $g$ of a proper CAT$(-1)$ space is \emph{hyperbolic} in the CAT(0) terminology
of Bridson--Haefliger, equivalently \emph{axial}, when its positive translation length is
attained; it then preserves a unique geodesic axis $\Ax(g)$
\cite[Chapter~II.6, Definition~6.3 and Theorem~6.8]{BridsonHaefliger1999}.  For such $g$ and
$x\in X$, set
\begin{equation}\label{eq:intro-axis-bound}
 a_x(g)=\tau_X(g)+2d(x,\Ax(g)),
 \qquad \tau_X(g)=\inf_{z\in X}d(z,gz).
\end{equation}
The translation length, axis, and distance to the axis are standard.  Throughout the paper,
$a_x(g)$ denotes exactly the displayed quantity.

Fix $\beta$, a block $C\in\mathscr C_\beta$, and an allowed transition $e\to f$ in $C$.
Choose a finite set $\mathcal F_{ef}$ of distinct allowed relative cosets.  For each
$c=rG_{\beta,f}\in\mathcal F_{ef}$, choose one representative
$g_{ef,c}=rh\tau_{\beta,f}$ in that coordinate class that is hyperbolic.  Such a choice is
made only for cosets containing a hyperbolic representative.  We refer to the collection
$\mathcal F=(\mathcal F_{ef})$ simply as a finite choice of hyperbolic representatives.

For this choice, let
\[
 K^{\mathcal F}_{\beta,C,o}(s)_{ef}=
 \sum_{c\in\mathcal F_{ef}}\e^{-sL_{\beta,f,o}(c)}
 \quad\text{and}\quad
 A^{\mathcal F}_{\beta,C,o}(s)_{ef}=
 \sum_{c\in\mathcal F_{ef}}\e^{-sa_o(g_{ef,c})}.
\]
These matrices have the same state set and satisfy
$0\le A^{\mathcal F}_{\beta,C,o}(s)\le K^{\mathcal F}_{\beta,C,o}(s)$ entrywise.  Define
\[
 \sigma^{\ax}_{\beta,C,\mathcal F}(o)=
 \inf\{s>0:\rho(A^{\mathcal F}_{\beta,C,o}(s))<1\},
 \qquad
 \sigma^{\rel}_{\beta,C,\mathcal F}(o)=
 \inf\{s>0:\rho(K^{\mathcal F}_{\beta,C,o}(s))<1\}.
\]
These are the critical parameters of the two finite matrices.
The entrywise comparison and the root description are proved in
\cref{lem:relative-axis-comparison,prop:axis-matrix-root}.  The superscript $``\ax"$ signifies the use of the explicitly defined quantity $a_o(g)$. It does
not introduce a separate standard notion of length.
\end{definition*}

\begin{remark*}
Statements above about convergence ``at the group critical exponent'' mean evaluation at the
explicit parameter $s=\delta_X(\Gamma)$.  For a block $C\in\mathscr C_\beta$, entrywise finiteness there means that every entry of
$K_{\beta,C,o}(\delta_X(\Gamma))$, with all allowed transitions retained, is a finite real
number.  The assertion above the
critical exponent is the inequality $\rho(K_{\beta,C,o}(s))<1$ for every
$s>\delta_X(\Gamma)$.

In additon, each assertion is pointwise in $\beta$ and uses the same
number $\delta_X(\Gamma)$. No finiteness, nesting, of the family is assumed.  A
graph-of-groups hierarchy is a rooted tree of groups whose children are vertex groups of a
chosen nontrivial graph-of-groups decomposition, as in
\cite[Definition~2.1]{LouderTouikan2017}.  The hierarchies we used are finite.  At every
node $H\le\Gamma$, reduced Bass--Serre paths remain distinct as elements of $\Gamma$, so the
same group critical exponent applies.  The relevant results are the bounded-multiplicity
transfer-matrix theorem \cref{thm:critical-exponent-matrix}, the Bass--Serre normal form
\cref{lem:bass-serre-normal-form}, and \cref{thm:hierarchy}.
\end{remark*}

\begin{theorem}
\label{thm:accessible-main}
Let $X$ be a proper CAT$(-1)$ space and let $\Gamma\le\Isom(X)$ be finitely generated,
discrete, and accessible over finite subgroups.  If $\delta_X(\Gamma)<1$, then $\Gamma$ is
virtually free, the action on $\Lambda_\Gamma$ is geometrically finite, and there are finitely
many conjugacy classes of infinite maximal parabolic subgroups $P_1,\dots,P_k$, all virtually
cyclic.  More precisely, $\Gamma$ is the fundamental group of a finite graph of groups with
finite edge groups and with every infinite vertex group conjugate to one of the $P_i$.  There is
a $\Gamma$-equivariant quotient
\begin{equation}\label{eq:main-accessible-quotient}
 q_\Gamma:\partial\Gamma\twoheadrightarrow\Lambda_\Gamma,
 \qquad |q_\Gamma^{-1}(\xi)|>1
 \qquad\mbox{if and only if}\qquad
 q_\Gamma^{-1}(\xi)=g\partial P_i
 \text{ for some }i, g\in\Gamma
\end{equation}
Every finitely generated subgroup $H\le\Gamma$ satisfies the same conclusions for its induced
action.  In the torsion-free case the graph-of-groups decomposition reduces to
$\Gamma\cong F_r*P_1*\cdots*P_k$ with every $P_i\cong\mathbb Z$.  Moreover, the following are equivalent: (a) $k=0$; (b) $q_\Gamma$ is a homeomorphism; (c) the
action has no parabolic point; (d) every infinite-order element is hyperbolic (equivalently, axial); (e) and the action is
convex-cobounded.  In particular,
\begin{equation}\label{eq:main-below-half-hereditary}
 \delta_X(\Gamma)<\tfrac12\quad\mbox{implies}\quad :
 \begin{gathered}
  \text{every finitely generated }H\le\Gamma\text{ acts convex-coboundedly,}\\
  H\longrightarrow X,\quad h\longmapsto ho,\quad
  \text{is a quasi-isometric embedding.}
 \end{gathered}
\end{equation}
\end{theorem}

\begin{theorem}
\label{thm:finite-vertex-deletion-main}
Let $X$ be a proper CAT$(-1)$ space and let $\Gamma\le\Isom(X)$ be finitely generated,
non-elementary, and discrete, with $0<\delta_X(\Gamma)<\infty$.  Fix a finite symmetric
generating set $S$, distinct points $A,B\in\Lambda_\Gamma$, a point $o\in(A,B)$ on the geodesic, the constants $D,L$, and a map $g\mapsto\xi_g$ satisfying the conditions defined earlier.  For every finite
$Q\subset\Gamma$ and every $0<q\le1$, the nonempty component cluster sets
$K_{Q,U}$ cover $\Lambda_\Gamma$ and
\begin{equation}\label{eq:main-finite-vertex-deletion-bound}
 \sum_{U\in\mathscr U(Q)}\diam(K_{Q,U})^q\le \EdgeSum_{q,o}(Q;\xi),
 \qquad
 \EdgeSum_{q,o}(Q;\xi)<\varepsilon^q
 \quad\mbox{implies}\quad
 c_\infty(Q)\ge\pack_\varepsilon(\Lambda_\Gamma,\rho_o).
\end{equation}
If $P_{\Gamma,o}(q)<\infty$, then
\begin{equation}\label{eq:main-tail-ultrametric}
 d^{\mathcal E}_{q,o,S}(\epsilon,\eta)=
 \begin{cases}
 0,&\epsilon=\eta,\\
 \Theta_{q,o,S}(N_S(\epsilon,\eta))^{1/q},&\epsilon\ne\eta,
 \end{cases}
 \qquad
 \rho_o(\Pi_o\epsilon,\Pi_o\eta)\le d^{\mathcal E}_{q,o,S}(\epsilon,\eta),
\end{equation}
where $d^{\mathcal E}_{q,o,S}$ is an ultrametric inducing the end topology and
$\Pi_o:(\mathcal E(\Gamma),d^{\mathcal E}_{q,o,S})\twoheadrightarrow
(\Lambda_\Gamma,\rho_o)$ is a surjective $1$-Lipschitz equivariant map and hence a topological quotient.
\end{theorem}
Next theorem is about weighted non-backtracking matrices for Bass--Serre normal forms at the group critical exponent.
\begin{theorem}
\label{thm:bass-serre-main}
For every member of the arbitrary family
$\{(H_\beta,\mathcal G_\beta)\}_{\beta\in I}$, the critical exponent of the Bass--Serre normal-form length is determined by the abscissae
of convergence of the coefficient series and by matrices obtained from finite sets of
transitions:
\begin{equation}\label{eq:main-finite-recovery}
 \delta^{\BS}_{\beta,o}=
 \max\left\{\delta_{\beta,\coef}(o),
 \max_{C\in\mathscr C_\beta}\sup_F
 \sigma^{\rel}_{\beta,C,F}(o)\right\}.
\end{equation}
For every $C\in\mathscr C_\beta$, every entry of
$K_{\beta,C,o}(\delta_X(\Gamma))$ is finite, and
$\rho(K_{\beta,C,o}(s))<1$ for every $s>\delta_X(\Gamma)$.  Every comparison matrix is
dominated entrywise by the corresponding matrix $K^{\mathcal F}_{\beta,C,o}(s)$.  The same assertions apply to each subgroup splitting in the
family and at every node of each chosen finite graph-of-groups hierarchy:
\begin{equation}\label{eq:main-transfer-comparison}
 \rho\!\left(K_{\beta,C,o}(\delta_X(\Gamma))\right)\le1,
 \qquad
 0\le A^{\mathcal F}_{\beta,C,o}(s)\le K^{\mathcal F}_{\beta,C,o}(s),
 \qquad
 \sigma^{\ax}_{\beta,C,\mathcal F}(o)
 \le\sigma^{\rel}_{\beta,C,\mathcal F}(o)
 \le\delta_X(\Gamma).
\end{equation}
\end{theorem}
First we have result for finitely presented and bounded-torsion extension.
\begin{corollary}
\label{cor:accessibility-extension}
Let $X$ be proper CAT$(-1)$ and let $\Gamma\le\Isom(X)$ be finitely generated and discrete
with $\delta_X(\Gamma)<1$.  The conclusions of
\cref{thm:accessible-main} hold whenever either
\begin{enumerate}[label=\textup{(\roman*)}]
\item $\Gamma$ is finitely presented, or
\item the orders of the finite subgroups of $\Gamma$ are uniformly bounded.
\end{enumerate}
\end{corollary}

\begin{remark*}
Every torsion-free group satisfies part~\textup{(ii)} of
\cref{cor:accessibility-extension}, because its only finite subgroup is trivial.  
The two hypotheses of the corollary cover genuinely larger classes with torsion. The explicit
scaled Bass--Serre tree example $C_2*C_3$ is given in
\cref{rem:strictness-sharpness}.
\end{remark*}
Next we have the critical exponent gap for accessible groups.
\begin{corollary}
\label{cor:accessible-gap}
Let $\Gamma\le\Isom(X)$ be finitely generated, discrete, and accessible over finite subgroups.
If $\Gamma$ is not virtually free, then $\delta_X(\Gamma)\ge1$.  The same conclusion holds,
in particular, for every finitely presented group and every finitely generated group with
uniformly bounded finite-subgroup orders.  If such a group has a non-virtually-cyclic maximal
parabolic subgroup, then again $\delta_X(\Gamma)\ge1$.
\end{corollary}
In terms of Hausdorff measure, dimension, and ends we have:
\begin{corollary}
\label{cor:boundary-dimension-ends}
Let $X$ be a proper CAT$(-1)$ space and let $\Gamma\le\Isom(X)$ be finitely generated,
non-elementary, and discrete, with $0<\delta_X(\Gamma)<\infty$.  Fix the boundary notation in the introduction.
For every $q\in(0,1]$ with $P_{\Gamma,o}(q)<\infty$, the orbit map extends
continuously and $\Gamma$-equivariantly to the end compactification of the vertex set, its end
restriction is the $1$-Lipschitz quotient in \eqref{eq:main-tail-ultrametric}, and
$\mathcal H^q(\Lambda_\Gamma,\rho_o)=0$.  In particular, if
$\delta_X(\Gamma)<1$, then we have
\begin{equation}\label{eq:full-limit-dimension}
 \dim_H(\Lambda_\Gamma,\rho_o)=\delta_X(\Gamma),\qquad
 \Lambda_\Gamma\text{ is Cantor},\qquad e(\Gamma)=\infty.
\end{equation}
For every $m\ge2$, the edge-weight sums also satisfy
\begin{equation}\label{eq:main-edge-sum-bound}
 \inf_{\substack{Q\sqsubset\Gamma\\c_\infty(Q)<m}}
 \EdgeSum_{q,o}(Q;\xi)\ge\varepsilon_m(\Lambda_\Gamma,\rho_o)^q.
\end{equation}
\end{corollary}

\begin{corollary}[Exact resolvent and Hausdorff-dimension lower bounds]
\label{cor:resolvent-dimension-bounds}
If the trim state set $E_{\beta,*}$ is empty, then the normal-form Poincar\'e series is
$V_{\beta,o}(s)$, as in \eqref{eq:bass-serre-edge-free}.  If
$E_{\beta,*}\ne\varnothing$, then whenever the trim coefficient series are finite and
$\rho(K_{\beta,*,o}(s))<1$, the normal-form Poincar\'e series is
\begin{equation}\label{eq:main-resolvent}
 \sum_{h\in H_\beta}\e^{-sL^{\BS}_{\beta,o}(h)}
 =V_{\beta,o}(s)
 \Bigl(1+u_{\beta,o}(s)^T(I-K_{\beta,*,o}(s))^{-1}q_\beta\Bigr).
\end{equation}
For $E_{\beta,*}\ne\varnothing$, these convergence conditions are necessary.  If $\delta_X(\Gamma)<1$ or the action is
convex-cobounded, then every finite comparison matrix at every basepoint and every hierarchy
level satisfies
\begin{equation}\label{eq:main-dimension-lower-bound}
 \sup_{y\in X}\sup_{\beta,C,\mathcal F}
 \sigma^{\ax}_{\beta,C,\mathcal F}(y)
 \le\delta_X(\Gamma)=\dim_H(\Lambda_\Gamma,\rho_o).
\end{equation}
\end{corollary}

\begin{corollary}
\label{cor:linear-groups}
Let $X$ be proper CAT$(-1)$ and let $\Gamma\le\Isom(X)$ be finitely generated and
discrete.  Suppose $\Gamma$ admits a faithful finite-dimensional linear representation over a
field of characteristic zero.  If $\delta_X(\Gamma)<1$, then every conclusion of
\cref{thm:accessible-main} holds.  Consequently, if such a group is not virtually
free, then $\delta_X(\Gamma)\ge1$. If $\delta_X(\Gamma)<1/2$, every finitely generated
subgroup acts convex-coboundedly.
\end{corollary}

\begin{corollary}[Kleinian groups and the parabolic-free case]
\label{cor:kleinian-classical}
Let $\Gamma\le\mathrm{PSL}_2(\mathbb C)$ be finitely generated and discrete with
$\delta_{\mathbb H^3}(\Gamma)<1$.  Then $\Gamma$ is virtually free and geometrically finite, and every
infinite maximal parabolic subgroup is virtually cyclic.  If $\Gamma$ is infinite and the action has
no parabolic point---in particular, if $\Gamma$ is infinite and
$\delta_{\mathbb H^3}(\Gamma)<1/2$---then $\Gamma$ contains a positive-rank classical Schottky subgroup
of finite index.  If $\Gamma$ is finite, the trivial subgroup has finite index.  If, in addition,
$\Gamma$ is nontrivial and torsion-free, then under the same no-parabolic hypothesis $\Gamma$ itself
is classical Schottky.
\end{corollary}

\begin{corollary}
\label{cor:free-family-axis-bound}
Let $a_1,\dots,a_k\in\Gamma$ be axial isometries freely generating a rank-$k$ subgroup,
$k\ge2$.  We set
\[
 \tau_i=\tau_X(a_i),\qquad r_i(o)=d(o,\Ax(a_i)),\qquad \delta=\delta_X(\Gamma).
\]
Then
\begin{equation}\label{eq:free-family-axis-bound}
 \sum_{i=1}^k
 \frac{2\e^{-2\delta r_i(o)}}
 {\e^{\delta\tau_i}-1+2\e^{-2\delta r_i(o)}}\le1.
\end{equation}
If $\tau_i\le T$ and one point lies within distance $R$ of every axis, then
\begin{equation}\label{eq:free-family-uniform-bound}
 \delta(T+2R)\ge\log(2k-1).
\end{equation}
The same conclusion holds with $R=\inf_y\max_i d(y,\Ax(a_i))$.
\end{corollary}
Finally we state our result for finite matrix and cycle inequalities which will be used for Grushko-JSJ hierarchy study as follows:
\begin{corollary}
\label{cor:finite-matrix-cycle}
Fix $C\in\mathscr C_\beta$ and a finite choice of hyperbolic representatives as in the
preceding definition.  Recall from \eqref{eq:intro-axis-bound} that
$a_o(g)=\tau_X(g)+2d(o,\Ax(g))$.  Suppose that, for a transition $e\to f$, there are
$N_{ef}$ selected distinct cosets and that every chosen representative $g_{ef,c}$ satisfies
$a_o(g_{ef,c})\le R_{ef}$.  Set $B_{ef}(s)=N_{ef}\e^{-sR_{ef}}$.  Then
$\rho(B(\delta_X(\Gamma)))\le1$.  If $R_{ef}\le R$ for every retained transition and
$N=(N_{ef})$, then
\begin{equation}\label{eq:matrix-dimension-bound}
 \delta_X(\Gamma)R\ge\log\rho(N)
\end{equation}
whenever $\rho(N)>1$.  Along a directed cycle whose successive retained transition
multiplicities are $N_j$ and whose corresponding upper bounds for $a_o(g)$ are $R_j$,
\begin{equation}\label{eq:cycle-axis-bound}
 \delta_X(\Gamma)\sum_jR_j\ge\sum_j\log N_j.
\end{equation}
If $\delta_X(\Gamma)<1$, or if the action is convex-cobounded, the right sides give lower
bounds for $\dim_H(\Lambda_\Gamma)$ at every level of every chosen finite Grushko--JSJ
hierarchy.
\end{corollary}

The main structural theorem establishes a sharp gap for accessible CAT$(-1)$ groups: critical
exponent less than one forces virtual freeness.  The argument identifies the convergence
action, not only the abstract group.  It gives geometric finiteness, virtually cyclic maximal
parabolic subgroups, the exact two-point peripheral fibers of the equivariant quotient map
$\partial\Gamma\to\Lambda_\Gamma$, the same classification for all finitely generated
subgroups, and convex-coboundedness for all such subgroups below one half.  Through Dunwoody
and Linnell, the theorem applies to finitely presented groups and to finitely generated groups
with uniformly bounded finite-subgroup orders; Selberg's lemma adds finitely generated linear
groups in characteristic zero.  These classes strictly extend the Liu--Wang torsion-free
statement \cite{LiuWang2023}.

The boundary theorem gives a quantitative estimate after deletion of any finite vertex set.
The edge-weight sum in the remaining graph dominates the total $q$-diameter of the cluster
sets carried by its components.  For convergent Poincar\'e series, the word-ball tails define
an ultrametric on the end space, and the map from the end space to the visual limit set is a
surjective $1$-Lipschitz equivariant map and hence a topological quotient.  This strengthens
the Hausdorff-measure conclusion by giving an explicit metric quotient and lower bounds that
relate finite vertex deletions to packing numbers.

The Bass--Serre matrix theorem supplies independent finite-state lower bounds.  Trimness
makes the resolvent exact, Bass--Serre uniqueness gives bounded multiplicity for evaluated
closed walks, and a fixed return path proves entrywise finiteness at
$s=\delta_X(\Gamma)$ on every strongly connected diagonal block containing a directed closed
walk.  Matrices obtained from finite transition sets, together with the coefficient series, determine
the abscissa of convergence of the normal-form series.  Distinct relative cosets represented by hyperbolic isometries
yield matrix, cycle, amalgam, HNN, Hashimoto, and hierarchy lower bounds for
$\delta_X(\Gamma)$ and, when $\delta_X(\Gamma)<1$ or the action is convex-cobounded, for the
Hausdorff dimension of the full limit set.

In the Kleinian setting, Selberg's lemma and the accessible theorem give virtual freeness with
virtually cyclic cusps for every finitely generated group of critical exponent less than one.
When parabolics are absent and the group is infinite, a torsion-free finite-index subgroup is
either rank-one classical Schottky or is classical Schottky by Hou's classification.  

The paper is organized as follows:
\Cref{sec:cat-geometry} collect some basic on the CAT$(-1)$ boundary and the displacement estimate involving $a_x(g)$.
\Cref{sec:word-ball} proves \cref{thm:finite-vertex-deletion-main,cor:boundary-dimension-ends}
and the ultrametric theorem for the end space.  \Cref{sec:small-exponent} proves
\cref{thm:accessible-main,cor:accessibility-extension,cor:accessible-gap,cor:linear-groups,cor:kleinian-classical}
using the sharp parabolic growth estimate and accessibility theory.
\Cref{sec:finite-state,sec:bass-serre,sec:main-proof} prove
\cref{thm:bass-serre-main,cor:resolvent-dimension-bounds}.  The remaining sections
develop free-product, amalgam, HNN, boundary-dimension, hierarchy, and bounded-cancellation
consequences without weaking the supporting arguments.

\section{\texorpdfstring{CAT$(-1)$}{CAT(-1)} orbit growth and hyperbolic isometries}
\label{sec:cat-geometry}

\begin{definition}
Let a countable group $\Gamma$ act isometrically on a metric space $X$, and fix $x\in X$.
The orbit length is
\[
 \ell_x(g)=d(x,gx).
\]
Its Poincar\'e series and critical exponent are
\[
 P_{\Gamma,x}(s)=\sum_{g\in\Gamma}\e^{-s\ell_x(g)},
 \qquad
 \delta_X(\Gamma)=\inf\{s>0:P_{\Gamma,x}(s)<\infty\}.
\]
\end{definition}

\begin{lemma}
\label{lem:orbit-basic}
For all $g,h\in\Gamma$ and $x,y\in X$, we have
\begin{align*}
 \ell_x(1)&=0, & \ell_x(g^{-1})&=\ell_x(g),
 & \ell_x(gh)&\le\ell_x(g)+\ell_x(h),\\
 |\ell_x(g)-\ell_y(g)|&\le2d(x,y).
\end{align*}
Consequently $\delta_X(\Gamma)$ is independent of $x$.  If
$0<\delta_X(\Gamma)<\infty$, then every orbit-length ball
$\{g:\ell_x(g)\le R\}$ is finite, and every nonempty subset of $\Gamma$ contains an element
of least orbit length.
\end{lemma}

\begin{proof}
Symmetry and subadditivity follow from invariance of the metric and the triangle inequality.
For the basepoint estimate,
\[
 d(x,gx)\le d(x,y)+d(y,gy)+d(gy,gx)=d(y,gy)+2d(x,y),
\]
and interchange $x,y$.  The two Poincar\'e series therefore differ by a fixed multiplicative
factor at every $s$, so their abscissae coincide.  If infinitely many elements had length at
most $R$, every Poincar\'e series would diverge.  The minimum statement follows by
intersecting a nonempty subset with a finite ball containing one of its elements.
\end{proof}

\begin{lemma}
\label{lem:subgroup-finite-index}
Let $H\le\Gamma$.
\begin{enumerate}[label=\textup{(\roman*)}]
\item $\delta_X(H)\le\delta_X(\Gamma)$.
\item If $H$ has finite index in $\Gamma$, then $\delta_X(H)=\delta_X(\Gamma)$.
\item If two nonnegative lengths $L_1,L_2$ on a countable set satisfy
$L_1\le L_2+C$, then the critical exponent of $L_1$ is at least that of $L_2$.
\end{enumerate}
\end{lemma}

\begin{proof}
The first and third assertions are termwise comparisons.  For finite index, choose a finite
right transversal $F$ for $H$ in $\Gamma$.  Every $g$ is uniquely $hf$, and
$\ell_x(h)\le\ell_x(hf)+\ell_x(f^{-1})$.  Thus
\[
 P_{\Gamma,x}(s)
 \le \left(\sum_{f\in F}\e^{s\ell_x(f)}\right)P_{H,x}(s),
\]
which, together with subgroup monotonicity, proves equality.
\end{proof}

Assume now that $X$ is proper CAT$(-1)$.  Its Gromov boundary is denoted $\partial X$.
For $\xi,\eta\in\partial X$, the boundary Gromov product is the limit of the ordinary
Gromov products along the unique geodesic rays from $o$.
Next result can be seen as an generalized Bishop--Jones theorem and Bourdon metric.
\begin{theorem}
\label{thm:bishop-jones}
For every $o\in X$,
\[
 \rho_o(\xi,\eta)=\e^{-(\xi\mid\eta)_o}
\]
is a metric on $\partial X$.  If $\Gamma\le\Isom(X)$ is non-elementary and discrete, then
\begin{equation}\label{eq:bishop-jones}
 \dim_H(\Lambda_{\Gamma,c},\rho_o)=\delta_X(\Gamma),
\end{equation}
where $\Lambda_{\Gamma,c}$ is the conical limit set.  If $R=d(o,o')$, then
\begin{equation}\label{eq:bourdon-basepoint}
 \e^{-R}\rho_o\le\rho_{o'}\le\e^R\rho_o.
\end{equation}
\end{theorem}

\begin{proof}
The metric assertion is Bourdon's CAT$(-1)$ theorem \cite{Bourdon1995}.  The dimension
identity is the generalized Bishop--Jones theorem
\cite[Theorem~1.2.1]{DasSimmonsUrbanski2017}. A proper-space formulation also follows from
\cite{Cavallucci2025}.  Finally,
$|(\xi\mid\eta)_o-(\xi\mid\eta)_{o'}|\le d(o,o')$, which gives
\eqref{eq:bourdon-basepoint}.
\end{proof}

\begin{theorem}
\label{thm:bourdon-coornaert}
If $\Gamma$ acts properly discontinuously and non-elementarily on $X$, and the action is
convex-cobounded,
then
\begin{equation}\label{eq:coornaert-dimension}
 \dim_H(\Lambda_\Gamma,\rho_o)=\delta_X(\Gamma).
\end{equation}
\end{theorem}

\begin{proof}
Convex-coboundedness gives
$\Lambda_\Gamma=\Lambda_{\Gamma,c}$ by
\cite[Theorem~12.2.7]{DasSimmonsUrbanski2017}.  The conclusion is therefore the generalized
Bishop--Jones identity \eqref{eq:bishop-jones}.
\end{proof}

\begin{definition}
\label{def:hyperbolic-axis-bound}
An isometry $g$ of $X$ is \emph{hyperbolic} in the CAT$(0)$ terminology of
Bridson--Haefliger if its translation length is positive and attained; equivalently, it preserves
a geodesic line and translates along it
\cite[Chapter~II.6, Definition~6.3 and Theorem~6.8]{BridsonHaefliger1999}.
Such an isometry is also called \emph{axial} (and, in boundary dynamics, loxodromic).
In a CAT$(-1)$ space the axis is unique. We write it $\Ax(g)$ and write
\[
 \tau_X(g)=\inf_{z\in X}d(z,gz)>0.
\]
For $x\in X$, use the shorthand
\begin{equation}\label{eq:axis-displacement-bound}
 a_x(g)=\tau_X(g)+2d(x,\Ax(g)).
\end{equation}
This notation agrees with \eqref{eq:intro-axis-bound}.  It denotes the displayed upper bound
and is not a standard name for a new geometric invariant.
\end{definition}

\begin{lemma}
\label{lem:axis-power}
If $g$ is axial, then for every $n\in\Z$,
\begin{equation}\label{eq:axis-power}
 d(x,g^n x)\le |n|\tau_X(g)+2d(x,\Ax(g)).
\end{equation}
In particular $d(x,gx)\le a_x(g)$.  Also
\begin{equation}\label{eq:axial-basepoint}
 |a_x(g)-a_y(g)|\le2d(x,y).
\end{equation}
\end{lemma}

\begin{proof}
Let $p$ be the nearest-point projection of $x$ to the closed convex geodesic $\Ax(g)$.  Since
$g^np$ lies on the same axis and $d(p,g^np)=|n|\tau_X(g)$,
\[
 d(x,g^nx)
 \le d(x,p)+d(p,g^np)+d(g^np,g^nx)
 =2d(x,\Ax(g))+|n|\tau_X(g).
\]
The basepoint estimate follows from
$|d(x,A)-d(y,A)|\le d(x,y)$.
\end{proof}

\begin{remark}
The word \emph{axial} is the standard CAT$(0)$ synonym for a hyperbolic isometry cited in
\cref{def:hyperbolic-axis-bound}.  The proof of \cref{lem:axis-power} uses only CAT$(0)$ projection
and axis properties.  Strictly negative curvature is used later for the Bourdon metric and the
limit-set dimension identity.  The notation $a_x(g)$ is only the abbreviation in
\eqref{eq:axis-displacement-bound}. Later statements refer directly to that number.
\end{remark}

\label{sec:finite-vertex-deletion}
\label{sec:word-ball}
Throughout this section, $X$ is a proper CAT$(-1)$ space and
$\Gamma\le\Isom(X)$ is finitely generated, non-elementary, and discrete, with
$0<\delta_X(\Gamma)<\infty$.  The set $S=S^{-1}$ is the fixed finite generating set,
$A,B\in\Lambda_\Gamma$ are distinct, $o\in(A,B)$, $D$ is the fixed ideal-triangle
thinness constant, $L=\max_{s\in S}d(o,so)$, and $g\mapsto\xi_g\in\{gA,gB\}$ is the
fixed map $g\mapsto\xi_g$ satisfying $d(go,[o,\xi_g))\le D$.

\begin{lemma}
\label{lem:boundary-choice}
For every $g\in\Gamma$ there is a point $p_g\in[o,\xi_g)$ such that
\begin{equation}\label{eq:boundary-choice-shadow}
 d(go,p_g)\le D.
\end{equation}
If $g_i o\to\lambda\in\Lambda_\Gamma$ in the visual compactification, then
$\xi_{g_i}\to\lambda$ in $(\partial X,\rho_o)$.
\end{lemma}

\begin{proof}
The point $go$ lies on $(gA,gB)$.  The ideal triangle with vertices $o,gA,gB$ is $D$-thin,
so $go$ lies within $D$ of one of the two rays from $o$.  For the second assertion,
\[
 (p_{g_i}\mid\lambda)_o\ge(g_i o\mid\lambda)_o-D,
 \qquad
 (\xi_{g_i}\mid\lambda)_o\ge(p_{g_i}\mid\lambda)_o.
\]
The right side tends to infinity, hence $\rho_o(\xi_{g_i},\lambda)\to0$.
\end{proof}

\begin{lemma}[Adjacent-orbit visual contraction]
\label{lem:adjacent-contraction}
For every $g\in\Gamma$ and $s\in S$,
\begin{equation}\label{eq:adjacent-contraction}
 \rho_o(\xi_g,\xi_{gs})\le \e^{L+2D}\e^{-d(o,go)}.
\end{equation}
\end{lemma}

\begin{proof}
Choose $p\in[o,\xi_g)$ and $q\in[o,\xi_{gs})$ within $D$ of $go$ and $gso$.
Then
\[
 d(o,p)\ge d(o,go)-D,
 \quad d(o,q)\ge d(o,go)-L-D,
 \quad d(p,q)\le L+2D.
\]
Thus $(\xi_g\mid\xi_{gs})_o\ge d(o,go)-L-2D$, which is equivalent to the claim.
\end{proof}

\subsection{Deletion of an arbitrary finite vertex set}
If $Q=\varnothing$, then $\Cay(\Gamma,S)\setminus Q$ is connected and has one component.
If $Q\subset\Gamma$ is finite and nonempty, every component of
$\Cay(\Gamma,S)\setminus Q$ is incident to an edge with its other endpoint in $Q$.  Since
the Cayley graph is locally finite, only finitely many such edges exist, and hence the
complement has only finitely many components.

\begin{theorem}[Boundary estimate after deleting a finite vertex set]
\label{thm:finite-vertex-deletion-boundary}
Let $0<q\le1$ and let $Q\sqsubset\Gamma$.  The nonempty cluster sets $K_{Q,U}$ form a finite
cover of $\Lambda_\Gamma$ and
\begin{equation}\label{eq:finite-vertex-deletion-estimate}
 \sum_{U\in\mathscr U(Q)}\diam(K_{Q,U})^q
 \le \sum_{e\subset\Cay(\Gamma,S)\setminus Q}w_o(e)^q
 =\EdgeSum_{q,o}(Q;\xi).
\end{equation}
Consequently,
\begin{equation}\label{eq:cut-packing}
 \EdgeSum_{q,o}(Q;\xi)<\varepsilon^q
 \qquad\mbox{implies}\qquad
 c_\infty(Q)\ge\pack_\varepsilon(\Lambda_\Gamma,\rho_o),
\end{equation}
and the lower bound obtained by deleting finite vertex sets is \eqref{eq:main-edge-sum-bound}.
\end{theorem}

\begin{proof}
Given $\lambda\in\Lambda_\Gamma$, choose distinct $g_i$ with $g_i o\to\lambda$.  After
discarding finitely many terms and passing to a subsequence, all $g_i$ lie in one component
$U$ of the complement of $Q$.  By \cref{lem:boundary-choice}, $\xi_{g_i}\to\lambda$, so
$\lambda\in K_{Q,U}$.

Fix an infinite component $U$.  For $g,h\in U$, any finite path in $U$ from $g$ to $h$ and
the triangle inequality give
\[
 \rho_o(\xi_g,\xi_h)\le\sum_{e\in E(U)}w_o(e).
\]
Passing to escaping sequences gives
$\diam(K_{Q,U})\le\sum_{e\in E(U)}w_o(e)$.  Since $0<q\le1$,
$(\sum a_i)^q\le\sum a_i^q$.  The edge sets of distinct components are disjoint, proving
\eqref{eq:finite-vertex-deletion-estimate}.

If the edge-weight sum is less than $\varepsilon^q$, every cluster set has diameter less than
$\varepsilon$.  An $\varepsilon$-separated subset of the limit set therefore meets distinct
cluster sets, proving \eqref{eq:cut-packing}.  If $c_\infty(Q)<m$ and
$0<\varepsilon<\varepsilon_m(\Lambda_\Gamma,\rho_o)$, then
$\pack_\varepsilon\ge m$, so the contrapositive gives
$\EdgeSum_{q,o}(Q;\xi)\ge\varepsilon^q$.  Let $\varepsilon\uparrow\varepsilon_m$ and take the
infimum over $Q$.
\end{proof}
The above theorem gives our component cover outside a word ball.
\begin{corollary}
\label{thm:word-ball-cover}
For $Q=B_S(n)$ we have,
\begin{equation}\label{eq:word-ball-cover-estimate}
 \sum_{U\in\mathscr U(B_S(n))}\diam(K_{B_S(n),U})^q
 \le\EdgeSum_{q,o}(B_S(n);\xi)
 \le |S|\e^{q(L+2D)}
 \sum_{|g|_S>n}\e^{-q d(o,go)}.
\end{equation}
If $P_{\Gamma,o}(q)<\infty$, then the covers have mesh and total $q$-diameter tending to zero.
Hence
\begin{equation}\label{eq:word-ball-Hq-zero}
 \mathcal H^q(\Lambda_\Gamma,\rho_o)=0.
\end{equation}
\end{corollary}

\begin{proof}
Orient each undirected edge of $\Cay(\Gamma,S)\setminus B_S(n)$ once and overcount it by all
pairs $(g,s)$ with $|g|_S>n$.  Apply \cref{lem:adjacent-contraction} and then
\cref{thm:finite-vertex-deletion-boundary}.  If the Poincar\'e series converges, its tail tends to zero.  The
right side controls both the sum of the $q$th powers of the diameters and their maximum, so
the Hausdorff $q$-contents tend to zero.
\end{proof}

\subsection{Extension to the end compactification}
Let $\widehat\Gamma=\Gamma\sqcup\mathcal E(\Gamma)$ denote the end compactification of the vertex set
of the locally finite Cayley graph.

\begin{theorem}[Extension of the orbit map to the end compactification]
\label{thm:end-boundary-map}
Assume $P_{\Gamma,o}(q)<\infty$ for some $q\in(0,1]$.  Let
\[
 \iota_o:\Gamma\longrightarrow X,\qquad \iota_o(g)=go,
\]
be the orbit map on the vertex set of the Cayley graph.  There is a unique continuous
$\Gamma$-equivariant extension
\[
 \overline\iota_o:\widehat\Gamma=\Gamma\sqcup\mathcal E(\Gamma)
 \longrightarrow X\sqcup\Lambda_\Gamma
\]
whose restriction to $\Gamma$ is $\iota_o$.  Its end restriction
$\Pi_o=\overline\iota_o|_{\mathcal E(\Gamma)}$ is onto $\Lambda_\Gamma$.  If a basic end neighborhood is represented by an infinite component $U$ of
$\Cay(\Gamma,S)\setminus Q$, where $Q\subset\Gamma$ is finite, then
\begin{equation}\label{eq:basic-end-neighborhood}
 \Pi_o(\mathcal E(U))\subseteq K_{Q,U},\qquad
 \diam\Pi_o(\mathcal E(U))^q\le
 \sum_{e\in E(U)}w_o(e)^q.
\end{equation}
\end{theorem}

\begin{proof}
We first define the image of an end.  Let
$r=(g_0,g_1,\ldots)$ be a ray representing $\epsilon\in\mathcal E(\Gamma)$.  For every
$n\ge0$, a tail of $r$ lies in a unique infinite component $U_n$ of
$\Cay(\Gamma,S)\setminus B_S(n)$, and these components are nested, i.e.
$U_{n+1}\subseteq U_n$.  Convergence of $P_{\Gamma,o}(q)$ implies that every orbit ball
$\{g:d(o,go)\le R\}$ is finite.  Since the vertices $g_j$ on a ray are distinct, it follows
that $d(o,g_jo)\to\infty$.

The visual compactification $\overline X=X\sqcup\partial X$ of a proper CAT$(-1)$ space is
compact, so $(g_jo)$ has a subsequence converging to some $\lambda\in\partial X$.  By
\cref{lem:boundary-choice}, the corresponding chosen boundary points converge to the same $\lambda$.
For each fixed $n$, the subsequence is eventually in $U_n$ and leaves every finite subset of
$U_n$, hence $\lambda\in K_{B_S(n),U_n}$.  The sets
$K_{B_S(n),U_n}$ are nonempty, compact, and nested.  By
\cref{thm:word-ball-cover}, their diameters tend to zero, so their intersection is a singleton.
Every convergent subsequence of $(g_jo)$ has its limit in that intersection. Compactness of
$\overline X$ then implies that the whole sequence converges to its unique point.  Define
$\Pi_o(\epsilon)$ to be this limit.

If two rays represent the same end, then outside every word ball their tails lie in the same
component.  The preceding nested cluster sets are therefore the same, so the limit is
independent of the chosen ray.  This proves that $\Pi_o$ is well defined.  We now define
$\overline\iota_o(g)=go$ for $g\in\Gamma$ and
$\overline\iota_o(\epsilon)=\Pi_o(\epsilon)$ for an end.

We state the inclusion for basic end neighborhoods before proving continuity.  Let $Q\subset\Gamma$ be finite,
let $U$ be an infinite component of $\Cay(\Gamma,S)\setminus Q$, and let
$\eta\in\mathcal E(U)$.  A representing ray has a tail in $U$. Along that tail the chosen boundary points converge to $\Pi_o(\eta)$ and leave every finite subset of $U$.  Hence
$\Pi_o(\eta)\in K_{Q,U}$.  The component estimate in
\cref{thm:finite-vertex-deletion-boundary} now gives
\[
 \diam\Pi_o(\mathcal E(U))^q
 \le \diam(K_{Q,U})^q
 \le \sum_{e\in E(U)}w_o(e)^q,
\]
which proves \eqref{eq:basic-end-neighborhood}.

Continuity at a vertex is immediate because vertices are isolated in the end compactification of the vertex set.  Let $z_j\to\epsilon$ in $\widehat\Gamma$ and set
$\lambda=\Pi_o(\epsilon)$.  Suppose that
$\overline\iota_o(z_j)$ does not converge to $\lambda$.  Passing to a subsequence in the
compact visual compactification, assume
$\overline\iota_o(z_j)\to\mu\ne\lambda$.  For each fixed $n$, all sufficiently large $z_j$
belong to $U_n\cup\mathcal E(U_n)$.  If infinitely many of these $z_j$ are ends, the inclusion for basic end neighborhoods and closedness of $K_{B_S(n),U_n}$ give
$\mu\in K_{B_S(n),U_n}$.  Otherwise a tail consists of vertices $g_j\in U_n$.  Convergence to
an end forces the vertices to leave every finite subset of the Cayley graph. The finiteness of
orbit balls then gives $d(o,g_jo)\to\infty$.  By \cref{lem:boundary-choice}, any boundary limit of
$g_jo$ is also the limit of $\xi_{g_j}$, and therefore again lies in
$K_{B_S(n),U_n}$.  Thus $\mu$ belongs to every member of the nested family
$K_{B_S(n),U_n}$, whose intersection is $\{\lambda\}$.  This contradiction proves continuity
at $\epsilon$.

For $\gamma\in\Gamma$, the translated ray $(\gamma g_j)$ represents $\gamma\epsilon$ and
$\gamma g_jo\to\gamma\Pi_o(\epsilon)$.  Hence
$\Pi_o(\gamma\epsilon)=\gamma\Pi_o(\epsilon)$, and the extension is equivariant.  Uniqueness
follows because every end is the limit of each of its representing rays and a continuous
extension must preserve that limit.

To prove surjectivity, let $\lambda\in\Lambda_\Gamma$ and choose distinct $h_j$ with
$h_jo\to\lambda$.  For each $n$, infinitely many $h_j$ lie in one component outside
$B_S(n)$.  Repeated subsequence selection produces nested infinite components
$U_{n+1}\subseteq U_n$ such that
$\lambda\in K_{B_S(n),U_n}$ for every $n$.  By the standard nested-component
characterization of ends of a locally finite graph
\cite[Section~8.5]{Diestel2010}---equivalently, by K\"onig's infinity lemma---there is an end
$\epsilon$ with $\epsilon\in\mathcal E(U_n)$ for all $n$.  The singleton-intersection argument
then gives $\Pi_o(\epsilon)=\lambda$.
\end{proof}

Next result is about ultrametric on ends from Poincar\'e-series tails.

\begin{theorem}
\label{thm:end-ultrametric}
Assume $P_{\Gamma,o}(q)<\infty$ for some $q\in(0,1]$.  The function
$d^{\mathcal E}_{q,o,S}$ defined in \eqref{eq:main-tail-ultrametric} is an ultrametric on
$\mathcal E(\Gamma)$ and induces the end topology.  The map
$\Pi_o=\overline\iota_o|_{\mathcal E(\Gamma)}$ from
\cref{thm:end-boundary-map} is a surjective, $\Gamma$-equivariant, $1$-Lipschitz map
\[
 (\mathcal E(\Gamma),d^{\mathcal E}_{q,o,S})
 \longrightarrow(\Lambda_\Gamma,\rho_o).
\]
Consequently it is a quotient map and
$\dim_H(\Lambda_\Gamma,\rho_o)\le
\dim_H(\mathcal E(\Gamma),d^{\mathcal E}_{q,o,S})$.
\end{theorem}

\begin{proof}
For distinct ends $\epsilon,\eta$, the integer $N_S(\epsilon,\eta)$ is finite because two
distinct ends are separated by the complement of some finite vertex set, and that finite set
is contained in a sufficiently large word ball.  Every value
$\Theta_{q,o,S}(n)$ is positive.  Because the non-elementary finitely generated group is
infinite and every word sphere is nonempty, the tail is strictly decreasing in $n$, and convergence
of the Poincar\'e series gives $\Theta_{q,o,S}(n)\to0$.

Let $\epsilon,\eta,\zeta$ be ends.  If $\epsilon$ and $\eta$ have tails in the same component
outside $B_S(n)$, and the same is true for $\eta$ and $\zeta$, then all three tails lie in that
one component.  Therefore
\[
 N_S(\epsilon,\zeta)\ge
 \min\{N_S(\epsilon,\eta),N_S(\eta,\zeta)\}.
\]
Since $\Theta_{q,o,S}$ is nonincreasing, taking its $q$th root gives
\[
 d^{\mathcal E}_{q,o,S}(\epsilon,\zeta)
 \le\max\{d^{\mathcal E}_{q,o,S}(\epsilon,\eta),
 d^{\mathcal E}_{q,o,S}(\eta,\zeta)\}.
\]
The function is symmetric and separates distinct ends, so it is an ultrametric.

We next compare its metric balls with the basic sets $\mathcal E(U_n)$.  Fix an end $\epsilon$ and let $U_n$ be its
component outside $B_S(n)$.  Every end in $\mathcal E(U_n)$ agrees with $\epsilon$ through
level $n$, and hence has distance at most $\Theta_{q,o,S}(n)^{1/q}$ from $\epsilon$.
Conversely, if an end does not lie in $\mathcal E(U_n)$, then it separates from $\epsilon$
at a level at most $n-1$ and has distance at least
$\Theta_{q,o,S}(n-1)^{1/q}>\Theta_{q,o,S}(n)^{1/q}$.  Therefore the closed metric ball
of radius $\Theta_{q,o,S}(n)^{1/q}$ about $\epsilon$ is exactly the basic end neighborhood $\mathcal E(U_n)$.  Thus the basic end neighborhoods
and metric balls determine the same neighborhood basis.  Since the sets $\mathcal E(U)$ form
a basis for the end topology, the ultrametric
induces exactly that topology.  In particular the end space remains compact.

Let $N=N_S(\epsilon,\eta)$ and let $U$ be the common component outside $B_S(N)$.  By
\cref{thm:end-boundary-map}, both images lie in $K_{B_S(N),U}$.  The word-ball estimate gives
\[
 \rho_o(\Pi_o\epsilon,\Pi_o\eta)^q
 \le\diam(K_{B_S(N),U})^q
 \le\Theta_{q,o,S}(N),
\]
which is the required $1$-Lipschitz inequality.  Surjectivity and equivariance were proved in
\cref{thm:end-boundary-map}.  A continuous surjection from a compact space to a Hausdorff
space is a quotient map.  Finally, a Lipschitz map does not increase Hausdorff dimension.
This follows directly by pushing forward the covers in the definition of Hausdorff measure,
or from \cite[Chapter~8]{Heinonen2001}.
\end{proof}

\begin{lemma}[Connected sets and one-dimensional Hausdorff measure]
\label{lem:connected-H1}
Let $C$ be a connected subset of a metric space.  With the Hausdorff-measure normalization
used above,
\[
 \mathcal H^1(C)\ge \diam(C).
\]
Consequently, a metric space of zero $1$-dimensional Hausdorff measure is totally
disconnected.
\end{lemma}

\begin{proof}
Fix $x,y\in C$.  The function $z\mapsto d(x,z)$ is $1$-Lipschitz.  Its image of $C$ is a
connected subset of $\mathbb R$ containing $0$ and $d(x,y)$, and hence contains the interval
$[0,d(x,y)]$.  Hausdorff measure does not increase under a $1$-Lipschitz map, while
$\mathcal H^1([0,d(x,y)])=d(x,y)$.  Thus
$\mathcal H^1(C)\ge d(x,y)$.  Taking the supremum over $x,y$ proves the displayed
inequality.  If the ambient space has zero $\mathcal H^1$-measure, every connected subset has
diameter zero and is therefore a singleton.
\end{proof}

\begin{corollary}[Full limit-set dimension and ends]
\label{cor:word-ball-dimension-ends}
Let $X$ be a proper CAT$(-1)$ space and let $\Gamma\le\Isom(X)$ be finitely generated,
non-elementary, and discrete, with $0<\delta_X(\Gamma)<\infty$.  Fix the boundary
notation in the introduction.  If $\delta_X(\Gamma)<1$, then
\eqref{eq:full-limit-dimension} holds.  More generally, every
finite vertex deletion satisfies \eqref{eq:cut-packing}, and word-ball cuts satisfy
\[
 |S|\e^{q(L+2D)}\sum_{|g|_S>n}\e^{-q d(o,go)}<\varepsilon^q
 \Longrightarrow
 c_\infty(B_S(n))\ge\pack_\varepsilon(\Lambda_\Gamma,\rho_o).
\]
\end{corollary}

\begin{proof}
For $\delta_X(\Gamma)<q<1$, \cref{thm:word-ball-cover} gives
$\dim_H\Lambda_\Gamma\le q$.  Let $q\downarrow\delta_X(\Gamma)$. The reverse inequality
follows from the conical limit set and the generalized Bishop--Jones theorem.  Taking one
$q<1$ also gives $\mathcal H^1(\Lambda_\Gamma,\rho_o)=0$, because the Bourdon metric has
diameter at most one and $r\le r^q$ for $0\le r\le1$.  By
\cref{lem:connected-H1}, the limit set is totally disconnected.  It is compact and perfect
because the action is non-elementary, and therefore it is homeomorphic to the Cantor set.

For completeness, fix $m\ge2$ and choose $m$ distinct points of the perfect compact limit
set.  Their minimum pairwise distance is some $\varepsilon>0$, so
$\pack_\varepsilon(\Lambda_\Gamma,\rho_o)\ge m$.  The Poincar\'e tail in
\eqref{eq:word-ball-cover-estimate} tends to zero; hence for all sufficiently large $n$,
\eqref{eq:cut-packing} gives $c_\infty(B_S(n))\ge m$.  Since the number of ends of a locally
finite graph is the supremum of the numbers of infinite components obtained by deleting
finite vertex sets, and $m$ was arbitrary, $e(\Gamma)=\infty$.
\end{proof}

\section{Finite-state normal forms at the group critical exponent}
\label{sec:finite-state}

Let $\mathcal A=(V,E)$ be a directed multigraph with a finite vertex set $V$ and at most
countably many parallel edges.  A vertex of $V$ is called a \emph{state}.  Each individual
edge $e\in E$ is a transition, and a \emph{group label} is the value
$\lambda(e)\in\Gamma$ assigned by a fixed map $\lambda:E\to\Gamma$.  The set of transition
symbols may be regarded as a countable alphabet, but ``label'' and ``alphabet'' are not
synonyms: the alphabet consists of the transition symbols, whereas $\lambda$ evaluates each
symbol in the group.

A \emph{directed walk} in the multigraph is a finite sequence of specified edges
$p=e_1\cdots e_n$ with compatible initial and terminal states; repetitions are allowed.  It is
closed when its initial and terminal states agree.  Because the graph is a multigraph, the individual edges---not merely the sequence of states---are part of the walk.  We set
\[
 \lambda(p)=\lambda(e_1)\cdots\lambda(e_n),
 \qquad
 L_x(p)=\sum_{j=1}^n\ell_x(\lambda(e_j)).
\]
Subadditivity gives $\ell_x(\lambda(p))\le L_x(p)$.

For $s\ge0$, give a transition the weight
$w_s(e)=\exp(-s\ell_x(\lambda(e)))$.  The \emph{weighted adjacency matrix} is the finite
matrix, possibly with infinite entries,
\begin{equation}\label{eq:finite-state-matrix}
 M_{uv}(s)=\sum_{e:u\to v}w_s(e)
 =\sum_{e:u\to v}\e^{-s\ell_x(\lambda(e))}\in[0,\infty].
\end{equation}
Thus $M_{uv}(s)$ is the total weight of all one-step transitions from state $u$ to state $v$.
When the relevant entries are finite, $(M(s)^n)_{uv}$ is the total product weight of all
$n$-step directed walks from $u$ to $v$.  This is the standard path-sum
interpretation of a weighted transition matrix, see
\cite[Sections~1.1--1.2]{DrosteKuichVogler2009} and
\cite[Section~2]{Mohri2009}.  For a finite subset $F\subset E$, let $M^F(s)$ be the matrix
obtained by retaining only the transitions in $F$.

A state $u$ \emph{reaches} $v$ if a directed walk goes from $u$ to $v$.  A subset of states
is strongly connected when every two states reach each other, and a \emph{strongly connected
component} is a maximal strongly connected subset.  The support digraph of a nonnegative
matrix has an arrow $u\to v$ precisely when its $(u,v)$ entry is positive.  A finite
nonnegative matrix is irreducible exactly when its support digraph is strongly connected
\cite[Theorem~3.2.1]{BrualdiRyser1991}.  Simultaneous permutation of rows and columns gives
Frobenius normal form, whose diagonal blocks correspond to the strongly connected components
\cite[Theorem~3.2.4]{BrualdiRyser1991}.  A directed closed walk of positive length will not be
called a ``loop'' below, because that word is also commonly reserved for a single edge from a
state to itself.

\begin{lemma}
\label{lem:closed-path-comparison}
Let $C$ be a strongly connected component whose support contains a directed closed walk of
positive length.  Suppose there are a state $v\in C$, an element $b\in\Gamma$, and an integer
$Q\ge1$ such that, for every $g\in\Gamma$, at most $Q$ directed closed walks $p$
in $C$ based at $v$ satisfy $b\lambda(p)=g$.  If $P_{\Gamma,x}(s)<\infty$, then
\begin{equation}\label{eq:closed-path-comparison}
 \sum_{n\ge0}(M_C(s)^n)_{vv}
 \le Q\e^{s\ell_x(b)}P_{\Gamma,x}(s).
\end{equation}
If entries of the full matrix have not yet been proved finite, then the left side is defined by
monotone exhaustion over finite transition sets.
\end{lemma}

\begin{proof}
For a finite transition set $F\subset E(C)$, the quantity
$\sum_{n\ge0}((M_C^F(s))^n)_{vv}$ is the sum of $\e^{-sL_x(p)}$ over all directed
closed walks based at $v$ that use only transitions in $F$, including the empty walk.  As
$F$ increases, these sums increase.  We define
\[
 \sum_{n\ge0}(M_C(s)^n)_{vv}
 :=\sup_{F\sqsubset E(C)}\sum_{n\ge0}((M_C^F(s))^n)_{vv}.
\]
Every finite walk uses only finitely many transitions, so this supremum is exactly the
nonnegative sum over all based closed walks.  This is the monotone-exhaustion
interpretation in the statement.

For each such walk,
\[
 \ell_x(b\lambda(p))\le\ell_x(b)+L_x(p),
 \qquad
 \e^{-sL_x(p)}\le
 \e^{s\ell_x(b)}\e^{-s\ell_x(b\lambda(p))}.
\]
All summands are nonnegative.  Tonelli's theorem, applied to
counting measure on the countable set of based closed walks and on
\(\Gamma\)
\cite[Theorem~2.37, p.~67]{FollandRealAnalysis},
allows the sum to be rearranged according to
\(g=b\lambda(p)\).  Since every fibre has cardinality at most \(Q\),
we obtain\[
 \sum_p\e^{-sL_x(p)}
 \le Q\e^{s\ell_x(b)}\sum_{g\in\Gamma}\e^{-s\ell_x(g)},
\]
which is \eqref{eq:closed-path-comparison}.
\end{proof}

\begin{lemma}
\label{lem:loop-series}
Let $M$ be a finite irreducible nonnegative matrix, and let $v$ be one of its states.  Then
\[
 \sum_{n\ge0}(M^n)_{vv}<\infty
 \qquad\mbox{if and only if}\qquad
 \rho(M)<1.
\]
\end{lemma}

\begin{proof}
Assume first that the diagonal return series at $v$ converges.  For arbitrary states $i,j$,
irreducibility supplies a directed walk from $v$ to $i$ and one from $j$ to $v$.  Let their
lengths be $a,b$ and the products of the corresponding positive matrix entries be
$\alpha,\beta>0$.  Concatenation gives, for every $n\ge0$,
$
 (M^{a+n+b})_{vv}\ge \alpha(M^n)_{ij}\beta.
$
Therefore every scalar series $\sum_n(M^n)_{ij}$ converges.  In particular $M^n\to0$
entrywise and hence in any matrix norm.  By
\cite[Theorem~5.6.12]{HornJohnson2013}, this is equivalent to $\rho(M)<1$.

Conversely, assume that \(\rho(M)<1\).  By
\cite[Theorem~5.6.12]{HornJohnson2013},
$
  M^N\longrightarrow 0
  \quad\mbox{for}\quad N\to\infty.
$
Since \(\rho(M)<1\), one has \(1\notin\sigma(M)\), and hence
\(I-M\) is invertible.  For
$
  S_N=\sum_{n=0}^{N}M^n,
$
the finite geometric-series identity gives
$
  (I-M)S_N=S_N(I-M)=I-M^{N+1}.
$
Letting \(N\to\infty\) yields
$
  \sum_{n=0}^{\infty}M^n=(I-M)^{-1}.
$
In particular, the \((v,v)\)-entry is summable. See also
\cite[Chapter~1]{Seneta2006} for the irreducible nonnegative-matrix background.
\end{proof}

\begin{lemma}[Lower bound from a directed closed walk]
\label{lem:closed-walk}
Let $M$ be a finite nonnegative matrix and let
$i_0,i_1,\ldots,i_q=i_0$ be a directed closed walk of positive length $q$.  Then
\begin{equation}\label{eq:closed-walk}
 \rho(M)^q\ge\prod_{j=0}^{q-1}M_{i_ji_{j+1}}.
\end{equation}
\end{lemma}

\begin{proof}
Repeating the prescribed closed walk $n$ times gives one summand in the $(i_0,i_0)$ entry of
$M^{nq}$, and hence
\[
 (M^{nq})_{i_0i_0}
 \ge\left(\prod_{j=0}^{q-1}M_{i_ji_{j+1}}\right)^n.
\]
Use the operator norm induced by the \(\ell^\infty\)-norm. Then $(M^{nq})_{i_0i_0}\le\|M^{nq}\|_\infty$.  Take $(nq)$th roots and let
$n\to\infty$.  Gelfand's spectral-radius formula
$\rho(M)=\lim_{m\to\infty}\|M^m\|^{1/m}$, cited as
\cite[Corollary~5.6.14]{HornJohnson2013}, yields \eqref{eq:closed-walk}.
\end{proof}
Next we have are convergence at the group critical exponent.
\begin{theorem}
\label{thm:critical-exponent-matrix}
Assume $0<\delta_X(\Gamma)<\infty$.  Let $C$ be a strongly connected component whose
support contains a directed closed walk of positive length.  Assume explicitly that there are
$v\in C$, $b\in\Gamma$, and $Q\ge1$ such that the map
\[
 p\longmapsto b\lambda(p)
\]
from directed closed walks in $C$ based at $v$ to $\Gamma$ has fibers of
cardinality at most $Q$.  Then
\begin{enumerate}[label=\textup{(\roman*)}]
\item every entry of $M_C(s)$ is finite and $\rho(M_C(s))<1$ for
$s>\delta_X(\Gamma)$,
\item every entry of $M_C(\delta_X(\Gamma))$ is finite and
$\rho(M_C(\delta_X(\Gamma)))\le1$,
\item every finite truncation $F$ satisfies
$\rho(M_C^F(\delta_X(\Gamma)))\le1$.
\end{enumerate}
\end{theorem}

\begin{proof}
Write $\delta=\delta_X(\Gamma)$.  If $s>\delta$, the right side of
\eqref{eq:closed-path-comparison} is finite.  Fix states $u,w\in C$.  Strong connectivity
provides a fixed directed walk $\alpha$ from $v$ to $u$ and a fixed directed walk $\beta$
from $w$ to $v$.  Let $a_s,b_s>0$ be their product weights.  For every one-step transition
$e:u\to w$, the concatenation $\alpha e\beta$ is a based closed walk at $v$.  Summing over
all parallel transitions gives
\[
 a_s M_{uw}(s)b_s
 \le\sum_{n\ge0}(M_C(s)^n)_{vv}<\infty.
\]
Thus every entry of the full block is finite for $s>\delta$.  The bounded-multiplicity comparison and the
diagonal return-series criterion now imply $\rho(M_C(s))<1$.

For a finite transition set $F$, entrywise comparison gives
$0\le M_C^F(s)\le M_C(s)$.  Monotonicity of the spectral radius for finite nonnegative
matrices, a standard consequence of Perron--Frobenius comparison
\cite[Chapter~1]{Seneta2006}, yields $\rho(M_C^F(s))<1$ for $s>\delta$.  The entries of the
finite matrix $M_C^F(s)$ are continuous on $[0,\infty)$, and the spectral radius is continuous
in finite dimension.  Letting $s\downarrow\delta$ proves part~\textup{(iii)}.

We next prove finiteness of an entry of the full block at $s=\delta$.  Fix $u,w\in C$ and choose a fixed
directed return walk $\gamma$ from $w$ to $u$; when $u=w$, the empty return walk is
allowed.  Retain its finitely many transitions, and let $c_\delta>0$ be their product weight
(with $c_\delta=1$ for the empty walk).  If $F_0$ is any finite set of parallel transitions
from $u$ to $w$, form the finite truncation containing $F_0$ and the return walk.  The state
sequence obtained by using the entry from $u$ to $w$ and then following $\gamma$ is a
directed closed walk.  Part~\textup{(iii)} and \cref{lem:closed-walk} give
\[
 c_\delta\sum_{e\in F_0}\e^{-\delta\ell_x(\lambda(e))}\le1.
\]
Taking the supremum over all finite $F_0$ proves $M_{uw}(\delta)<\infty$.

Finally choose increasing finite transition sets $F_n$ that exhaust all transitions in $C$.
The already established entrywise finiteness gives
$M_C^{F_n}(\delta)\to M_C(\delta)$ entrywise.  By continuity of the spectral radius in the
fixed finite matrix dimension and part~\textup{(iii)},
\[
 \rho(M_C(\delta))=\lim_n\rho(M_C^{F_n}(\delta))\le1.
\]
This proves part~\textup{(ii)} and completes the proof.
\end{proof}

\begin{corollary}
\label{cor:finite-state-subgroup}
Let $H\le\Gamma$, and assume every transition label $\lambda(e)$ lies in $H$.  Let $C$ be a
strongly connected component containing a directed closed walk of positive length.  Suppose
there are $v\in C$, $b\in H$, and $Q\ge1$ such that, for every $h\in H$, at most $Q$
directed closed walks $p$ in $C$ based at $v$ satisfy $b\lambda(p)=h$.  Then every
entry of $M_C(s)$ is finite for $s\ge\delta_X(\Gamma)$, and
\[
 \rho(M_C(s))<1\quad(s>\delta_X(\Gamma)),
 \qquad
 \rho(M_C(\delta_X(\Gamma)))\le1.
\]
Every finite truncation also has spectral radius at most one at
$s=\delta_X(\Gamma)$.
\end{corollary}

\begin{proof}
The word ``label'' here means the group element attached to a transition. It is not being
replaced by ``alphabet'', which denotes the set of transition symbols.  Regard the same
labelled walks as taking values in $\Gamma$.  If two of them have the same value in
$\Gamma$, they already have the same value in the subgroup $H$, so the fiber bound remains
$Q$.  Moreover the Poincar\'e series of $\Gamma$ converges for every
$s>\delta_X(\Gamma)$.  All hypotheses of \cref{thm:critical-exponent-matrix} are therefore
satisfied with the stated parameter, and every conclusion follows directly.
\end{proof}

\subsection{Finite transition sets}
The basic background that we use in this subsection is standard in finite-matrix theory.  Strongly connected
components, irreducibility, and Frobenius normal form are as in
\cite[Theorems~3.2.1 and~3.2.4]{BrualdiRyser1991}; Perron--Frobenius comparison is as in
\cite[Chapter~1]{Seneta2006}.  The classical theory of convergence parameters for countable
nonnegative matrices begins with \cite[Sections~2--3]{VereJones1967}.  Our matrices have a
finite state set but countably many parallel group-labelled transitions, and the statements
below concern approximation of the entries of their full blocks by finite transition sets.

Let $C$ be a strongly connected component whose support contains a directed closed walk of
positive length.  Define
\[
 \sigma_C=
 \inf\{s>0:M_C(s)\text{ is finite entrywise and }\rho(M_C(s))<1\},
\]
and, for a finite transition set $F$,
\[
 \sigma_{C,F}=\inf\{s>0:\rho(M_C^F(s))<1\}.
\]
For $s=0$, every retained transition has weight
$\exp(-0\cdot\ell_x(\lambda(e)))=1$.  Thus $M_C^F(0)$ is the nonnegative integer matrix
whose $(u,w)$ entry counts the retained parallel transitions from $u$ to $w$. This convention
extends the finite matrix continuously to $s=0$.

\begin{lemma}
\label{lem:positive-cycle-length}
Assume the bounded-multiplicity hypothesis of \cref{thm:critical-exponent-matrix}: there are
$v\in C$, $b\in\Gamma$, and $Q\ge1$ such that the map
$p\mapsto b\lambda(p)$ on directed closed walks based at $v$ has fibers of size at
most $Q$.  Then every nonempty directed closed walk $p$ in $C$, based at any state,
has
\[
 L_x(p)=\sum_{e\in p}\ell_x(\lambda(e))>0.
\]
\end{lemma}

\begin{proof}
Suppose a nonempty closed walk $p$ based at a state $u$ had $L_x(p)=0$.  Choose fixed
directed walks $\alpha:v\leadsto u$ and $\beta:u\leadsto v$.  For every $n\ge1$,
$\alpha p^n\beta$ is a directed closed walk based at $v$, and these walks are
pairwise distinct because they contain different numbers of occurrences of the nonempty
transition word $p$.  Their total transition lengths are all the same number
$L_x(\alpha)+L_x(\beta)$.  The fiber bound $Q$ implies that the group elements
$b\lambda(\alpha p^n\beta)$ assume infinitely many distinct values.  Subadditivity places all
of those values in the orbit-length ball of radius
$\ell_x(b)+L_x(\alpha)+L_x(\beta)$, contradicting the finiteness of orbit-length balls from
\cref{lem:orbit-basic}.
\end{proof}

\begin{proposition}
\label{prop:finite-transition-variation}
Assume $0<\delta_X(\Gamma)<\infty$, let $C$ be a strongly connected component containing a
directed closed walk of positive length, and assume the bounded-multiplicity hypothesis stated
in \cref{lem:positive-cycle-length}.  Then
\begin{equation}\label{eq:finite-transition-variation}
 \sigma_C=\sup_{F\sqsubset E(C)}\sigma_{C,F}.
\end{equation}
Whenever $M_C(s)$ is finite entrywise we have,
\begin{equation}\label{eq:spectral-monotone-exhaustion}
 \rho(M_C(s))=
 \sup_{F\sqsubset E(C)}\rho(M_C^F(s)).
\end{equation}
For each finite $F$, if $\rho(M_C^F(0))\le1$, then $\sigma_{C,F}=0$. Otherwise
$\sigma_{C,F}$ is the unique positive solution of $\rho(M_C^F(s))=1$.
\end{proposition}

\begin{proof}
Fix a finite transition set $F$ and set $M_F(s)=M_C^F(s)$.  Set this finite matrix in
Frobenius normal form.  On every irreducible diagonal block containing a directed closed
walk, \cref{lem:positive-cycle-length} implies that every closed walk has positive total transition length.  Hence for $0\le r<t$, at least one entry on each such block decreases strictly from
$M_F(r)$ to $M_F(t)$, while no entry increases and the support is unchanged.  Strict
Perron--Frobenius comparison therefore makes the spectral radius of each strongly connected diagonal block containing a directed closed walk strictly
decreasing.  The spectral radius of $M_F(s)$ is the maximum of the radius of its finitely many
strongly connected diagonal blocks containing directed closed walks.  Hence it too is strictly
decreasing whenever such a block is present.
It is continuous because its entries are continuous.

As $s\to\infty$, transitions of positive length have weights tending to zero.  The
support formed by the retained zero-length transitions contains no directed closed walk, by
\cref{lem:positive-cycle-length}, and is therefore acyclic, and its matrix is nilpotent.  Consequently
$\rho(M_F(s))\to0$.  The intermediate value theorem now gives the stated root description:
if $\rho(M_F(0))\le1$, the infimum is $0$, while otherwise there is exactly one positive
parameter at which the spectral radius equals one.

Suppose next that $M_C(s)$ is finite entrywise.  Choose increasing finite transition sets
$F_n$ exhausting $E(C)$.  Then $M_C^{F_n}(s)\to M_C(s)$ entrywise.  Because the state set
is finite, spectral-radius continuity yields
\[
 \rho(M_C(s))=\lim_n\rho(M_C^{F_n}(s)).
\]
Every finite truncation is dominated by the full matrix, so the limit equals the supremum over
all finite $F$, proving \eqref{eq:spectral-monotone-exhaustion}.  It also follows immediately
that $\sup_F\sigma_{C,F}\le\sigma_C$.

For the reverse inequality, take $0<r<\sigma_C$.  If some entry $M_{uw}(r)$ is infinite,
choose a fixed directed return walk from $w$ to $u$ and retain its transitions.  Its product
weight is a positive constant $c$.  Choose finitely many transitions from $u$ to $w$ whose
total weight is greater than $1/c$.  The resulting finite truncation contains a directed closed
walk whose entry product is greater than one, so \cref{lem:closed-walk} gives spectral radius
greater than one at $r$.

If all entries of $M_C(r)$ are finite, then the definition of $\sigma_C$ gives
$\rho(M_C(r))\ge1$.  Equality is impossible.  Indeed, for every $t>r$, the matrix
$M_C(t)$ is finite and is entrywise no larger than $M_C(r)$, with at least one strict decrease
on the irreducible component $C$.  Strict Perron--Frobenius comparison would give
$\rho(M_C(t))<1$, forcing $\sigma_C\le r$, a contradiction.  Thus
$\rho(M_C(r))>1$, and \eqref{eq:spectral-monotone-exhaustion} supplies a finite $F$ with
$\rho(M_C^F(r))>1$.

In either case the root description for that finite truncation gives
$\sigma_{C,F}>r$.  Letting $r\uparrow\sigma_C$ proves
\eqref{eq:finite-transition-variation}.
\end{proof}

\subsection{Poincar\'e series of trim unambiguous normal forms}
We use basic ideas in weighted-automaton theory.  A finite automaton has a finite set of states,
initial transitions, internal transitions, and terminal states. A path is accepted when it
begins with an initial transition and ends at a terminal state.  A state is accessible when it
occurs after some initial path and coaccessible when a terminal state can be reached from it.
The automaton is trim when every state has both properties
\cite[Definition~1.3.9]{BertheRigo2016}.  Weighted automata assign weights to initial,
transition, and terminal symbols and sum products over accepted paths, see
\cite[Section~2]{Mohri2009}.  We now add a group-evaluation map and state the exact
nonnegative-series form needed below.

Fix a group $\Gamma$ acting isometrically on $X$, with orbit length $\ell_x$.  Let $J$ be a
finite state set.  For each $j\in J$, let $\mathcal I_j$ be an at most countable set of
initial transition symbols ending at $j$, and for $i,j\in J$ let $\mathcal E_{ij}$ be an at
most countable set of internal transition symbols from $i$ to $j$.  Let
$q=(q_j)_{j\in J}\in\{0,1\}^J$ be a nonzero terminal-state indicator.  The disjoint union
\[
 \Sigma=\bigsqcup_j\mathcal I_j\sqcup
 \bigsqcup_{i,j}\mathcal E_{ij}
\]
is the transition alphabet.  The maps
\[
 \lambda_{\mathrm I}:\bigsqcup_j\mathcal I_j\to\Gamma,
 \qquad
 \lambda_{\mathrm E}:\bigsqcup_{i,j}\mathcal E_{ij}\to\Gamma
\]
are the group label maps.

An \emph{accepted transition sequence} is a tuple
\[
 \omega=(a,e_1,\ldots,e_n),
 \quad
 a\in\mathcal I_{j_0},\quad
 e_r\in\mathcal E_{j_{r-1}j_r},\quad q_{j_n}=1,
\]
where $n\ge0$.  Thus the initial symbol enters state $j_0$, the internal transitions follow
compatible states, and the final state is terminal.  We also use a separate empty sequence
$\varnothing$.  Let $\mathcal B$ be a nonempty at most countable set of coefficient symbols
with a group-label map $\lambda_{\mathcal B}:\mathcal B\to\Gamma$.  A normal-form
representation is a pair $(b,\omega)$, where $b\in\mathcal B$ and $\omega$ is empty or an
accepted transition sequence.  Its group evaluation is
\[
 \operatorname{ev}(b,\varnothing)=\lambda_{\mathcal B}(b),
\]
\[
 \operatorname{ev}(b,\omega)=
 \lambda_{\mathrm I}(a)\lambda_{\mathrm E}(e_1)\cdots
 \lambda_{\mathrm E}(e_n)\lambda_{\mathcal B}(b).
\]
Assume that evaluation is a bijection from the set $\mathcal L$ of normal-form
representations onto a countable subset $\mathscr S\subseteq\Gamma$.  This is the
unambiguous normal-form hypothesis.

Assign to a representation the nonnegative length
\[
 L_{\mathcal N}(b,\varnothing)=\ell_x(\lambda_{\mathcal B}(b)),
\]
\[
 L_{\mathcal N}(b,\omega)=
 \ell_x(\lambda_{\mathcal B}(b))+
 \ell_x(\lambda_{\mathrm I}(a))+
 \sum_{r=1}^n\ell_x(\lambda_{\mathrm E}(e_r)).
\]
Delete every state that is not accessible or not coaccessible, then the resulting automaton is
trim.  For $s>0$,  we define
\begin{equation}\label{eq:abstract-coefficients}
 V(s)=\sum_{b\in\mathcal B}\e^{-s\ell_x(\lambda_{\mathcal B}(b))},
 \qquad
 u_j(s)=\sum_{a\in\mathcal I_j}\e^{-s\ell_x(\lambda_{\mathrm I}(a))},
\end{equation}
\[
 K_{ij}(s)=\sum_{e\in\mathcal E_{ij}}
 \e^{-s\ell_x(\lambda_{\mathrm E}(e))}.
\]
All sums take values in $[0,\infty]$.  The abscissa of convergence of these coefficient series is
\[
 \delta_{\mathcal N,\coef}=
 \inf\{s>0:V(s)\text{ and every entry of }u(s)\text{ and }K(s)\text{ are finite}\}.
\]

\begin{theorem}
\label{thm:normal-form-resolvent}
Assume $0<\delta_X(\Gamma)<\infty$.  For every $s>0$, as an identity of extended
nonnegative real numbers,
\begin{equation}\label{eq:abstract-resolvent-expanded}
 \sum_{(b,\omega)\in\mathcal L}\e^{-sL_{\mathcal N}(b,\omega)}
 =V(s)\left(1+\sum_{n\ge0}u(s)^TK(s)^nq\right).
\end{equation}
The series on the left is finite if and only if $V(s)$ and every entry of $u(s)$ and $K(s)$
are finite and $\rho(K(s))<1$.  On that set of parameters we have,
\begin{equation}\label{eq:abstract-resolvent}
 \sum_{(b,\omega)\in\mathcal L}\e^{-sL_{\mathcal N}(b,\omega)}
 =V(s)\bigl(1+u(s)^T(I-K(s))^{-1}q\bigr).
\end{equation}
Its abscissa of convergence is the maximum of $\delta_{\mathcal N,\coef}$ and the
critical parameters $\sigma_C$ of the diagonal blocks in Frobenius normal form that contain
directed closed walks.  For every such block,
$\sigma_C=\sup_{F\sqsubset E(C)}\sigma_{C,F}$.
\end{theorem}

\begin{proof}
The representations with empty transition sequence contribute $V(s)$.  For $n\ge0$, the
sum over accepted transition sequences having exactly $n$ internal transitions is
$u(s)^TK(s)^nq$: matrix multiplication sums the products of the weights over all compatible
state sequences and all choices of parallel transition symbols.  The coefficient symbol \(b\in B\) is chosen independently.  All
path weights are nonnegative.  Tonelli's theorem for counting
measures
\cite[Theorem~2.37, p.~67]{FollandRealAnalysis}
therefore permits the countable summations over path length, state
sequences, initial symbols, internal transition symbols, and
coefficient symbols to be interchanged.  This proves \((37)\) as
an identity in \([0,\infty]\), without a prior convergence
assumption.

Assume now that all displayed coefficient entries are finite and
that \(\rho(K(s))<1\).  By
\cite[Theorem~5.6.12]{HornJohnson2013},
$
  K(s)^N\longrightarrow0.
$
Moreover \(I-K(s)\) is invertible, since
\(1\notin\sigma(K(s))\).  The finite geometric-series identity
$
  (I-K(s))\sum_{n=0}^{N}K(s)^n
  =
  I-K(s)^{N+1}
$
therefore gives
$
  \sum_{n=0}^{\infty}K(s)^n
  =
  (I-K(s))^{-1}.
$
Substitution in \((37)\) proves \((38)\).
Conversely, suppose the total normal-form series is finite.  The subseries with empty
transition sequence gives $V(s)<\infty$.  Choose one coefficient symbol
$b_0\in\mathcal B$.  If $\mathcal I_j$ is nonempty, coaccessibility of $j$ gives a fixed
internal suffix from $j$ to a terminal state.  Keep that suffix and $b_0$ fixed and vary only
the initial symbol in $\mathcal I_j$.  The resulting terms form a fixed positive multiple of
$u_j(s)$, so $u_j(s)<\infty$; if $\mathcal I_j$ is empty, this is automatic.  To control
$K_{ij}(s)$, accessibility of $i$ gives a fixed accepted prefix ending at $i$, and
coaccessibility of $j$ gives a fixed suffix from $j$ to a terminal state.  Keeping the prefix,
suffix, and $b_0$ fixed while varying the transition symbol in $\mathcal E_{ij}$ gives a fixed
positive multiple of $K_{ij}(s)$ as a subseries.  Hence every matrix entry is finite.

Set $K(s)$ in Frobenius normal form.  Suppose a diagonal block $C$ containing a directed closed walk had spectral
radius at least one.  Choose a state $v\in C$.  Trimness supplies a fixed accepted prefix
ending at $v$ and a fixed suffix beginning at $v$.  By
\cref{lem:loop-series}, the weighted return series of closed walks in $C$ based at $v$
diverges.  Inserting each such walk between the fixed prefix and suffix, and adjoining
$b_0$, produces a divergent subseries of the total normal-form series.  This contradiction
shows that every block containing a directed closed walk has spectral radius below one.  Acyclic diagonal blocks are
nilpotent, so $\rho(K(s))<1$.

It remains to identify the abscissa.  Fix a strongly connected diagonal block containing a directed closed walk and fixed accepted
entrance and exit paths through it.  If two based closed walks in the block had the same group
evaluation after the fixed entrance, appending the same exit and coefficient symbol would
produce two different normal-form representations with the same group evaluation.  The
bijection hypothesis forbids this.  Thus the bounded-multiplicity hypothesis holds with
$Q=1$, and \cref{prop:finite-transition-variation} applies.  Frobenius normal form says that
$\rho(K(s))<1$ exactly when every diagonal block containing a directed closed walk has spectral radius less than one.  Combining this
with coefficient finiteness gives the stated maximum formula and the supremum over finite truncations.
\end{proof}

\begin{theorem}
\label{thm:assigned-length-resolvent}
Keep the finite trim automaton, the coefficient symbols, and the unambiguous set of normal-form
representations from \cref{thm:normal-form-resolvent}.  Prescribe nonnegative functions
\[
 c_{\mathcal B}:\mathcal B\to[0,\infty),\qquad
 c_{\mathrm I}:\bigsqcup_j\mathcal I_j\to[0,\infty),\qquad
 c_{\mathrm E}:\bigsqcup_{i,j}\mathcal E_{ij}\to[0,\infty).
\]
For a representation set
\[
 L_c(b,\varnothing)=c_{\mathcal B}(b),
\]
\[
 L_c(b,\omega)=c_{\mathcal B}(b)+c_{\mathrm I}(a)
 +\sum_{r=1}^n c_{\mathrm E}(e_r),
 \qquad \omega=(a,e_1,\ldots,e_n),
\]
and define
\[
 V_c(s)=\sum_{b\in\mathcal B}\e^{-s c_{\mathcal B}(b)},\qquad
 (u_c(s))_j=\sum_{a\in\mathcal I_j}\e^{-s c_{\mathrm I}(a)},
\]
\[
 (K_c(s))_{ij}=\sum_{e\in\mathcal E_{ij}}\e^{-s c_{\mathrm E}(e)}.
\]
Define
\[
 \delta_{c,\coef}=\inf\{s>0:V_c(s),\ u_c(s),\text{ and }K_c(s)
 \text{ are entrywise finite}\}.
\]
Assume that every nonempty directed closed walk in every strongly connected diagonal block has
positive total $c_{\mathrm E}$-length.  For each strongly connected diagonal block $C$ whose
support contains a directed closed walk, define
\[
 \sigma_{c,C}=\inf\{s>0:K_{c,C}(s)\text{ is entrywise finite and }
 \rho(K_{c,C}(s))<1\},
\]
and, for a finite set $F$ of transition symbols in $C$,
\[
 \sigma_{c,C,F}=\inf\{s>0:\rho(K^F_{c,C}(s))<1\}.
\]
Then
\begin{equation}\label{eq:assigned-length-resolvent}
 \sum_{(b,\omega)\in\mathcal L}\e^{-sL_c(b,\omega)}
 =V_c(s)\left(1+\sum_{n\ge0}u_c(s)^T K_c(s)^nq\right)
\end{equation}
as an identity in $[0,\infty]$.  This series is finite exactly when all displayed coefficient
entries are finite and $\rho(K_c(s))<1$. On that domain the sum in
\eqref{eq:assigned-length-resolvent} equals
\[
 V_c(s)\bigl(1+u_c(s)^T(I-K_c(s))^{-1}q\bigr).
\]
Its abscissa of convergence is
\[
 \max\left\{\delta_{c,\coef},\ \max_C\sigma_{c,C}\right\},
 \qquad
 \sigma_{c,C}=\sup_F\sigma_{c,C,F},
\]
with the value $0$ assigned to a maximum or supremum over an empty collection.
\end{theorem}

\begin{proof}
The path-sum identity follows from Tonelli's theorem for counting
measures
\cite[Theorem~2.37, p.~67]{FollandRealAnalysis},
exactly as in the proof of Theorem~3.8.  If all coefficient
entries are finite and \(\rho(K_c(s))<1\), then
\cite[Theorem~5.6.12]{HornJohnson2013}
gives \(K_c(s)^N\to0\).  Applying the finite geometric-series
identity gives
$
  \sum_{n=0}^{\infty}K_c(s)^n
  =
  (I-K_c(s))^{-1}.
$
These arguments use only the nonnegative multiplicative weights
attached to the transition symbols.  For a finite transition set, the positivity assumption implies
that every directed closed walk has positive total length.  Consequently the spectral radius
of each irreducible diagonal block is continuous and strictly decreasing until it is below one,
and it tends to zero as $s\to\infty$.  The proof of
\cref{prop:finite-transition-variation} now applies with $c_{\mathrm E}$ in place of the orbit
length: if an entry of $K_{c,C}(s)$ is infinite, finitely many parallel transitions together
with a fixed return path give a finite truncation of spectral radius greater than one. If all
entries are finite, monotone exhaustion and strict Perron--Frobenius comparison give the same
conclusion below the critical parameter $\sigma_{c,C}$.  This proves the stated supremum
formula.

Conversely, suppose that the total prescribed-length normal-form series is finite. The proof of the converse implication in Theorem 3.8 depends only on nonnegativity of the symbol weights and trimness of the automaton, and therefore applies without change. The empty-transition subseries gives \(V_c(s)<\infty\). Fixing a coefficient symbol and suitable entrance and exit paths while varying only an initial or internal transition symbol proves that every entry of \(u_c(s)\) and \(K_c(s)\) is finite. If a strongly connected diagonal block containing a directed closed walk had spectral radius at least one, Lemma 3.2 would make its diagonal return series divergent; inserting those return walks between fixed entrance and exit paths would then produce a divergent subseries of the total normal-form series. Therefore every such block has spectral radius less than one, and hence \(\rho(K_c(s))<1\).
\end{proof}

\begin{theorem}
\label{thm:two-sided-comparison}
Let $\mathcal L$ be a countable set, let
$\operatorname{ev}:\mathcal L\to\mathscr S\subseteq\Gamma$ be a bijection, and let
$L_{\mathcal N}:\mathcal L\to[0,\infty)$ be a specified nonnegative length function.  Define
\[
 \delta_{\mathcal N}=
 \inf\left\{s>0:\sum_{p\in\mathcal L}\e^{-sL_{\mathcal N}(p)}<\infty\right\},
\]
\[
 \delta_x(\mathscr S)=
 \inf\left\{s>0:\sum_{g\in\mathscr S}\e^{-s\ell_x(g)}<\infty\right\}.
\]
Suppose there are $a\in(0,1]$ and $b\ge0$ such that, for every $p\in\mathcal L$,
\begin{equation}\label{eq:two-sided-comparison}
 aL_{\mathcal N}(p)-b\le
 \ell_x(\operatorname{ev}(p))\le L_{\mathcal N}(p).
\end{equation}
Then
\begin{equation}\label{eq:two-sided-exponents}
 \delta_{\mathcal N}\le\delta_x(\mathscr S)\le a^{-1}\delta_{\mathcal N}.
\end{equation}
In particular, bounded additive error with $a=1$ gives equality of the two abscissae.
\end{theorem}

\begin{proof}
The upper bound on the orbit length gives, term by term and using the bijection,
\[
 \sum_{p\in\mathcal L}\e^{-sL_{\mathcal N}(p)}
 \le
 \sum_{g\in\mathscr S}\e^{-s\ell_x(g)}.
\]
Thus convergence of the orbit-length series implies convergence of the normal-form series,
and $\delta_{\mathcal N}\le\delta_x(\mathscr S)$.  The lower bound in
\eqref{eq:two-sided-comparison} gives
\[
 \sum_{g\in\mathscr S}\e^{-s\ell_x(g)}
 \le\e^{sb}\sum_{p\in\mathcal L}\e^{-saL_{\mathcal N}(p)}.
\]
If $as>\delta_{\mathcal N}$, the right side is finite.  Hence
$\delta_x(\mathscr S)\le a^{-1}\delta_{\mathcal N}$.
\end{proof}

\section{Bass--Serre coordinates and weighted non-backtracking matrices}
\label{sec:bass-serre}

Let $\mathcal G$ be a finite connected graph of groups with fundamental group $H$, Bass--Serre
tree $T$, and finite quotient graph $Y=H\backslash T$.  We use the action without inversion.
For each vertex $v\in V(Y)$ choose a lift $\widetilde v\in T$.  For each oriented edge $f$,
choose a lift $\widetilde f$ with origin $\widetilde{o(f)}$.  There is $\tau_f\in H$ such that
\[
 t(\widetilde f)=\tau_f\widetilde{t(f)}.
\]
Choose reverse lifts compatibly, so
\begin{equation}\label{eq:reverse-edge-lifts}
 \widetilde{\bar f}=\tau_f^{-1}\overline{\widetilde f},
 \qquad
 \tau_{\bar f}=\tau_f^{-1}.
\end{equation}
Set
\[
 G_v=\operatorname{Stab}_H(\widetilde v),
 \qquad
 G_f=\operatorname{Stab}_H(\widetilde f)\le G_{o(f)}.
\]
Reversing orientation does not change an edge stabilizer, and
\eqref{eq:reverse-edge-lifts} therefore gives
\begin{equation}\label{eq:reverse-edge-stabilizer}
 G_{\bar f}=\tau_f^{-1}G_f\tau_f.
\end{equation}
The lifts of $f$ issuing from $\widetilde{o(f)}$ are $r\widetilde f$, indexed by left cosets
$rG_f\in G_{o(f)}/G_f$.

\begin{definition}[Relative orbit length]
\label{def:relative-length}
For $c=rG_f\in G_{o(f)}/G_f$, define
\begin{equation}\label{eq:relative-length}
 L_{f,x}(c)=\min_{h\in G_f}d(x,rh\tau_f x).
\end{equation}
The minimum exists by \cref{lem:orbit-basic} and is independent of the representative $r$.
\end{definition}

A transition from an oriented edge $e$ to an oriented edge $f$ is allowed when
$t(e)=o(f)$.  If $f=\bar e$, omit the identity coset $G_f$. It is precisely the immediate
reverse of the edge just traversed.

\begin{definition}[Weighted non-backtracking matrix]
\label{def:relative-matrix}
For $s>0$, define the matrix on the oriented-edge states $E^{\pm}(Y)$ by
\begin{equation}\label{eq:relative-matrix}
 K_{ef,x}(s)=
 \begin{cases}
 \displaystyle
 \sum_{c\in G_{o(f)}/G_f\setminus E_{ef}}
 \e^{-sL_{f,x}(c)},&t(e)=o(f),\\[1.2ex]
 0,&t(e)\ne o(f),
 \end{cases}
\end{equation}
where $E_{ef}=\{G_f\}$ if $f=\bar e$ and $E_{ef}=\varnothing$ otherwise.  A finite truncation
retains finitely many relative cosets in each entry.
\end{definition}

For every relative coset $c=rG_f$, choose $h_f(c)\in G_f$ realising the minimum in
\eqref{eq:relative-length}, and set
\begin{equation}\label{eq:minimising-label}
 b_f(c)=rh_f(c)\tau_f.
\end{equation}
Thus $d(x,b_f(c)x)=L_{f,x}(c)$.

Fix a base vertex $v_*$ of $Y$.  A reduced coset-labelled loop based at $v_*$ is a sequence
\[
 \omega=(f_1,c_1;\dots;f_n,c_n)
\]
with $o(f_1)=v_*$, $t(f_n)=v_*$, $t(f_j)=o(f_{j+1})$, and
$c_{j+1}\ne G_{f_{j+1}}$ whenever $f_{j+1}=\bar f_j$.  Set
\[
 b(\omega)=b_{f_1}(c_1)\cdots b_{f_n}(c_n),
 \qquad b(\varnothing)=1.
\]

\begin{lemma}
\label{lem:bass-serre-normal-form}
Every $g\in H$ has a unique representation
\begin{equation}\label{eq:bass-serre-normal-form}
 g=b(\omega)h,
 \qquad h\in G_{v_*},
\end{equation}
where $\omega$ is a reduced coset-labelled loop based at $v_*$.  Consequently
\begin{equation}\label{eq:bass-serre-length}
 L^{\BS}_x(g)=\sum_{j=1}^nL_{f_j,x}(c_j)+d(x,hx)
\end{equation}
defines a normal form length and
\begin{equation}\label{eq:bass-serre-subadditive-upper}
 d(x,gx)\le L^{\BS}_x(g).
\end{equation}
\end{lemma}

\begin{proof}
Write $x_*=\widetilde v_*$.  For a relative coset
$c_j=r_jG_{f_j}$ set
\[
 a_j=r_jh_{f_j}(c_j)\in G_{o(f_j)},
 \qquad
 p_0=1,
 \qquad
 p_j=b_{f_1}(c_1)\cdots b_{f_j}(c_j)
      =p_{j-1}a_j\tau_{f_j}.
\]
The representative $a_j$ belongs to the coset $c_j$, and multiplication by an element of
$G_{f_j}$ does not change the actual lift of $f_j$.  Hence, whenever the current vertex is
$p_{j-1}\widetilde{o(f_j)}$, the coset $c_j$ determines the actual oriented edge
\begin{equation}\label{eq:actual-bass-serre-edge}
 E_j=p_{j-1}a_j\widetilde f_j.
\end{equation}
Its terminal vertex is
\begin{equation}\label{eq:actual-bass-serre-terminal}
 t(E_j)=p_{j-1}a_j\tau_{f_j}\widetilde{t(f_j)}
       =p_j\widetilde{t(f_j)}.
\end{equation}
These two identities are the induction that relates the coordinate word to the path in $T$.

We next verify the non-backtracking rule exactly.  Suppose
$f_{j+1}=\bar f_j$.  By \eqref{eq:reverse-edge-lifts}, the next actual edge is
\[
 p_j a_{j+1}\widetilde{\bar f_j}
 =p_{j-1}a_j\tau_{f_j}a_{j+1}\tau_{f_j}^{-1}
   \overline{\widetilde f_j}.
\]
This equals the reverse
$\overline{E_j}=p_{j-1}a_j\overline{\widetilde f_j}$ if and only if
$\tau_{f_j}a_{j+1}\tau_{f_j}^{-1}\in G_{f_j}$, or equivalently, by
\eqref{eq:reverse-edge-stabilizer}, if and only if
$a_{j+1}\in G_{\bar f_j}$.  This is precisely the identity relative coset in
$G_{o(\bar f_j)}/G_{\bar f_j}$.  Thus omitting that coset is exactly the condition that the
actual tree path have no immediate reversal.

Now let $g\in H$.  The geodesic edge path from $x_*$ to $gx_*$ projects to an edge loop
$f_1\cdots f_n$ based at $v_*$.  Inductively, after the first $j-1$ edges have been encoded,
the $j$th actual outgoing edge is uniquely of the form
$p_{j-1}r\widetilde f_j$ for one left coset
$rG_{f_j}\in G_{o(f_j)}/G_{f_j}$.  Choose this coset as $c_j$ and use its fixed minimising
representative $a_j$. Equations \eqref{eq:actual-bass-serre-edge} and
\eqref{eq:actual-bass-serre-terminal} show that the coordinate word follows the same actual
edge and reaches the same terminal vertex.  The geodesic has no immediate reversal, so the
coset-labelled loop is reduced by the preceding calculation.  At the endpoint,
$p_nx_*=gx_*$, and consequently
$h=p_n^{-1}g$ belongs to $G_{v_*}$.  This proves existence of
\eqref{eq:bass-serre-normal-form}.

Conversely, equations \eqref{eq:actual-bass-serre-edge} and
\eqref{eq:actual-bass-serre-terminal} associate to every reduced coset-labelled loop a path
in $T$ from $x_*$ to $p_nx_*$.  The calculation above shows that no two consecutive edges
are reverses, so this path is reduced and hence geodesic.  Suppose
$p_nh=p'_mh'$ are two representations of the same element.  Since $h,h'$ fix $x_*$, the two
reduced paths have the same endpoints.  A tree has a unique geodesic between two vertices,
so the actual edge paths agree.  Their projections give $n=m$ and
$f_j=f'_j$ for every $j$.  Starting with $p_0=p'_0=1$, equality of the $j$th actual outgoing
edges gives equality of the corresponding left cosets, hence $c_j=c'_j$; because the
representative $b_{f_j}(c_j)$ was fixed once and for all, it then gives $p_j=p'_j$.
Induction recovers all cosets and finally $h=h'$.  This proves uniqueness.

Finally, subadditivity of orbit length and the minimising property
$d(x,b_{f_j}(c_j)x)=L_{f_j,x}(c_j)$ give
$d(x,gx)\le L_x^{\BS}(g)$.
\end{proof}

\subsection{Coefficient series and the resolvent}
Define
\[
 V_x(s)=\sum_{h\in G_{v_*}}\e^{-sd(x,hx)}.
\]
Let $E_*$ be the oriented-edge states occurring in some reduced coset-labelled loop based at
$v_*$.  Restrict $K_x(s)$ to $E_*$ and denote the restriction $K_{*,x}(s)$.  Define
\[
 (q_*)_f=\begin{cases}1,&t(f)=v_*,\\0,&t(f)\ne v_*,\end{cases}
\]
and
\[
 (u_x(s))_f=
 \begin{cases}
 \displaystyle\sum_{c\in G_{v_*}/G_f}\e^{-sL_{f,x}(c)},&o(f)=v_*,\\[1.2ex]
 0,&o(f)\ne v_*.
 \end{cases}
\]

\begin{definition}
\label{def:coef-abscissa}
Let $\delta_{\coef}(x)$ be the infimum of the $s>0$ for which $V_x(s)$, every entry of
$u_x(s)$, and every entry of $K_{*,x}(s)$ are finite.  Let
$\delta^{\BS}_x$ be the critical exponent of
$\sum_{g\in H}\e^{-sL^{\BS}_x(g)}$.
\end{definition}

\begin{theorem}[Exact Bass--Serre Poincar\'e series]
\label{thm:bass-serre-resolvent}
If $E_*=\varnothing$, then every reduced based normal form has empty edge word,
$H=G_{v_*}$, and
\begin{equation}\label{eq:bass-serre-edge-free}
 \sum_{g\in H}\e^{-sL^{\BS}_x(g)}=V_x(s)\qquad(s>0).
\end{equation}
In this case convergence is equivalent to $V_x(s)<\infty$ and
$\delta_x^{\BS}=\delta_{\coef}(x)$.

If $E_*\ne\varnothing$, then for every $s>0$,
\begin{equation}\label{eq:bass-serre-expanded}
 \sum_{g\in H}\e^{-sL^{\BS}_x(g)}
 =V_x(s)\left(1+\sum_{n\ge0}u_x(s)^TK_{*,x}(s)^nq_*\right)
\end{equation}
as an identity of extended nonnegative real numbers.  The series converges if and only if all
coefficient entries are finite and $\rho(K_{*,x}(s))<1$.  Under those conditions,
\begin{equation}\label{eq:bass-serre-resolvent}
 \sum_{g\in H}\e^{-sL^{\BS}_x(g)}
 =V_x(s)\left(1+u_x(s)^T(I-K_{*,x}(s))^{-1}q_*\right).
\end{equation}
In either case, if $\mathscr C$ denotes the strongly connected diagonal blocks whose supports
contain directed closed walks, then
\begin{equation}\label{eq:bass-serre-critical}
 \delta^{\BS}_x
 =\max\left\{\delta_{\coef}(x),
 \max_{C\in\mathscr C}\sup_F\sigma^{\rel}_{C,F}(x)\right\},
\end{equation}
where the maximum and supremum over the empty collection are $0$.
\end{theorem}

\begin{proof}
If $E_*=\varnothing$, \cref{lem:bass-serre-normal-form} shows that the edge word $\omega$
is empty for every element of $H$.  Hence $H=G_{v_*}$,
$L_x^{\BS}(g)=d(x,gx)$, and \eqref{eq:bass-serre-edge-free} and the asserted exponent formula
follow directly.

Assume $E_*\ne\varnothing$.  Then $q_*$ is nonzero, because every retained state occurs in a
reduced based loop and the last state of such a loop terminates at $v_*$.  Thus
\cref{thm:normal-form-resolvent} applies to the unique normal form in
\cref{lem:bass-serre-normal-form}.  The first edge contributes $u_x$, each subsequent allowed
transition contributes $K_{*,x}$, the last edge is selected by $q_*$, and the terminal
vertex-group element contributes $V_x$.  The critical-exponent formula follows from
\cref{prop:finite-transition-variation} on each block in $\mathscr C$.
\end{proof}

\subsection{Injectivity of closed paths in strongly connected blocks}

\begin{lemma}
\label{lem:bass-serre-closed-injective}
In every strongly connected component of the state graph containing a directed cycle, based
closed coset-labelled paths have pairwise distinct products in $H$ after multiplication by one
fixed prefix.
\end{lemma}

\begin{proof}
Fix a state $e_0$ in the component.  Since the state graph has been trimmed, choose once and
for all a reduced based coset-labelled path $\alpha$ from $v_*$ whose final oriented-edge
state is $e_0$.  Let $b\in H$ be the product of its coordinate labels.  If $p$ is a directed
closed path based at $e_0$, append its transitions to $\alpha$.  The identity coset is excluded
on every reverse transition, so the calculation in the proof of
\cref{lem:bass-serre-normal-form} shows that $\alpha p$ determines a reduced path in $T$.
Its initial vertex is $\widetilde v_*$ and its terminal vertex is
$
 b\lambda(p)\widetilde{t(e_0)},
$
where $\lambda(p)$ is the product of the transition labels of $p$.

If $b\lambda(p)=b\lambda(p')$ for two based closed paths, the two reduced paths
$\alpha p$ and $\alpha p'$ have the same initial and terminal vertices.  Uniqueness of the
geodesic in a tree makes their actual edge sequences identical.  Removing the fixed prefix
$\alpha$ and using the uniqueness of the outgoing relative coset at each vertex recovers every
individual transition. Hence $p=p'$.  Thus the map
$p\mapsto b\lambda(p)$ is injective, so the bounded-multiplicity hypothesis of
\cref{thm:critical-exponent-matrix} holds with $Q=1$.
\end{proof}

\begin{theorem}
\label{thm:bass-serre-critical-exponent}
Let $H\le\Gamma$ have a finite graph-of-groups decomposition.  For every strongly connected block $C$ whose support contains a directed closed walk, $K_{C,x}(s)$ is finite entrywise for $s\ge\delta_X(\Gamma)$ and
\begin{equation}\label{eq:bass-serre-critical-exponent}
 \rho(K_{C,x}(s))<1\quad(s>\delta_X(\Gamma)),
 \qquad
 \rho(K_{C,x}(\delta_X(\Gamma)))\le1.
\end{equation}
Every finite truncation has spectral radius at most one at $s=\delta_X(\Gamma)$.
\end{theorem}

\begin{proof}
With the minimising labels \eqref{eq:minimising-label}, the countable labelled transition
system has weighted adjacency matrix exactly $K_x(s)$.  The injectivity in
\cref{lem:bass-serre-closed-injective} verifies the hypothesis of
\cref{thm:critical-exponent-matrix}; use the subgroup form
\cref{cor:finite-state-subgroup}.
\end{proof}

\Cref{thm:finite-vertex-deletion-main,cor:boundary-dimension-ends} follow from
\cref{thm:finite-vertex-deletion-boundary,thm:word-ball-cover,thm:end-boundary-map,thm:end-ultrametric,cor:word-ball-dimension-ends}.
\Cref{thm:accessible-main,cor:accessibility-extension,cor:accessible-gap,cor:kleinian-classical}
are proved in \cref{sec:small-exponent}.  We now prove
\cref{thm:bass-serre-main,cor:resolvent-dimension-bounds}.

Fix $\beta$.  Apply \cref{thm:bass-serre-resolvent} to
$H_\beta$ with the restricted orbit length.  This gives
\eqref{eq:main-resolvent} and \eqref{eq:main-finite-recovery}.  By
\cref{thm:bass-serre-critical-exponent}, every full block $K_{\beta,C,o}$ is entrywise finite at $s=\delta_X(\Gamma)$
and satisfies the corresponding assertion in \eqref{eq:main-transfer-comparison}.  Finally,
$d(o,ho)\le L^{\BS}_{\beta,o}(h)$ gives
$\delta^{\BS}_{\beta,o}\le\delta_X(H_\beta)\le\delta_X(\Gamma)$.
\section{Bass--Serre transfer matrices and hyperbolic representatives}
\label{sec:main-proof}

For an allowed relative coset $c=rG_f$, every element of the form
\[
 g_{f,c}=rh\tau_f,\qquad h\in G_f,
\]
represents the same coordinate transition.  It need not realise the minimum in
\eqref{eq:relative-length}.

\begin{lemma}
\label{lem:relative-axis-comparison}
If $g_{f,c}=rh\tau_f$ is hyperbolic (equivalently axial), then
\begin{equation}\label{eq:axial-relative-domination}
 L_{f,o}(c)\le d(o,g_{f,c}o)\le a_o(g_{f,c}).
\end{equation}
Consequently, for every $s>0$ we have,
\begin{equation}\label{eq:axial-weight-domination}
 \e^{-sa_o(g_{f,c})}\le\e^{-sL_{f,o}(c)}.
\end{equation}
\end{lemma}

\begin{proof}
The first inequality follows because $L_{f,o}(c)$ is the minimum over $h\in G_f$.  The second
is \cref{lem:axis-power} with $n=1$.  Exponentiating reverses the inequalities.
\end{proof}

For each transition $e\to f$, let $\mathcal F_{ef}$ be a finite set of distinct allowed
relative cosets, and choose exactly one hyperbolic representative $g_{ef,c}$ for every
$c\in\mathcal F_{ef}$.  Summing \eqref{eq:axial-weight-domination} once for each selected
coset proves the entrywise inequality in \eqref{eq:main-transfer-comparison}.  Perron--Frobenius
monotonicity and the finite-transition estimate at $s=\delta_X(\Gamma)$ gives
\[
 \rho(A^{\mathcal F}_{\beta,C,o}(\delta_X(\Gamma)))
 \le\rho(K^{\mathcal F}_{\beta,C,o}(\delta_X(\Gamma)))\le1.
\]
Since $A^{\mathcal F}(s)\le K^{\mathcal F}(s)$ for all $s$, the set on which the exact relative
matrix has spectral radius less than one is contained in the corresponding set for the comparison matrix.  Hence $\sigma^{\ax}\le\sigma^{\rel}$.  Distinctness of the selected cosets is essential:
without it the same relative weight could be counted more than once.  This proves
\eqref{eq:main-transfer-comparison}.

If the action is convex-cobounded, \cref{thm:bourdon-coornaert} identifies $\delta_X(\Gamma)$ with the Hausdorff dimension of the limit set.  If instead
$\delta_X(\Gamma)<1$, the same identity is \eqref{eq:full-limit-dimension}, proved by the
word-ball case of the theorem.  The bound holds for every basepoint $y$, and
\eqref{eq:bourdon-basepoint} makes the boundary dimension basepoint-independent.  Taking
suprema gives \eqref{eq:main-dimension-lower-bound}.  A finite hierarchy is a finite family of
subgroup splittings, so the subgroup theorem applies at every node.  This completes
the proof of \cref{thm:bass-serre-main,cor:resolvent-dimension-bounds}.

\subsection{Spectral-radius equation for finite comparison matrices}

\begin{proposition}
\label{prop:axis-matrix-root}
Let $A^{\mathcal F}(s)$ be a finite comparison matrix.  If its support contains no directed cycle, then
$\sigma^{\ax}_{\mathcal F}=0$.  Otherwise $s\mapsto\rho(A^{\mathcal F}(s))$ is continuous and strictly decreasing on
$[0,\infty)$ and tends to zero as $s\to\infty$.  Thus
\[
 \sigma^{\ax}_{\mathcal F}=0\quad\text{if }\rho(A^{\mathcal F}(0))\le1,
\]
and otherwise $\sigma^{\ax}_{\mathcal F}$ is the unique positive solution of
\begin{equation}\label{eq:axis-matrix-root}
 \rho(A^{\mathcal F}(s))=1.
\end{equation}
\end{proposition}

\begin{proof}
Set $A^{\mathcal F}(s)$ in Frobenius normal form.  Every acyclic diagonal block is nilpotent
and contributes spectral radius zero.  Every remaining diagonal block is irreducible and its
support contains a directed cycle.  Each retained number $a_x(g)$ is strictly positive because
$a_x(g)\ge\tau_X(g)>0$. Hence, whenever $0\le r<t$, every positive entry of every strongly connected diagonal block containing a directed closed
walk strictly decreases from parameter $r$ to parameter $t$.  Strict Perron--Frobenius
comparison therefore gives
\[
 \rho(B(t))<\rho(B(r))
\]
for each strongly connected diagonal block $B$ containing a directed closed walk.

The spectral radius of $A^{\mathcal F}(s)$ is the maximum of the radius of its finitely many
strongly connected diagonal blocks containing directed closed walks.  Given $r<t$, choose a block $B$ attaining that maximum at $t$.
Then
\[
 \rho(A^{\mathcal F}(r))\ge\rho(B(r))>
 \rho(B(t))=\rho(A^{\mathcal F}(t)),
\]
so the full spectral radius is strictly decreasing.  It is continuous because the matrix is
finite and its entries are continuous.  Since the transition set is finite and every retained
length is positive, every entry tends to zero as $s\to\infty$, and hence so does the spectral
radius.  If the support contains no directed cycle, the whole matrix is nilpotent for every
$s$.  The formula for the critical parameter and the unique-root assertion now follow from
continuity and strict monotonicity.
\end{proof}

\section{Groups of critical exponent less than one: parabolics and structure}
\label{sec:small-exponent}

For $\zeta\in\partial X$, write $b_\zeta$ for a Busemann function.  If an isometry fixes
$\zeta$, the quantity $b_\zeta(gx)-b_\zeta(x)$ is independent of $x$ and defines its
Busemann character.  For every isometry $g$ fixing $\zeta$,
\[
 d(x,g^n x)\ge |b_\zeta(g^n x)-b_\zeta(x)|
 =n|\chi_\zeta(g)|.
\]
Thus a nonzero Busemann character gives positive stable translation length; the isometry
classification in a proper CAT$(-1)$ space then makes $g$ axial
\cite[Chapter~II.6]{BridsonHaefliger1999}.  Therefore a subgroup fixing $\zeta$ and
containing no axial isometry preserves every horosphere centred at $\zeta$.

\begin{lemma}
\label{lem:horosphere-contraction}
Let $P\le\Isom(X)$ fix $\zeta\in\partial X$ and contain no axial isometry.  Let
$c:[0,\infty)\to X$ be the ray from $o$ to $\zeta$.  For every finite set $T\subset P$ there
is $C_T>0$ such that
\begin{equation}\label{eq:horosphere-contraction}
 d(c(t),u c(t))\le C_T\e^{-t}\qquad(u\in T,\ t\ge0).
\end{equation}
\end{lemma}

\begin{proof}
Every $u\in P$ has zero Busemann character, so the rays $c$ and $uc$ have the same endpoint
$\zeta$ and are synchronized by the same Busemann parameter.  Apply the CAT$(-1)$ comparison
inequality for ideal triangles, in the precise form of
\cite[Proposition~4.4.13]{DasSimmonsUrbanski2017}, to the ideal triangle
$\Delta(o,uo,\zeta)$.  Normalize its comparison triangle in the upper half-plane model of
$\mathbb H^2$ so that the ideal vertex is $\infty$ and the comparison points for $o$ and $uo$
are $(0,1)$ and $(a,1)$.  Equality of the Busemann parameters is exactly the condition that
these two finite vertices lie on the same horosphere.  Since
\[
 d_{\mathbb H^2}((0,1),(a,1))=2\operatorname{arsinh}(|a|/2)=d(o,uo),
\]
the points at time $t$ on the two comparison rays are $(0,\e^t)$ and $(a,\e^t)$, and hence
\[
 \sinh\!\left(\frac{d_{\mathbb H^2}((0,\e^t),(a,\e^t))}{2}\right)
 =\frac{|a|}{2\e^t}
 =\e^{-t}\sinh\!\left(\frac{d(o,uo)}2\right).
\]
The ideal-triangle CAT$(-1)$ inequality now gives
\[
 \sinh\frac{d(c(t),uc(t))}{2}
 \le \e^{-t}\sinh\frac{d(o,uo)}2.
\]
Since $v\le\sinh v$ for $v\ge0$, one may take
$C_T=2\max_{u\in T}\sinh(d(o,uo)/2)$.
\end{proof}

\begin{proposition}
\label{prop:parabolic-growth}
Let $P\le\Isom(X)$ be a finitely generated subgroup fixing a point of $\partial X$ and
containing no axial isometry.  For every finite symmetric generating set $T$ there is $C_T\ge0$ such that
\begin{equation}\label{eq:parabolic-log-displacement}
 d(o,po)\le2\log(1+|p|_T)+C_T\qquad(p\in P).
\end{equation}
For every $s>\delta_X(P)$ we have,
\begin{equation}\label{eq:parabolic-polynomial-growth}
 \beta_T(n):=\#\{p:|p|_T\le n\}
 \le \e^{sC_T}P_{P,o}(s)(1+n)^{2s}.
\end{equation}
Consequently,
\begin{equation}\label{eq:parabolic-growth-degree}
 \limsup_{n\to\infty}\frac{\log\beta_T(n)}{\log n}
 \le2\delta_X(P),
\end{equation}
and:
\begin{enumerate}[label=\textup{(\roman*)}]
\item if $\delta_X(P)<1$, then $P$ is virtually cyclic,
\item if $\delta_X(P)<1/2$, then $P$ is finite.
\end{enumerate}
\end{proposition}

\begin{proof}
Write $p=u_1\cdots u_n$ with $u_j\in T$ and $n=|p|_T$.  By
\cref{lem:horosphere-contraction} and invariance,
\[
 d(c(t),pc(t))
 \le\sum_{j=1}^n d(u_1\cdots u_{j-1}c(t),u_1\cdots u_jc(t))
 \le C_Tn\e^{-t}.
\]
Thus $d(o,po)\le2t+C_Tn\e^{-t}$.  Taking $t=\log n$ for $n\ge1$ proves
\eqref{eq:parabolic-log-displacement}, after enlarging the constant to include $p=1$.
Every element of the word ball $B_T(n)$ then contributes at least
$\e^{-sC_T}(1+n)^{-2s}$ to $P_{P,o}(s)$, proving
\eqref{eq:parabolic-polynomial-growth}.  Taking logarithms, then the limsup, and finally
letting $s\downarrow\delta_X(P)$ proves \eqref{eq:parabolic-growth-degree}.

If $\delta_X(P)<1$, choose $s$ with $\delta_X(P)<s<1$.  The growth is polynomial of degree
strictly less than two.  Gromov's polynomial-growth theorem makes $P$ virtually nilpotent
\cite{Gromov1981}. The Bass--Guivarc'h degree formula then implies that an infinite such group
is virtually cyclic, since every non-virtually-cyclic finitely generated virtually nilpotent
group has growth degree at least two \cite{Bass1972,Guivarch1973}.  If
$\delta_X(P)<1/2$, choose $s<1/2$.  Then $\beta_T(n)=O(n^\alpha)$ with $\alpha<1$, whereas
every infinite locally finite Cayley graph contains a geodesic ray and therefore has at least
$n+1$ vertices in its radius-$n$ ball.  Hence $P$ is finite.
\end{proof}

\begin{corollary}
\label{prop:parabolic-half}
If $p$ is a parabolic isometry of infinite order, then
$\delta_X(\langle p\rangle)\ge1/2$.  The constant is sharp for $z\mapsto z+1$ in
$\mathbb H^2$.
\end{corollary}

\begin{proof}
The proof of \cref{prop:parabolic-growth} gives
$d(o,p^n o)\le2\log n+C$.  Thus the cyclic Poincar\'e series dominates a constant multiple of
$\sum n^{-2s}$ and diverges for $s\le1/2$.  In $\mathbb H^2$,
$d(i,i+n)=2\operatorname{arsinh}(n/2)=2\log n+O(1)$.
\end{proof}

\begin{corollary}[Parabolic growth degree and abelian rank]
\label{cor:parabolic-rank}
Let $P$ satisfy the hypotheses of \cref{prop:parabolic-growth}.
If $P$ is virtually nilpotent, then
\begin{equation}\label{eq:parabolic-bg-bound}
 d_{\mathrm{BG}}(P)\le2\delta_X(P),
\end{equation}
where $d_{\mathrm{BG}}(P)$ is its Bass--Guivarc'h degree of polynomial growth.  More generally,
if $P$ contains a virtually abelian subgroup of rank $r$, then
\begin{equation}\label{eq:parabolic-rank-bound}
 r\le2\delta_X(P).
\end{equation}
Both constants are sharp: a rank-$r$ lattice of horospherical translations in
$\mathbb H^{r+1}$ has critical exponent $r/2$.
\end{corollary}

\begin{proof}
The first inequality is \eqref{eq:parabolic-growth-degree} together with the
Bass--Guivarc'h formula.  Apply the same inequality to a finite-index free abelian subgroup
$A\cong\mathbb Z^r$ and use $\delta_X(A)\le\delta_X(P)$ to obtain the second.  In the upper
half-space model of $\mathbb H^{r+1}$, a translation vector $v\in\mathbb Z^r$ satisfies
$d(o,vo)=2\log\|v\|+O(1)$; hence its Poincar\'e series has the same abscissa as
$\sum_{v\ne0}\|v\|^{-2s}$, namely $r/2$.
\end{proof}

\subsection{Accessible groups below one}

\begin{lemma}
\label{lem:accessible-vertex-fg}
Let a finitely generated group be the fundamental group of a finite graph of groups with finite
edge groups.  Then every vertex group is finitely generated.
\end{lemma}

\begin{proof}
Every finite edge group is finitely generated.  The conclusion is therefore exactly
Haglund--Wise \cite[Theorem~1.3]{HaglundWise2021}: if a finitely generated group splits as a
graph of groups with finitely generated edge groups, then every vertex group is finitely
generated.  Their theorem does not require the underlying graph to be finite; the present
finite-graph case is an immediate specialization.
\end{proof}

\begin{proposition}
\label{prop:accessible-virtually-free}
Let $\Gamma\le\Isom(X)$ be finitely generated, discrete, and accessible over finite subgroups.
If $\delta_X(\Gamma)<1$, then $\Gamma$ is virtually free.
\end{proposition}

\begin{proof}
The elementary cases are immediate from the classification of discrete elementary CAT$(-1)$
actions and \cref{prop:parabolic-growth}: finite and axial elementary groups are virtually
cyclic, and every finitely generated parabolic elementary group with critical exponent less than one is
virtually cyclic.

Assume $\Gamma$ is non-elementary.  Choose a terminal finite graph-of-groups decomposition
with finite edge groups and vertex groups finite or one-ended.  By
\cref{lem:accessible-vertex-fg}, each vertex group $V$ is finitely generated, and subgroup
monotonicity gives $\delta_X(V)<1$.  If an infinite vertex group $V$ is non-elementary, then its critical exponent is positive.
Indeed, properness of $X$ and discreteness of $V$ imply that every $V$-orbit ball is finite,
so $V$ is strongly discrete in the terminology of
\cite[Definition~5.2.1]{DasSimmonsUrbanski2017}.  Such a group is not focal by
\cite[Proposition~6.4.1]{DasSimmonsUrbanski2017}; since it is non-elementary,
\cite[Proposition~7.3.1]{DasSimmonsUrbanski2017} therefore makes it a group of general type.
Then \cite[Proposition~7.4.7]{DasSimmonsUrbanski2017} supplies axial elements $a,b\in V$ with
disjoint fixed-point sets.  Their north--south dynamics
\cite[Theorem~6.1.10]{DasSimmonsUrbanski2017} and the ping-pong lemma give an integer
$N\ge1$ such that $A=a^N$ and $B=b^N$ generate a free semigroup.  If
\[
 R=\max\{d(o,Ao),d(o,Bo)\},
\]
then its $2^n$ positive words of length $n$ are distinct and move $o$ by at most $nR$.
Consequently
\[
 P_{V,o}(s)\ge\sum_{n\ge0}2^n\e^{-snR},
\]
which diverges for $s\le(\log2)/R$.  Thus $\delta_X(V)\ge(\log2)/R>0$.
The standing hypotheses of \cref{sec:finite-vertex-deletion} therefore apply to $V$, and
\cref{cor:word-ball-dimension-ends} gives $e(V)=\infty$, contradicting that $V$ is one-ended.  If $V$ is elementary, then it is finite, two-ended, or parabolic.  The first
case is not infinite, the second is not one-ended, and the parabolic case is virtually cyclic by
\cref{prop:parabolic-growth}, hence again not one-ended.  Therefore every terminal vertex
group is finite.

Thus $\Gamma$ is the fundamental group of a finite graph of finite groups.  Such groups are
virtually free by Bass--Serre theory \cite{KarrassPietrowskiSolitar1973,Serre2003}.
\end{proof}

\subsection{Geometric finiteness and the peripheral boundary quotient}

\begin{proposition}
\label{prop:accessible-peripheral-quotient}
Assume the hypotheses of \cref{prop:accessible-virtually-free} and that $\Gamma$ is
non-elementary.  Then the action on $\Lambda_\Gamma$ is geometrically finite, has finitely
many conjugacy classes of infinite maximal parabolic subgroups, every such subgroup is
virtually cyclic, and the quotient map in \eqref{eq:main-accessible-quotient} has precisely the
stated fibers.
\end{proposition}

\begin{proof}
By \cref{cor:word-ball-dimension-ends}, $\Lambda_\Gamma$ is a Cantor set.  The group
$\Gamma$ is virtually free by \cref{prop:accessible-virtually-free}, hence hyperbolic and
finitely presented.  Tukia's convergence theorem shows that a discrete group of isometries of a proper
Gromov-hyperbolic space acts as a convergence group on its limit set
\cite{Tukia1994}.  Moreover, properness and discreteness make $\Gamma$ strongly discrete;
\cite[Proposition~6.4.1]{DasSimmonsUrbanski2017} rules out focality, and
\cite[Corollary~7.4.3(ii)]{DasSimmonsUrbanski2017} identifies $\Lambda_\Gamma$ as the
unique minimal nonempty closed $\Gamma$-invariant subset.  Thus the action on
$\Lambda_\Gamma$ is a minimal convergence action.  Bowditch's Cantor-action theorem gives a finite graph-of-groups decomposition with finite edge
groups, geometric finiteness, and relative hyperbolicity with respect to the infinite vertex
stabilizers, which are exactly the maximal parabolic subgroups
\cite[Theorem~1.3 and pp.~33--34]{Bowditch2002}.

Those vertex groups are finitely generated by \cref{lem:accessible-vertex-fg}.  A maximal
parabolic stabilizer fixes one boundary point and contains no axial isometry, because axial
isometries are loxodromic for the convergence action.  Therefore
\cref{prop:parabolic-growth} makes every infinite maximal parabolic subgroup virtually
cyclic.

Let $\mathcal P=\{P_1,\dots,P_k\}$ represent their conjugacy classes.  In the virtually free
hyperbolic group $\Gamma$, the groups $P_i$ are quasiconvex and the peripheral collection is
almost malnormal.  Let $D(\partial\Gamma,\mathcal P)$ be the decomposition space obtained by
collapsing every translate $g\partial P_i$ to one point.  The quotient
\[
 \pi_{\mathcal P}:\partial\Gamma\longrightarrow
 D(\partial\Gamma,\mathcal P)
\]
is canonical.  Tran's boundary quotient theorem identifies this decomposition space
$\Gamma$-equivariantly with the Bowditch boundary of $(\Gamma,\mathcal P)$
\cite[Main Theorem]{Tran2013}.

The equivariant homeomorphism from that boundary to the concrete visual limit set
$\Lambda_\Gamma$ is unique.  Indeed, let $\Phi_1,\Phi_2$ be two such homeomorphisms and choose
a loxodromic element $a\in\Gamma$ for the Bowditch action.  Such an element is
nonperipheral and therefore acts axially on $X$.  North--south dynamics and equivariance imply
that both
maps send the attracting fixed point $a^+$ in the Bowditch boundary to the visual attracting
fixed point of $a$.  The same is true for every conjugate of $a$.  The visual action on
$\Lambda_\Gamma$ is minimal, so the orbit of the visual attracting point is dense; because
$\Phi_1$ is a homeomorphism, the orbit of $a^+$ is dense in the Bowditch boundary.  Hence
$\Phi_1$ and $\Phi_2$ agree on a dense set and therefore everywhere.  Denote this unique map
by $\Phi$.  The composition
\[
 q_\Gamma=\Phi\circ\pi_{\mathcal P}
\]
is consequently independent of all accessary choices and is the natural quotient in
\eqref{eq:main-accessible-quotient}.  Since each infinite $P_i$ is two-ended,
$\partial P_i$ has exactly two points, and every other fiber is a singleton.  The same Bowditch
decomposition has finite edge groups and finite or parabolic vertex groups, so its infinite
vertex groups are precisely conjugates of the $P_i$.
\end{proof}

\begin{proof}[Proof of \cref{thm:accessible-main}]
Virtual freeness is \cref{prop:accessible-virtually-free}; for non-elementary groups,
geometric finiteness, the finite peripheral family, and the boundary-fiber statement are
\cref{prop:accessible-peripheral-quotient}.  The elementary cases are explicit: finite groups
have empty boundaries; an axial two-ended group has its ordinary two-point boundary. And a
parabolic virtually cyclic group has one visual limit point, obtained by collapsing its
two-point Gromov boundary.

If $H\le\Gamma$ is finitely generated, then $H$ is virtually free and hence accessible.
Subgroup monotonicity gives $\delta_X(H)<1$, so the preceding argument applies to the induced
action of $H$.  This proves the hereditary assertion.  If $\Gamma$ is torsion-free, all finite
edge and finite vertex groups in the Bowditch decomposition are trivial; hence Bass--Serre
normal form gives $\Gamma\cong F_r*P_1*\cdots*P_k$, and each $P_i\cong\mathbb Z$.

For a non-elementary group, the quotient description shows that $q_\Gamma$ is a homeomorphism
exactly when there is no maximal parabolic subgroup.  Geometric finiteness identifies absence
of parabolic points with the assertion that every limit point is conical.  The compact-type
criterion then identifies this with convex-coboundedness
\cite[Theorem~12.2.7]{DasSimmonsUrbanski2017}.  By the isometry classification
\cite[Chapter~II.6]{BridsonHaefliger1999}, every isometry of the proper CAT$(-1)$ space is
elliptic, parabolic, or hyperbolic/axial.  An elliptic isometry fixes a point $x\in X$.  The
stabilizer $\Isom(X)_x$ is compact: properness of $X$ and a diagonal Arzel\`a--Ascoli argument
make every sequence in the stabilizer subconvergent uniformly on compact sets.  Its
intersection with the discrete group $\Gamma$ is therefore finite.  Hence no infinite-order
element of $\Gamma$ is elliptic, and every infinite-order element is axial or parabolic.
Moreover, every infinite maximal parabolic subgroup has already been proved virtually cyclic,
so every
parabolic point has a stabilizer containing an infinite-order parabolic element.  It follows
that absence of parabolic points is equivalent, under the standing hypothesis $\delta_X(\Gamma)<1$, to every
infinite-order element being hyperbolic/axial.  Under convex-coboundedness
the orbit map is a quasi-isometric embedding
\cite[Theorem~12.2.12]{DasSimmonsUrbanski2017}.

If $\delta_X(\Gamma)<1/2$, an infinite maximal parabolic subgroup would be finitely generated
and would contain no axial isometry, contradicting the finiteness conclusion of
\cref{prop:parabolic-growth}.  Thus there are no parabolic points.  The same argument applies
to every finitely generated subgroup $H$, because $\delta_X(H)\le\delta_X(\Gamma)$.  This
proves \eqref{eq:main-below-half-hereditary}.
\end{proof}

\subsection{Accessibility consequences and the sharp gap}

\begin{proof}[Proof of \cref{cor:accessibility-extension}]
A finitely presented group is accessible over finite subgroups by Dunwoody's theorem
\cite{Dunwoody1985}.  A finitely generated group whose finite subgroups have uniformly
bounded order is accessible by Linnell's theorem \cite{Linnell1983}.  Apply
\cref{thm:accessible-main}.  Torsion-free groups satisfy the second hypothesis with
bound one.
\end{proof}

\begin{proof}[Proof of \cref{cor:accessible-gap}]
The first assertion is the contrapositive of
\cref{prop:accessible-virtually-free}.  Dunwoody's and Linnell's theorems give the two stated
classes.  Finally, if $\delta_X(\Gamma)<1$, every finitely generated maximal parabolic subgroup
is virtually cyclic by \cref{prop:parabolic-growth}. In the geometrically finite conclusion of
\cref{thm:accessible-main}, maximal parabolics are finitely generated.  This proves
the parabolic gap.
\end{proof}

We now remark on strictness and sharpness of  \cref{thm:accessible-main}.
\begin{remark}
\label{rem:strictness-sharpness}
The extension beyond the torsion-free theorem is nontrival.  For example,
$C_2*C_3$ acts properly and cocompactly on its Bass--Serre tree.  A metric tree is
CAT$(-1)$, and multiplying all edge lengths by $R$ divides the critical exponent by $R$.
For large $R$ this gives a group with torsion and critical exponent below one covered by
\cref{cor:accessibility-extension} but not by a torsion-free theorem.  The constant one in
\cref{cor:accessible-gap} is sharp: a cocompact Fuchsian surface group is accessible,
not virtually free, and has critical exponent exactly one.
\end{remark}

\begin{proposition}
\label{prop:small-free}
Let $\Gamma\le\Isom(X)$ be finitely generated, torsion-free, and discrete.  If
$\delta_X(\Gamma)<1$, then $\Gamma$ is a finite-rank free group.
\end{proposition}

\begin{proof}
If $\Gamma$ is elementary, torsion-freeness and \cref{prop:parabolic-growth} leave only the
trivial group or an infinite cyclic axial or parabolic group.  Assume $\Gamma$ is
non-elementary and take a Grushko decomposition
\[
 \Gamma=F_r*G_1*\cdots*G_m,
\]
where every $G_i$ is finitely generated, freely indecomposable, nontrivial, and not infinite
cyclic \cite[Chapter~I.4]{LyndonSchupp1977}.  If $G_i$ is non-elementary, the independent-axial-element
and ping-pong argument in the proof of \cref{prop:accessible-virtually-free} gives
$\delta_X(G_i)>0$.  Subgroup monotonicity therefore gives
$0<\delta_X(G_i)\le\delta_X(\Gamma)<1$.  Choose $\delta_X(\Gamma)<q<1$.  The Poincar\'e
series of $G_i$ converges at $q$, so \cref{cor:word-ball-dimension-ends} gives infinitely many ends.
Stallings' theorem splits
$G_i$ over a finite subgroup \cite{Stallings1968}. Torsion-freeness makes the edge group
trivial, contradicting free indecomposability.  If $G_i$ is elementary, the first sentence
makes it cyclic, again a contradiction.  Thus $m=0$ and $\Gamma=F_r$.
\end{proof}

\begin{proof}[Proof of \cref{cor:linear-groups}]
By Selberg's lemma, $\Gamma$ has a torsion-free normal subgroup $\Gamma_0$ of finite index
\cite{Selberg1960}.  Every finite subgroup of $\Gamma$ intersects $\Gamma_0$ trivially and
therefore injects into the finite quotient $\Gamma/\Gamma_0$.  The orders of finite subgroups
are thus uniformly bounded.  Apply Linnell's accessibility theorem and
\cref{thm:accessible-main,cor:accessible-gap}.
\end{proof}

\subsection{The Kleinian consequence}

\begin{proof}[Proof of \cref{cor:kleinian-classical}]
A finitely generated Kleinian group is linear over characteristic zero, so
\cref{cor:linear-groups} gives the structural conclusions.  If $\Gamma$ is finite, the trivial subgroup
has finite index, and there is nothing further to prove.  Assume henceforth that $\Gamma$ is infinite.
By Selberg's lemma choose an infinite torsion-free normal subgroup $\Gamma_0$ of finite index
\cite{Selberg1960}.

If there are no parabolic points, the geometrically finite action is convex cocompact, and so is the
action of $\Gamma_0$.  Moreover
$\delta_{\mathbb H^3}(\Gamma_0)=\delta_{\mathbb H^3}(\Gamma)<1$.  If $\Gamma_0$ is elementary, then it is
infinite cyclic and generated by a loxodromic element, hence is a rank-one classical Schottky group.
If $\Gamma_0$ is non-elementary, Hou's theorem makes it a classical Schottky group
\cite[Theorem~1.1]{Hou2023}.  Thus in either case $\Gamma_0$ is a positive-rank classical Schottky
subgroup of finite index.  If $\Gamma$ is nontrivial and torsion-free and the action has no parabolic
point, take $\Gamma_0=\Gamma$.  The no-parabolic assertion below exponent one half follows from
\cref{thm:accessible-main}.
\end{proof}

\section{Free products and inequalities from translation lengths and axes}
\label{sec:free-products}

Let
\[
 H=H_1*\cdots*H_m\le\Gamma,
 \qquad m\ge2,
\]
be an embedded free product with $H_i\ne\{1\}$ for every $i$.  Set
\begin{equation}\label{eq:factor-series}
 Z_i(s)=\sum_{h\in H_i\setminus\{1\}}\e^{-sd(x,hx)}.
\end{equation}
For finite nonempty subsets $F_i\subset H_i\setminus\{1\}$, let
$z_i^F(s)=\sum_{h\in F_i}\e^{-sd(x,hx)}$ and define
\[
 B^F_{ij}(s)=
 \begin{cases}
 z_j^F(s),&i\ne j,\\
 0,&i=j.
 \end{cases}
\]
A transition from factor $i$ to factor $j$ chooses the next reduced syllable in $F_j$.

\begin{lemma}
\label{lem:factor-perron}
Let $z_1,\dots,z_m>0$ and let $B_{ij}=z_j$ for $i\ne j$, $B_{ii}=0$.  Its Perron root
$\lambda$ is the unique positive solution of
\begin{equation}\label{eq:factor-perron-equation}
 \sum_{i=1}^m\frac{z_i}{\lambda+z_i}=1.
\end{equation}
Consequently
\begin{equation}\label{eq:factor-rho-one}
 \rho(B)\le1
 \qquad\mbox{if and only if}\qquad
 \sum_{i=1}^m\frac{z_i}{1+z_i}\le1.
\end{equation}
For $m=2$, $\rho(B)=\sqrt{z_1z_2}$.
\end{lemma}

\begin{proof}
The left side of \eqref{eq:factor-perron-equation} is continuous and strictly decreasing from
$m$ to $0$.  For its unique solution, set $v_i=(\lambda+z_i)^{-1}$.  Then
\[
 (Bv)_i=\sum_{j\ne i}\frac{z_j}{\lambda+z_j}
 =1-\frac{z_i}{\lambda+z_i}=\lambda v_i.
\]
The positive eigenvector identifies $\lambda$ as the Perron root.  Evaluating at
$\lambda=1$ proves \eqref{eq:factor-rho-one}; the two-factor formula is immediate.
\end{proof}

\begin{theorem}
\label{thm:factor-series}
For the nontrivial factors fixed above, every $Z_i(\delta_X(\Gamma))$ is finite and
\begin{equation}\label{eq:factor-series-inequality}
 \sum_{i=1}^m
 \frac{Z_i(\delta_X(\Gamma))}{1+Z_i(\delta_X(\Gamma))}
 \le1.
\end{equation}
The same inequality holds for every $s\ge\delta_X(\Gamma)$.
\end{theorem}

\begin{proof}
Distinct syllable-labelled reduced paths represent distinct elements of the free product.
Apply \cref{thm:critical-exponent-matrix} to every finite factor matrix and then
\cref{lem:factor-perron}.  Exhaust each factor by finite subsets.  Monotone convergence gives
\eqref{eq:factor-series-inequality}, with the convention $\infty/(1+\infty)=1$.  No one
$Z_i$ can be infinite because another nontrivial factor contributes a positive summand.
Monotonicity in $s$ gives the last assertion.
\end{proof}

\subsection{Exact reduced-word exponent}
Let $L_{\red,x}$ be the sum of orbit lengths of the syllables in the unique reduced word.
Then $d(x,hx)\le L_{\red,x}(h)$.  Its Poincar\'e series has the exact matrix formula
\begin{equation}\label{eq:free-product-resolvent}
 P_{H,L_{\red,x}}(s)
 =1+z(s)^T\sum_{n\ge0}B(s)^n\one,
\end{equation}
whenever the factor series are finite, where $z=(Z_1,\dots,Z_m)^T$.

\begin{theorem}
\label{thm:exact-reduced-exponent}
For the nontrivial factors fixed above, the critical exponent $\delta_{\red,x}$ of
$L_{\red,x}$ is
\begin{equation}\label{eq:exact-reduced-exponent}
 \delta_{\red,x}
 =\inf\left\{s>0:
 \sum_{i=1}^m\frac{Z_i(s)}{1+Z_i(s)}<1\right\},
\end{equation}
with $\infty/(1+\infty)=1$.  Moreover
\begin{equation}\label{eq:finite-factor-recovery}
 \delta_{\red,x}
 =\sup_{F_1,\dots,F_m}
 \inf\left\{s>0:
 \sum_{i=1}^m\frac{Z_{F_i}(s)}{1+Z_{F_i}(s)}<1\right\},
\end{equation}
where the supremum ranges over finite nonempty $F_i\subset H_i\setminus\{1\}$.
Finally $\delta_{\red,x}\le\delta_X(H)\le\delta_X(\Gamma)$, and equality
$\delta_{\red,x}=\delta_X(H)$ holds whenever
$L_{\red,x}(h)-B\le d(x,hx)$ for one constant $B$ and all $h\in H$.
\end{theorem}

\begin{proof}
The convergence criterion and resolvent formula in
Theorem~3.8, together with Lemma~7.1, give
\eqref{eq:exact-reduced-exponent}.  For finite recovery, let $\Phi(s)$ be the scalar sum in
\eqref{eq:exact-reduced-exponent}.  If $r$ is smaller than the number defined in \eqref{eq:exact-reduced-exponent}, then
$\Phi(r)>1$. Equality would force the sum below one immediately to the right because at
least one positive-length syllable exists.  Monotone exhaustion supplies finite factor subsets
whose corresponding critical parameter is greater than $r$.  Let $r$ increase to the number
defined in \eqref{eq:exact-reduced-exponent}.  The exponent comparisons are \cref{thm:two-sided-comparison}.
\end{proof}

\subsection{Proof of the free-family inequality from translation lengths and axes}
Let $H_i=\langle a_i\rangle$.  By \cref{lem:axis-power}, for $n\ge1$,
\[
 d(x,a_i^{\pm n}x)\le n\tau_i+2r_i(x).
\]
Therefore
\begin{equation}\label{eq:cyclic-axis-lower}
 Z_i(\delta)
 \ge2\sum_{n\ge1}\e^{-\delta(n\tau_i+2r_i(x))}
 =\frac{2\e^{-2\delta r_i(x)}}{\e^{\delta\tau_i}-1}.
\end{equation}
The function $z\mapsto z/(1+z)$ is increasing.  Substituting
\eqref{eq:cyclic-axis-lower} into \eqref{eq:factor-series-inequality} gives
\eqref{eq:free-family-axis-bound}.

If $\tau_i\le T$ and $r_i(x)\le R$, every summand in
\eqref{eq:free-family-axis-bound} is at least
\[
 \frac{2\e^{-2\delta R}}
 {\e^{\delta T}-1+2\e^{-2\delta R}}.
\]
Hence
\[
 2(k-1)\e^{-2\delta R}\le\e^{\delta T}-1.
\]
After multiplication by $\e^{2\delta R}$,
\[
 \e^{\delta(T+2R)}\ge\e^{2\delta R}+2(k-1)\ge2k-1,
\]
which is \eqref{eq:free-family-uniform-bound}.  Taking an approximating sequence of basepoints proves
the statement with the infimal common-axis radius.  This proves
\cref{cor:free-family-axis-bound}.

\begin{remark}[Relation to Hou's inequality \cite{Hou2001}]
Using only $d(x,a_i^n x)\le n d(x,a_ix)$ in the same proof gives
\[
 \sum_{i=1}^k\frac{1}{1+\e^{\delta d(x,a_ix)}}\le\frac12,
\]
the curvature-free form of Hou's inequality evaluated at the group critical exponent.  The estimate using $a_x(a_i)$ instead retains the
fixed entrance-and-exit term $2d(x,\Ax(a_i))$ only once for the entire power $a_i^n$.
Neither bound uniformly dominates the other for every configuration, but the form using $a_x(a_i)$ is
strictly more informative when translation lengths and axis positions are the available
geometric information.
\end{remark}

\section{Amalgamated products and HNN extensions}
\label{sec:amalgams}

Let
\[
 H=A*_CB\le\Gamma
\]
be a nondegenerate embedded amalgam.  Let $C_A\le A$ and $C_B\le B$ denote the images of
$C$.  Here nondegenerate means that both inclusions are proper: $C_A\ne A$ and $C_B\ne B$.
Define relative orbit lengths
\[
 d_A(aC_A)=\min_{c\in C_A}d(x,acx),
 \qquad
 d_B(bC_B)=\min_{c\in C_B}d(x,bcx),
\]
and relative series
\begin{equation}\label{eq:amalgam-relative-series}
 Z_{A/C}(s)=\sum_{aC_A\ne C_A}\e^{-sd_A(aC_A)},
 \qquad
 Z_{B/C}(s)=\sum_{bC_B\ne C_B}\e^{-sd_B(bC_B)}.
\end{equation}
Choose minimum-length coset representatives.  Every element has a unique alternating normal
form ending in an element of $C$. Let $L_{\mathrm{nf},x}$ be the sum of its relative syllable
lengths and the terminal $C$-length.

\begin{proposition}[Exact amalgam normal-form series]
\label{prop:amalgam-exact}
Set
\[
 P_C(s)=\sum_{c\in C}\e^{-sd(x,cx)}.
\]
The normal-form Poincar\'e series converges exactly when the three coefficient series
$P_C(s)$, $Z_{A/C}(s)$, and $Z_{B/C}(s)$ are finite and
\[
 Z_{A/C}(s)Z_{B/C}(s)<1.
\]
On this domain,
\begin{equation}\label{eq:amalgam-resolvent}
 \sum_{h\in H}\e^{-sL_{\mathrm{nf},x}(h)}
 =P_C(s)
 \frac{(1+Z_{A/C}(s))(1+Z_{B/C}(s))}
 {1-Z_{A/C}(s)Z_{B/C}(s)}.
\end{equation}
Consequently
\begin{equation}\label{eq:amalgam-critical}
 \delta_{\mathrm{nf},x}
 =\max\left\{\delta_X(C),
 \inf\{s>0:Z_{A/C}(s)Z_{B/C}(s)<1\}\right\}.
\end{equation}
The second term in \eqref{eq:amalgam-critical} is the supremum of the corresponding
critical parameters obtained from finite nonempty subsets of the two relative-coset spaces.
\end{proposition}

\begin{proof}
Unique normal forms give two alternating geometric series.  Their sum is
\eqref{eq:amalgam-resolvent}.  Convergence is equivalent to convergence of the coefficient
series and of the alternating geometric series.  Finite subset determination follows by the same
monotone exhaustion argument as in \cref{thm:exact-reduced-exponent}.
\end{proof}

\begin{theorem}
\label{thm:amalgam-critical-exponent}
At $\delta=\delta_X(\Gamma)$,
\begin{equation}\label{eq:amalgam-product}
 Z_{A/C}(\delta)<\infty,
 \qquad
 Z_{B/C}(\delta)<\infty,
 \qquad
 Z_{A/C}(\delta)Z_{B/C}(\delta)\le1.
\end{equation}
Suppose finite relative-coset subsets $F_A,F_B$ are represented by hyperbolic elements
$g_c$ in their coordinate classes.  Set
\[
 Q_A(s)=\sum_{c\in F_A}\e^{-sa_x(g_c)},
 \qquad
 Q_B(s)=\sum_{c\in F_B}\e^{-sa_x(g_c)}.
\]
Then
\begin{equation}\label{eq:axis-amalgam-product}
 Q_A(\delta)Q_B(\delta)\le1.
\end{equation}
For a convex-cobounded action, the unique root of
$Q_A(s)Q_B(s)=1$, when positive, is a lower bound for
$\dim_H(\Lambda_\Gamma,\rho_x)$.
\end{theorem}

\begin{proof}
The two oriented-edge states form a strongly connected block with exact relative matrix
\[
 \begin{pmatrix}
 0&Z_{B/C}(s)\\
 Z_{A/C}(s)&0
 \end{pmatrix}.
\]
Apply \cref{thm:bass-serre-critical-exponent}, and its spectral radius is
$\sqrt{Z_{A/C}Z_{B/C}}$.  The comparison-matrix inequality follows from
\cref{lem:relative-axis-comparison} and matrix monotonicity.  The dimension conclusion is
\cref{thm:bourdon-coornaert}.
\end{proof}

\begin{corollary}
\label{cor:bounded-axis-amalgam}
If $F_A$ contains $N_A$ hyperbolic representatives $g$ satisfying $a_x(g)\le R_A$ and $F_B$ contains
$N_B$ hyperbolic representatives $g$ satisfying $a_x(g)\le R_B$, with $N_AN_B>1$, then
\begin{equation}\label{eq:bounded-axis-amalgam}
 \delta_X(\Gamma)(R_A+R_B)\ge\log(N_AN_B).
\end{equation}
For a convex-cobounded action this is a Hausdorff-dimension lower bound.
\end{corollary}

\begin{proof}
Use $Q_A(\delta)\ge N_A\e^{-\delta R_A}$ and
$Q_B(\delta)\ge N_B\e^{-\delta R_B}$ in
\eqref{eq:axis-amalgam-product}.
\end{proof}

\subsection{HNN extensions}
For a one-edge HNN extension, the quotient graph has one vertex and one unoriented loop.  The
two oriented states give a $2\times2$ relative matrix whose diagonal entries encode repeated
positive or negative stable-letter directions and whose off-diagonal entries encode mixed
transitions.

\begin{corollary}
\label{cor:HNN}
At $s=\delta_X(\Gamma)$, every diagonal entry lying on a one-cycle is finite and at most one.
If both mixed transitions are nonzero, both mixed entries are finite and their product is at
most one.  The same statements hold for every finite choice of hyperbolic representatives, with each exact relative
entry replaced by the corresponding finite sum of weights $\e^{-s a_x(g)}$.
\end{corollary}

\begin{proof}
Apply the cycle-product inequality
\eqref{eq:closed-walk} to the one-cycles and the mixed two-cycle, using
\cref{thm:bass-serre-critical-exponent} and \cref{lem:relative-axis-comparison}.
\end{proof}

\section{Lower bounds for boundary dimension}
\label{sec:boundary-lower-bounds}

We recall the finite nonnegative-matrix facts used in this section.  For a nonnegative matrix
$M$, its support digraph has an arrow $i\to j$ exactly when $M_{ij}>0$.  The strongly
connected components of this digraph are the irreducible diagonal blocks obtained after a
simultaneous permutation of rows and columns puts $M$ in Frobenius normal form
\cite[Theorems~3.2.1 and~3.2.4]{BrualdiRyser1991}.  This form is block upper triangular, so
the spectrum is the union of the spectra of its diagonal blocks and $\rho(M)$ is the maximum
of their spectral radius.  A diagonal block with acyclic support is nilpotent. On the remaining
irreducible blocks we use Perron--Frobenius comparison
\cite[Chapter~1]{Seneta2006}.  These are the Frobenius diagonal blocks referred to below.

\subsection{Matrix and cycle bounds}

\begin{proof}[Proof of \cref{cor:finite-matrix-cycle}]
Set $\delta=\delta_X(\Gamma)$.  For every transition,
\[
 B_{ef}(\delta)=N_{ef}\e^{-\delta R_{ef}}
 \le\sum_{c\in\mathcal F_{ef}}\e^{-\delta a_x(g_{ef,c})}
 =A^{\mathcal F}_{ef}(\delta).
\]
Thus
$\rho(B(\delta))\le\rho(A^{\mathcal F}(\delta))\le1$ by
\cref{thm:bass-serre-main}.  If all $R_{ef}\le R$, then
$B(\delta)\ge\e^{-\delta R}N$, so
$1\ge\e^{-\delta R}\rho(N)$, proving
\eqref{eq:matrix-dimension-bound}.  On a directed cycle,
\cref{lem:closed-walk} gives
\[
 1\ge\prod_jN_j\e^{-\delta R_j},
\]
which is \eqref{eq:cycle-axis-bound}.  The boundary statement follows from \cref{cor:boundary-dimension-ends,thm:bass-serre-main}: use \eqref{eq:full-limit-dimension}
when $\delta_X(\Gamma)<1$, and \cref{thm:bourdon-coornaert} in the convex-cobounded case.
\end{proof}

\subsection{Weighted Hashimoto matrices}
For a finite graph $Y$, the Hashimoto matrix, also called the non-backtracking edge matrix,
is the adjacency matrix on oriented edges in which $f$ may follow $e$ precisely when
$t(e)=o(f)$ and $f\ne\bar e$.  Hashimoto introduced this matrix in the study of graph zeta
functions \cite[Section~1]{Hashimoto1989}.  The weighted matrix below replaces
each permitted unit entry by an exponential weight coming from a selected hyperbolic
relative-coset representative.

Let $Y$ be the finite quotient graph of a splitting.  Its ordinary Hashimoto matrix has
oriented-edge states and entries
\[
 (B_Y)_{ef}=1
 \qquad\mbox{if and only if}\qquad
 t(e)=o(f),\ f\ne\bar e.
\]
Suppose for each allowed transition $e\to f$ we choose one hyperbolic relative-coset representative
$g_{ef}$ and set $R_{ef}=a_x(g_{ef})$.  Define
\[
 H^{\ax}_{Y,x}(s)_{ef}=
 \begin{cases}
 \e^{-sR_{ef}},&t(e)=o(f),\ f\ne\bar e,\\
 0,&\text{otherwise}.
 \end{cases}
\]

\begin{corollary}
\label{cor:hashimoto-axis-bound}
At $s=\delta_X(\Gamma)$,
\begin{equation}\label{eq:hashimoto-axis-bound}
 \rho(H^{\ax}_{Y,x}(\delta_X(\Gamma)))\le1.
\end{equation}
If $R_{ef}\le R$ on every transition belonging to a strongly connected component of $B_Y$ that contains a directed cycle, then
\begin{equation}\label{eq:hashimoto-dimension}
 \delta_X(\Gamma)R\ge\log\rho(B_Y)
\end{equation}
whenever $\rho(B_Y)>1$.  If $\delta_X(\Gamma)<1$, or if the action is convex-cobounded, this bounds
$\dim_H(\Lambda_\Gamma,\rho_x)$ from below.
\end{corollary}

\begin{proof}
Set $\delta=\delta_X(\Gamma)$.  The matrix $H^{\ax}_{Y,x}$ arises from a finite choice of hyperbolic representatives.  If every selected representative $g$ satisfies $a_x(g)\le R$, then
$H^{\ax}_{Y,x}(\delta)\ge\e^{-\delta R}B_Y$ on each strongly connected block containing a directed cycle.  Apply spectral-radius monotonicity on those blocks and then take the maximum over the Frobenius diagonal blocks.
\end{proof}

The exact relative matrix and the comparison matrix $A^{\mathcal F}_x$ both depend on $x$, whereas $\delta_X(\Gamma)$ and the limit-set dimension do not.

\begin{proposition}
\label{prop:basepoint-stability}
Let $D=d(x,y)$.  For every relative coset $c$ and every hyperbolic representative $g$,
\begin{align}
 |L_{f,x}(c)-L_{f,y}(c)|&\le2D,\label{eq:relative-basepoint-stability}\\
 |a_x(g)-a_y(g)|&\le2D.\label{eq:axis-bound-basepoint-stability}
\end{align}
Consequently, for every fixed finite support and every $s>0$,
\begin{equation}\label{eq:matrix-basepoint-stability}
 \e^{-2sD}K_x^F(s)\le K_y^F(s)\le\e^{2sD}K_x^F(s),
 \qquad
 \e^{-2sD}A_x^{\mathcal F}(s)\le A_y^{\mathcal F}(s)
 \le\e^{2sD}A_x^{\mathcal F}(s).
\end{equation}
In particular, the critical parameter of every finite matrix $A^{\mathcal F}_x$ varies continuously with the basepoint.
\end{proposition}

\begin{proof}
For every group element $u$,
$|d(x,ux)-d(y,uy)|\le2D$.  Taking minima over the same relative coset gives
\eqref{eq:relative-basepoint-stability}.  Equation
\eqref{eq:axis-bound-basepoint-stability} is \eqref{eq:axial-basepoint}.  Exponentiation and
summation give \eqref{eq:matrix-basepoint-stability}.  Continuity of the finite matrix entries and the root characterisation in
\cref{prop:axis-matrix-root} give continuity whenever the critical parameter is positive.  If
it is zero, then $\rho(A^{\mathcal F}(0))\le1$. This matrix collects only the fixed support
and is independent of the basepoint, so the critical parameter remains zero.  Hence continuity
holds in all cases.
\end{proof}

\begin{remark}
For any chosen family $\mathscr S$ of finite matrices $A^{\mathcal F}_x$,
\[
 \sup_{x\in X}\sup_{(\alpha,C,\mathcal F)\in\mathscr S}
 \sigma^{\ax}_{\alpha,C,\mathcal F}(x)
\]
is a lower bound for the orbit critical exponent, and hence for the Bourdon Hausdorff
dimension whenever $\delta_X(\Gamma)<1$ or the action is convex-cobounded.  This supremum
need not be a group invariant unless the splittings and representatives are chosen
equivariantly or canonically.
\end{remark}

\subsection{Strict dimension gaps for divergence-type factors}

\begin{corollary}
\label{cor:factor-dimension-gap}
Let $H=H_1*\cdots*H_m\le\Gamma$, $m\ge2$, where every $H_i$ is nontrivial.  Assume that
$\Gamma$ acts properly
discontinuously and non-elementarily, and that its action is convex-cobounded.  If the Poincar\'e series of $H_i$
diverges at $\delta_X(H_i)$, then
\begin{equation}\label{eq:factor-exponent-gap}
 \delta_X(H_i)<\delta_X(\Gamma).
\end{equation}
If $H_i$ is non-elementary and its action is convex-cobounded, then
\begin{equation}\label{eq:factor-dimension-gap}
 \dim_H(\Lambda_{H_i},\rho_x)
 <\dim_H(\Lambda_\Gamma,\rho_x).
\end{equation}
\end{corollary}

\begin{proof}
Subgroup monotonicity gives $\delta_X(H_i)\le\delta_X(\Gamma)$.  Equality would make the
factor series $Z_i$ diverge at $s=\delta_X(\Gamma)$, contradicting
\cref{thm:factor-series}.  The dimension statement follows from
\cref{thm:bourdon-coornaert} for the two convex-cobounded actions.
\end{proof}

\section{Grushko decompositions, JSJ trees, and finite hierarchies}
\label{sec:hierarchies}

Let a finitely generated group have a Grushko decomposition
\[
 \Gamma=G_1*\cdots*G_p*F_r,
\]
and choose a free basis $x_1,\dots,x_r$ of $F_r$.  We set
\[
 W_i(s)=\sum_{g\in G_i\setminus\{1\}}\e^{-sd(x,gx)},
 \qquad
 U_j(s)=\sum_{n\in\Z\setminus\{0\}}\e^{-sd(x,x_j^n x)}.
\]

\begin{corollary}[Grushko factor inequality]
\label{cor:grushko}
If $p+r\ge2$, then at $\delta=\delta_X(\Gamma)$,
\begin{equation}\label{eq:grushko-factor}
 \sum_{i=1}^p\frac{W_i(\delta)}{1+W_i(\delta)}
 +\sum_{j=1}^r\frac{U_j(\delta)}{1+U_j(\delta)}
 \le1.
\end{equation}
If the free generators $x_j$ act axially, their cyclic terms admit the lower estimates
\eqref{eq:cyclic-axis-lower}, producing the corresponding Grushko inequality from translation lengths and axis distances.
\end{corollary}

\begin{proof}
Apply \cref{thm:factor-series} to the displayed free-product factors and then
\cref{lem:axis-power} to the cyclic factors.
\end{proof}

\subsection{Finite JSJ representatives}
Fix a family of allowed edge groups and, where appropriate, a peripheral structure.
Guirardel--Levitt define the JSJ deformation space and prove existence under broad finite
presentation hypotheses \cite{GuirardelLevitt2017}.  Bowditch constructs a canonical finite
splitting over two-ended subgroups for the one-ended hyperbolic groups in the scope of
\cite[Theorem~0.1]{Bowditch1998}.  These results supply splittings, they are not used in the
proof of the spectral inequalities.

\begin{theorem}
\label{thm:JSJ}
Let $T_{\mathrm{JSJ}}$ be any chosen JSJ tree for $\Gamma$ with finite quotient.  Form the
weighted non-backtracking matrix associated to the Bass--Serre normal form from the CAT$(-1)$ orbit length.  For every strongly connected block whose support contains a directed closed walk, every
entry of the corresponding matrix $K_C(\delta_X(\Gamma))$ is finite and its spectral radius
is at most one.  Every finite choice of hyperbolic representatives gives a comparison matrix
whose critical parameter is at most $\delta_X(\Gamma)$.  If $\delta_X(\Gamma)<1$, or if
the action is convex-cobounded, each such critical parameter is a lower bound for
$\dim_H(\Lambda_\Gamma,\rho_x)$.
\end{theorem}

\begin{proof}
A finite JSJ representative is a finite graph of groups with fundamental group $\Gamma$.
Apply \cref{thm:bass-serre-main}.
\end{proof}

\begin{remark}
The relative series and hyperbolic representatives depend on the marked graph of groups, chosen
lifts, and basepoint.  The inequalities hold for every representative but do not, without
additional choices, define a numerical invariant of the entire JSJ deformation space.  A
canonical tree of cylinders or compatibility tree supplies an automorphism-invariant source
of states and stabilisers; an equivariant rule for choosing finite hyperbolic representatives is
still needed for a canonical numerical lower bound.
\end{remark}

\subsection{Finite hierarchies}
A finite graph-of-groups hierarchy is a finite rooted tree of subgroups.  Every nonterminal
node $H$ carries a finite graph-of-groups decomposition, and selected vertex groups become
children.  Later splittings may be relative to incident edge groups.

\begin{theorem}
\label{thm:hierarchy}
Let $\Gamma\le\Isom(X)$ be a countable discrete group acting on a proper CAT$(-1)$ space,
assume $0<\delta_X(\Gamma)<\infty$, and let a finite graph-of-groups hierarchy be rooted at
$\Gamma$.  At every nonterminal node $H\le\Gamma$, every entry of the matrix associated to a strongly
connected component containing a directed cycle is finite at $s=\delta_X(\Gamma)$, and the
spectral radius is at most one there.  Every finite choice of hyperbolic representatives at
every node gives a comparison matrix whose critical parameter is at most
$\delta_X(\Gamma)$.  If, in addition, $\Gamma$ is finitely generated and non-elementary
with $\delta_X(\Gamma)<1$, or if the action is convex-cobounded, all these critical
parameters are bounded by the single quantity
$\dim_H(\Lambda_\Gamma,\rho_x)$.
\end{theorem}

\begin{proof}
At each node, reduced Bass--Serre paths are injective as elements of $H$ and therefore as
elements of $\Gamma$.  Apply the subgroup form of \cref{thm:bass-serre-critical-exponent} and the comparison in \cref{lem:relative-axis-comparison} at every node.  Under the additional boundary hypotheses, use \cref{cor:word-ball-dimension-ends} or \cref{thm:bourdon-coornaert}.  There are only finitely many nodes, although the transfer-matrix theorem allows an arbitrary family.
\end{proof}
Next we give the simultaneous Grushko--JSJ lower bounds result.
\begin{corollary}
\label{cor:hierarchy-dimension}
Suppose that $\Gamma$ is equipped with a Grushko decomposition
$\Gamma=G_1*\cdots*G_p*F_r$ with $p+r\ge2$, finite relative JSJ representatives for any
selected noncyclic factors, and finitely many further relative splittings.  Then the root
factor series satisfy \eqref{eq:grushko-factor}.  At every descendant node, every entry of the matrix associated to a strongly connected
component containing a directed cycle is finite at $s=\delta_X(\Gamma)$, and its spectral
radius is at most one there.  Every finite choice of hyperbolic representatives gives a
comparison matrix whose critical parameter is at most $\delta_X(\Gamma)$.  The corresponding cycle, multiplicity, amalgam, HNN, and weighted
Hashimoto inequalities all hold with the same value $\delta_X(\Gamma)$.  If $\Gamma$ is finitely generated and its action is
non-elementary and discrete with $\delta_X(\Gamma)<1$, or if it is
convex-cobounded, then every one of these critical parameters is bounded above by
\[
 \dim_H(\Lambda_\Gamma,\rho_x)=\delta_X(\Gamma).
\]
\end{corollary}

\begin{proof}[Proof of \cref{cor:hierarchy-dimension}]
Apply \cref{cor:grushko} at the root and \cref{thm:hierarchy} at every descendant node.
The matrix, cycle, amalgam, HNN, and Hashimoto statements are
\cref{cor:finite-matrix-cycle,thm:amalgam-critical-exponent,cor:HNN,cor:hashimoto-axis-bound}.
\end{proof}

\section{Normal forms using translation lengths and distances to axes}
\label{sec:bounded-cancellation}

The comparison matrices above give lower bounds for the group critical exponent.  Equality
requires a geometric estimate controlling cancellation along the chosen Bass--Serre normal
form.

Assume that, for every allowed relative coset in a finite graph-of-groups decomposition of
$H$, one has fixed a representative $g_{f,c}=rh\tau_f$ that is hyperbolic (equivalently,
axial in the terminology of \cref{def:hyperbolic-axis-bound}).  The usual Bass--Serre argument still
gives a unique normal form
\[
 h=g_{f_1,c_1}\cdots g_{f_n,c_n}h_*,
 \qquad h_*\in G_{v_*}.
\]
Using the explicitly defined number $a_x(g)$ from \eqref{eq:axis-displacement-bound}, define the normal-form length
\begin{equation}\label{eq:axis-normal-form-length}
 L^{\ax}_x(h)
 =\sum_{j=1}^n a_x(g_{f_j,c_j})+d(x,h_*x).
\end{equation}
The superscript $\ax$ is only a mnemonic for the use of hyperbolic (axial) representatives.
By subadditivity and
\cref{lem:axis-power},
\begin{equation}\label{eq:axis-normal-form-upper}
 d(x,hx)\le L^{\ax}_x(h).
\end{equation}
Let
\[
 \delta^{\ax}_x(H)=
 \inf\left\{s>0:\sum_{h\in H}\e^{-sL^{\ax}_x(h)}<\infty\right\}.
\]

For completeness, we specify all transfer-matrix terms used below.  On the trim state set
$E_*$ from \cref{def:coef-abscissa},  we set
\[
 (u^{\ax}_x(s))_f=
 \begin{cases}
 \displaystyle\sum_{c\in G_{v_*}/G_f}\e^{-s a_x(g_{f,c})},&o(f)=v_*,\\[1.2ex]
 0,&o(f)\ne v_*,
 \end{cases}
\]
and
\begin{equation}\label{eq:full-comparison-matrix}
 A_{ef,x}(s)=
 \begin{cases}
 \displaystyle\sum_{c\in G_{o(f)}/G_f\setminus E_{ef}}
 \e^{-s a_x(g_{f,c})},&t(e)=o(f),\\[1.2ex]
 0,&t(e)\ne o(f).
 \end{cases}
\end{equation}
Write $A_{*,x}(s)$ for the restriction to $E_*$.  The terminal coefficient series is
\[
 V_x(s)=\sum_{h_*\in G_{v_*}}\e^{-sd(x,h_*x)}.
\]
Define
\begin{equation}\label{eq:axis-coefficient-abscissa}
 \delta^{\ax}_{\coef}(x)=
 \inf\left\{s>0:
 V_x(s),\ u^{\ax}_x(s),\text{ and }A_{*,x}(s)
 \text{ are entrywise finite}\right\}.
\end{equation}
For a strongly connected diagonal block $C$ whose support contains a directed closed walk,
and for a finite set $F$ of transition symbols in that block, set
\begin{equation}\label{eq:axis-finite-critical-parameter}
 \sigma^{\ax}_{C,F}(x)=
 \inf\{s>0:\rho(A^F_{C,x}(s))<1\}.
\end{equation}
The number $\sigma^{\ax}_{C,F}(x)$ is the critical parameter of the finite nonnegative
matrix $A^F_{C,x}(s)$.  The superscript $\ax$ keeps information only that the transition lengths are the
numbers $a_x(g)$.

\begin{theorem}
\label{thm:axis-bound-cancellation}
Suppose there are $a\in(0,1]$ and $b\ge0$ such that
\begin{equation}\label{eq:axis-bound-cancellation}
 aL^{\ax}_x(h)-b\le d(x,hx)\le L^{\ax}_x(h)
 \qquad(h\in H).
\end{equation}
Then
\begin{equation}\label{eq:axis-exponent-comparison}
 \delta^{\ax}_x(H)\le\delta_X(H)\le a^{-1}\delta^{\ax}_x(H).
\end{equation}
If $a=1$, then
\begin{equation}\label{eq:axis-exact-exponent}
 \delta_X(H)=\delta^{\ax}_x(H)
 =\max\left\{\delta^{\ax}_{\coef}(x),
 \max_{C\in\mathscr C}\sup_F\sigma^{\ax}_{C,F}(x)\right\}.
\end{equation}
If, moreover, $\delta^{\ax}_{\coef}(x)<\delta_X(H)$, then at least one strongly connected
block $C$ containing a directed closed walk satisfies
\begin{equation}\label{eq:axis-critical-rho-one}
 \rho(A_{C,x}(\delta_X(H)))=1.
\end{equation}
If $H$ has finite index in $\Gamma$, the common exponent in
\eqref{eq:axis-exact-exponent} is $\delta_X(\Gamma)$.
\end{theorem}

\begin{proof}
Apply \cref{thm:two-sided-comparison} to the unique normal form equipped with the length
$L^{\ax}_x$.  This gives \eqref{eq:axis-exponent-comparison}.

Suppose first that $E_*=\varnothing$.  Then \cref{thm:bass-serre-resolvent} gives
$H=G_{v_*}$ and $L^{\ax}_x(h)=d(x,hx)$ for every $h\in H$.  Therefore
\[
 \delta_x^{\ax}(H)=\delta_X(H)=\delta^{\ax}_{\coef}(x).
\]
The maximum over strongly connected blocks in \eqref{eq:axis-exact-exponent} is empty and is
understood to be $0$, so \eqref{eq:axis-exact-exponent} follows.  The strict inequality
$\delta^{\ax}_{\coef}(x)<\delta_X(H)$ cannot occur, and the assertion concerning a block of
spectral radius one is therefore vacuous.  If $H$ has finite index in $\Gamma$, the final
assertion follows from \cref{lem:subgroup-finite-index}.  We may henceforth assume
$E_*\ne\varnothing$.

Every transition length $a_x(g)$ is positive because $a_x(g)\ge\tau_X(g)>0$.  Hence every
nonempty directed closed walk has positive total transition length.  Since the trim state set
is nonempty, its terminal-state indicator is nonzero, and
\cref{thm:assigned-length-resolvent} applies with
\[
 c_{\mathcal B}(h_*)=d(x,h_*x),\qquad
 c_{\mathrm I}(f,c)=a_x(g_{f,c}),\qquad
 c_{\mathrm E}(e,f,c)=a_x(g_{f,c}).
\]
When $a=1$, the two critical exponents coincide and that theorem gives the maximum formula
and the determination from finite transition sets in \eqref{eq:axis-exact-exponent}.

Set $\delta=\delta_X(H)$ and assume
$\delta^{\ax}_{\coef}(x)<\delta$.  Choose
$s_0$ with $\delta^{\ax}_{\coef}(x)<s_0<\delta$.  Every coefficient entry is a
nonnegative Dirichlet series $\sum_j\e^{-sL_j}$ finite at $s_0$, so the Weierstrass
$M$-test gives uniform convergence and continuity on $[s_0,\infty)$.  In particular, the
entries of every matrix $A_{C,x}(s)$ obtained by retaining all allowed transitions, and the
corresponding finite-dimensional spectral radii, are continuous there.

We first establish the upper bound at $s=\delta$.  Since $a=1$,
\eqref{eq:axis-exact-exponent} gives $\delta_x^{\ax}(H)=\delta$.  Therefore, for every
$s>\delta$, the $L_x^{\ax}$-Poincar\'e series converges.  The convergence criterion in
\cref{thm:assigned-length-resolvent} yields
\[
 \rho(A_{C,x}(s))<1
\]
for every strongly connected diagonal block $C$ containing a directed closed walk.  Letting
$s\downarrow\delta$ and using the continuity just proved gives
\begin{equation}\label{eq:axis-critical-upper}
 \rho(A_{C,x}(\delta))\le1
\end{equation}
for every such block.

There are only finitely many diagonal blocks.  If all blocks containing directed closed walks
satisfied $\rho(A_{C,x}(\delta))<1$, continuity would give $\varepsilon>0$, with
$\delta-\varepsilon\ge s_0$, such that every such block had spectral radius less than one at
$\delta-\varepsilon$.  The coefficient entries are finite there, so the convergence criterion
in \cref{thm:assigned-length-resolvent} would make
$\sum_{h\in H}\e^{-(\delta-\varepsilon)L^{\ax}_x(h)}$ finite, contradicting
$\delta_x^{\ax}(H)=\delta$.  Hence at least one strongly connected diagonal block containing a directed closed walk satisfies
$\rho(A_{C,x}(\delta))\ge1$.  Combining this with
\eqref{eq:axis-critical-upper} proves \eqref{eq:axis-critical-rho-one}.  The finite-index
statement is \cref{lem:subgroup-finite-index}.
\end{proof}

\begin{remark}
Suppose instead that
\[
 d(x,hx)\ge L^{\ax}_x(h)-\kappa n(h)-b,
\]
where $n(h)$ is the number of transitions.  Set
$\widetilde a_x(g)=a_x(g)-\kappa$.  If these corrected transition lengths are nonnegative,
the corrected normal-form length
$\widetilde L_x(h)=L^{\ax}_x(h)-\kappa n(h)$ satisfies
$d(x,hx)\ge\widetilde L_x(h)-b$.  Consequently
\[
 \delta^{\ax}_x(H)\le\delta_X(H)\le\delta_{\widetilde L}(H),
\]
where $\delta_{\widetilde L}(H)$ is the abscissa of
$\sum_{h\in H}\e^{-s\widetilde L_x(h)}$.  The prescribed-symbol-length resolvent computes
this latter series.  The formula for its critical parameters over finite transition sets applies provided every
nonempty directed closed walk has positive total corrected transition length; this additional condition
is necessary when some $\widetilde a_x(g)$ may vanish.
\end{remark}

\subsection{Exact tree models}

\begin{example}
\label{ex:weighted-free-tree}
Let $F_k$ act freely and cocompactly on the universal cover of a metric rose whose $i$th edge
has length $d_i>0$.  The tree is CAT$(-1)$.  At a lift of the rose vertex, every basis axis
passes through the basepoint, so $a_x(a_i)=\tau_X(a_i)=d_i$.  Reduced words are geodesic and
have exactly additive length.  Hence \cref{thm:axis-bound-cancellation} applies with
$a=1$, $b=0$, and the critical exponent is the unique solution of
\begin{equation}\label{eq:weighted-rose-equation}
 \sum_{i=1}^k\frac{2}{1+\e^{s d_i}}=1.
\end{equation}
For equal $d_i=d$, this gives $sd=\log(2k-1)$, and equality holds in
\eqref{eq:free-family-axis-bound}.
\end{example}

\begin{example}[Point-glued CAT$(-1)$ trees of spaces]
Let a locally finite graph of complete proper CAT$(-1)$ spaces be glued along points, and
assume explicitly that the corresponding complete universal-cover tree of spaces is proper.
Local finiteness of the incidence graph alone does not guarantee this.  Reshetnyak's gluing
theorem gives the CAT$(-1)$ property, while the additional hypothesis gives properness, see
\cite[Chapter~II.11]{BridsonHaefliger1999}.  If the Bass--Serre coordinate changes carry
gluing points to gluing points, if the basepoint is a gluing point, and if every reduced
normal-form concatenation is geodesic through the successive gluing points, then the
relative-displacement normal-form length equals the orbit distance.  This gives exact examples for the Bass--Serre matrix statement in \cref{thm:bass-serre-main}.  Equality with the normal-form length $L_x^{\ax}$ is not automatic in a tree of spaces; it requires the separate two-sided estimate \eqref{eq:axis-bound-cancellation} for the numbers $a_x(g)$ at the fixed basepoint.
\end{example}

\subsection{Sharp constants}

\begin{proposition}[Sharpness]
\label{prop:sharpness}
The free-family inequality, the amalgam-product inequality, and the spectral-radius threshold for the exact relative Bass--Serre matrix are sharp. In the $C_2\ast C_3$ example, the exact relative two-state matrix has spectral radius one at $s=\delta_X(\Gamma)$. This example does not assert sharpness of the axial comparison cycle-product inequality (17).
\end{proposition}

\begin{proof}
Equality in the free-family formula is given by Example 11.3. For $C_2\ast C_3$ acting on its Bass--Serre tree with basepoint the midpoint of an edge of length $d>0$, the relative series are $e^{-sd}$ and $2e^{-sd}$, whose product is one at $s=\delta_X(\Gamma)$. Hence the exact relative two-state matrix has spectral radius one at that parameter. The nonidentity factor elements used here are elliptic, so this example does not concern the axial comparison matrix and gives no sharpness assertion for (17).
\end{proof}

Equality means that either a matrix obtained from a chosen finite transition set or the matrix
retaining all allowed transitions has critical parameter $\delta_X(\Gamma)$.  It does not force equality in every triangle inequality used to compare assigned normal-form lengths
with orbit displacement.  A geometric rigidity theorem would require additional control of
cancellation, equality cases in CAT$(-1)$ comparison, and the geometry of vertex and edge
stabilisers.  The present theorem makes no such conclusion.

A finite automaton over a redundant generating set may contain exponentially many paths
representing one group element.  Its Perron root can therefore exceed one even when the
group Poincar\'e series converges.  The bounded-multiplicity or unambiguous-normal-form
hypothesis is essential.  Bass--Serre normal forms supply precisely that injectivity for free
products and finite graphs of groups.

Here we want to make a note on finiteness. The Bass--Serre matrix theorem assumes a finite quotient graph.  For an infinite quotient,
every finite strongly connected subsystem still satisfies the estimate at $s=\delta_X(\Gamma)$, but
no single finite-dimensional block captures the full action.  The theorem does not assert the
existence, uniqueness, or canonicity of a JSJ tree.  It applies after a finite representative is
supplied by independent JSJ theory.

The comparison-matrix part also requires actual hyperbolic representatives in the selected relative cosets.
Parabolic or elliptic coordinate classes may still contribute to the exact relative matrix, but
they are absent from a comparison matrix formed only from hyperbolic representatives.  This can weaken the resulting finite lower bound. It does not affect the Bass--Serre matrix theorem at $s=\delta_X(\Gamma)$.

\subsection{Stable variants}

\begin{proposition}
\label{prop:stable-variants}
The conclusions remain valid under the following modifications.
\begin{enumerate}[label=\textup{(\roman*)}]
\item The map from based closed paths to $\Gamma$ may have uniformly bounded multiplicity
rather than being injective.
\item If a normal form is injective in a finite extension and then projected to $\Gamma$, the
multiplicity is bounded by the kernel order.
\item Replacing a relative length by any larger nonnegative length preserves every spectral-radius upper
bound.
\item Passing to a finite-index subgroup does not change the restricted orbit exponent.
\item A hyperbolic representative may be replaced by any larger explicit upper bound for its
one-step displacement. The resulting matrix remains an entrywise lower comparison matrix.
\end{enumerate}
\end{proposition}

\begin{proof}
Part~\textup{(i)} is built into \cref{lem:closed-path-comparison}.  Part~\textup{(ii)} follows
because two projected labels differ by an element of the finite kernel.  For
part~\textup{(iii)}, larger lengths give smaller exponential weights.  Part~\textup{(iv)} is
\cref{lem:subgroup-finite-index}.  Part~\textup{(v)} is the same entrywise monotonicity used in
\cref{lem:relative-axis-comparison}.
\end{proof}
\begin{center}
    \textbf{Acknowledgement}
\end{center}

This work was made possible by the unwavering support of Peter Sarnak, and Ying Zhou.  \par

henry.yhou@gmail.com


\begin{thebibliography}{99}

\bibitem{Bass1972}
H.~Bass,
\emph{The degree of polynomial growth of finitely generated nilpotent groups},
Proc. London Math. Soc. (3) \textbf{25} (1972), 603--614.

\bibitem{Dunwoody1985}
M.~J. Dunwoody,
\emph{The accessibility of finitely presented groups},
Invent. Math. \textbf{81} (1985), 449--457.

\bibitem{LouderTouikan2017}
L.~Louder and N.~Touikan,
\emph{Strong accessibility for finitely presented groups},
Geom. Topol. \textbf{21} (2017), no.~3, 1805--1835,
doi:10.2140/gt.2017.21.1805.

\bibitem{Stallings1971}
J.~R. Stallings,
\emph{Group Theory and Three-Dimensional Manifolds},
Yale Mathematical Monographs, vol.~4, Yale University Press, New Haven--London, 1971.

\bibitem{HaglundWise2021}
F.~Haglund and D.~T. Wise,
\emph{A note on finiteness properties of graphs of groups},
Proc. Amer. Math. Soc. Ser. B \textbf{8} (2021), 121--128,
doi:10.1090/bproc/81.

\bibitem{Diestel2010}
R.~Diestel,
\emph{Graph Theory}, 4th ed.,
Graduate Texts in Mathematics, vol.~173, Springer, Heidelberg, 2010.

\bibitem{DiestelKuhn2003}
R.~Diestel and D.~K\"uhn,
\emph{Graph-theoretical versus topological ends of graphs},
J. Combin. Theory Ser. B \textbf{87} (2003), no.~1, 197--206.

\bibitem{Gromov1981}
M.~Gromov,
\emph{Groups of polynomial growth and expanding maps},
Publ. Math. Inst. Hautes \`Etudes Sci. \textbf{53} (1981), 53--73.

\bibitem{Guivarch1973}
Y.~Guivarc'h,
\emph{Croissance polynomiale et p\'eriodes des fonctions harmoniques},
Bull. Soc. Math. France \textbf{101} (1973), 333--379.

\bibitem{Hou2023}
Y.~Hou,
\emph{The classification of Kleinian groups of Hausdorff dimensions at most one and
Burnside's conjecture},
Q. J. Math. \textbf{74} (2023), no.~2, 607--632.

\bibitem{KarrassPietrowskiSolitar1973}
A.~Karrass, A.~Pietrowski, and D.~Solitar,
\emph{Finite and infinite cyclic extensions of free groups},
J. Austral. Math. Soc. \textbf{16} (1973), 458--466.

\bibitem{Tran2013}
H.~C. Tran,
\emph{Relations between various boundaries of relatively hyperbolic groups},
Internat. J. Algebra Comput. \textbf{23} (2013), no.~7, 1551--1572.

\bibitem{Tukia1994}
P.~Tukia,
\emph{Convergence groups and Gromov's metric hyperbolic spaces},
New Zealand J. Math. \textbf{23} (1994), no.~2, 157--187;
erratum, \textbf{25} (1996), 105--106.

\bibitem{Bowditch1995}
B.~H. Bowditch,
\emph{Geometrical finiteness with variable negative curvature},
Duke Math. J. \textbf{77} (1995), no.~1, 229--274.

\bibitem{Bowditch2002}
B.~H. Bowditch,
\emph{Groups acting on Cantor sets and the end structure of graphs},
Pacific J. Math. \textbf{207} (2002), no.~1, 31--60.

\bibitem{Cavallucci2025}
N.~Cavallucci,
\emph{Bishop--Jones' theorem and the ergodic limit set},
Ergodic Theory Dynam. Systems \textbf{45} (2025), no.~3, 704--718.

\bibitem{DasSimmonsUrbanski2017}
T.~Das, D.~Simmons, and M.~Urba\'nski,
\emph{Geometry and Dynamics in Gromov Hyperbolic Metric Spaces},
Mathematical Surveys and Monographs, vol.~218, American Mathematical Society, 2017.

\bibitem{LiuWang2023}
B.~Liu and S.~Wang,
\emph{Discrete subgroups of small critical exponent},
Geom. Topol. \textbf{27} (2023), no.~6, 2347--2381.


\bibitem{Linnell1983}
P.~A. Linnell,
\emph{On accessibility of groups},
J. Pure Appl. Algebra \textbf{30} (1983), no.~1, 39--46.

\bibitem{Selberg1960}
A.~Selberg,
\emph{On discontinuous groups in higher-dimensional symmetric spaces},
in \emph{Contributions to Function Theory}, Tata Institute of Fundamental Research,
Bombay, 1960, 147--164.

\bibitem{LyndonSchupp1977}
R.~C. Lyndon and P.~E. Schupp,
\emph{Combinatorial Group Theory},
Ergebnisse der Mathematik und ihrer Grenzgebiete, vol.~89, Springer, 1977.

\bibitem{Stallings1968}
J.~R. Stallings,
\emph{On torsion-free groups with infinitely many ends},
Ann. of Math. (2) \textbf{88} (1968), 312--334.



\bibitem{AllenEtAl2011}
D.~Allen, M.~Cream, K.~Finlay, J.~Meier, and R.~Rohatgi,
\emph{Complete growth series and products of groups},
New York J. Math. \textbf{17} (2011), 321--329.

\bibitem{BalacheffMerlin2023}
F.~Balacheff and L.~Merlin,
\emph{A curvature-free $\log(2k-1)$ theorem},
Proc. Amer. Math. Soc. \textbf{151} (2023), no.~6, 2429--2434.

\bibitem{Bourdon1995}
M.~Bourdon,
\emph{Structure conforme au bord et flot g\'eod\'esique d'un espace CAT$(-1)$},
Enseign. Math. (2) \textbf{41} (1995), no.~1--2, 63--102.

\bibitem{Bowditch1998}
B.~H. Bowditch,
\emph{Cut points and canonical splittings of hyperbolic groups},
Acta Math. \textbf{180} (1998), no.~2, 145--186.

\bibitem{BridsonHaefliger1999}
M.~R. Bridson and A.~Haefliger,
\emph{Metric Spaces of Non-Positive Curvature},
Grundlehren der mathematischen Wissenschaften, vol.~319,
Springer, Berlin, 1999.


\bibitem{GuirardelLevitt2017}
V.~Guirardel and G.~Levitt,
\emph{JSJ decompositions of groups},
Ast\'erisque \textbf{395} (2017), vii+165 pp.

\bibitem{HersonskyHubbard1997}
S.~Hersonsky and J.~Hubbard,
\emph{Groups of automorphisms of trees and their limit sets},
Ergodic Theory Dynam. Systems \textbf{17} (1997), no.~4, 869--884.

\bibitem{Hou2001}
Y.~Hou,
\emph{Critical exponent and displacement of negatively curved free groups},
J. Differential Geom. \textbf{57} (2001), no.~1, 173--193.

\bibitem{Seneta2006}
E.~Seneta,
\emph{Non-negative Matrices and Markov Chains},
revised printing of the 2nd ed., Springer Series in Statistics, Springer, New York, 2006.

\bibitem{Serre2003}
J.-P. Serre,
\emph{Trees},
Springer Monographs in Mathematics, Springer, Berlin, 2003.


\bibitem{Hashimoto1989}
K.-i.~Hashimoto,
\emph{Zeta functions of finite graphs and representations of $p$-adic groups},
in \emph{Automorphic Forms and Geometry of Arithmetic Varieties},
Advanced Studies in Pure Mathematics, vol.~15, Kinokuniya, Tokyo, 1989, 211--280.

\bibitem{Heinonen2001}
J.~Heinonen,
\emph{Lectures on Analysis on Metric Spaces},
Universitext, Springer, New York, 2001.

\bibitem{DrosteKuichVogler2009}
M.~Droste, W.~Kuich, and H.~Vogler (eds.),
\emph{Handbook of Weighted Automata},
Monographs in Theoretical Computer Science, Springer, Berlin, 2009.

\bibitem{Mohri2009}
M.~Mohri,
\emph{Weighted automata algorithms},
in M.~Droste, W.~Kuich, and H.~Vogler (eds.),
\emph{Handbook of Weighted Automata}, Springer, Berlin, 2009, 213--254.

\bibitem{BertheRigo2016}
V.~Berth\'e and M.~Rigo,
\emph{Preliminaries},
in V.~Berth\'e and M.~Rigo (eds.),
\emph{Combinatorics, Words and Symbolic Dynamics},
Encyclopedia of Mathematics and its Applications, vol.~159,
Cambridge University Press, Cambridge, 2016, 1--37.

\bibitem{BrualdiRyser1991}
R.~A. Brualdi and H.~J. Ryser,
\emph{Combinatorial Matrix Theory},
Encyclopedia of Mathematics and its Applications, vol.~39,
Cambridge University Press, Cambridge, 1991.

\bibitem{HornJohnson2013}
R.~A. Horn and C.~R. Johnson,
\emph{Matrix Analysis}, 2nd ed.,
Cambridge University Press, Cambridge, 2013.

\bibitem{VereJones1967}
D.~Vere-Jones,
\emph{Ergodic properties of nonnegative matrices. I},
Pacific J. Math. \textbf{22} (1967), no.~2, 361--386.
\bibitem{FollandRealAnalysis}
G.~B. Folland,
\emph{Real Analysis: Modern Techniques and Their Applications},
2nd ed., John Wiley \& Sons, New York, 1999.

\end{thebibliography}
\end{document}